\documentclass[11pt]{amsart}
\usepackage{a4wide}
\usepackage[latin1]{inputenc}
\usepackage{amsmath}
\usepackage{amsfonts}
\usepackage{amssymb}
\usepackage{graphicx}
\usepackage{amsthm}
\usepackage{amssymb}
\usepackage{cite}
\usepackage{mathrsfs}
\usepackage{hyperref}
\usepackage{paralist}
\usepackage{enumitem}
\usepackage{tikz}
\usetikzlibrary{matrix,chains}
\makeindex
\begin{document}

\newcommand{\N}{\mathbf{N}}
\newcommand{\n}{\mathbf{n}}
\newcommand{\x}{\mathbf{x}}
\newcommand{\h}{\mathbf{h}}
\newcommand{\m}{\mathbf{m}}

\newcommand{\B}{\mathbf{B}}
\newcommand{\U}{\mathbf{U}}
\newcommand{\V}{\mathbf{V}}
\newcommand{\T}{\mathbf{T}}
\newcommand{\G}{\mathbf{G}}
\newcommand{\Para}{\mathbf{P}}
\newcommand{\Levi}{\mathbf{L}}
\newcommand{\Y}{\mathbf{Y}}
\newcommand{\X}{\mathbf{X}}
\newcommand{\M}{\mathbf{M}}
\newcommand{\pro}{\mathbf{prod}}
\renewcommand{\o}{\overline}

\newcommand{\Gtilde}{\mathbf{\tilde{G}}}
\newcommand{\Ttilde}{\mathbf{\tilde{T}}}
\newcommand{\Btilde}{\mathbf{\tilde{B}}}
\newcommand{\Ltilde}{\mathbf{\tilde{L}}}
\newcommand{\C}{\operatorname{C}}

\newcommand{\bl}{\operatorname{bl}}
\newcommand{\Z}{\operatorname{Z}}
\newcommand{\Gal}{\operatorname{Gal}}
\newcommand{\kernel}{\operatorname{ker}}
\newcommand{\Irr}{\operatorname{Irr}}
\newcommand{\D}{\operatorname{D}}
\newcommand{\I}{\operatorname{I}}
\newcommand{\GL}{\operatorname{GL}}
\newcommand{\SL}{\operatorname{SL}}
\newcommand{\W}{\operatorname{W}}
\newcommand{\R}{\operatorname{R}}
\newcommand{\Br}{\operatorname{Br}}
\newcommand{\Aut}{\operatorname{Aut}}
\newcommand{\End}{\operatorname{End}}
\newcommand{\Ind}{\operatorname{Ind}}
\newcommand{\Res}{\operatorname{Res}}
\newcommand{\br}{\operatorname{br}}
\newcommand{\Hom}{\operatorname{Hom}}
\newcommand{\Endo}{\operatorname{End}}
\newcommand{\Ho}{\operatorname{H}}
\newcommand{\Tr}{\operatorname{Tr}}
\newcommand{\opp}{\operatorname{opp}}
\theoremstyle{remark}

\theoremstyle{definition}
\newtheorem{definition}{Definition}[section]
\newtheorem{notation}[definition]{Notation}
\newtheorem{construction}[definition]{Construction}
\newtheorem{remark}[definition]{Remark}
\newtheorem{example}[definition]{Example}

\theoremstyle{plain}

\newtheorem{theorem}[definition]{Theorem}
\newtheorem{lemma}[definition]{Lemma}
\newtheorem{question}{Question}
\newtheorem{corollary}[definition]{Corollary}
\newtheorem{proposition}[definition]{Proposition}
\newtheorem{conjecture}[definition]{Conjecture}
\newtheorem{assumption}[definition]{Assumption}
\newtheorem{hypothesis}[definition]{Hypothesis}
\newtheorem{maintheorem}[definition]{Main Theorem}

\newtheorem*{theo*}{Theorem}
\newtheorem*{conj*}{Conjecture}
\newtheorem*{cor*}{Corollary}

\newtheorem{theo}{Theorem}
\newtheorem{conj}[theo]{Conjecture}
\newtheorem{cor}[theo]{Corollary}

\renewcommand{\thetheo}{\Alph{theo}}
\renewcommand{\theconj}{\Alph{conj}}
\renewcommand{\thecor}{\Alph{cor}}

%
%
%
%
%
%
%

%

\title{Derived equivalences and equivariant Jordan decomposition}

\date{\today}
\author{Lucas Ruhstorfer}
\address{Fachbereich Mathematik, TU Kaiserslautern, 67653 Kaiserslautern, Germany}
\email{ruhstorfer@mathematik.uni-kl.de}


\begin{abstract}
The Bonnafé--Rouquier equivalence can be seen as a modular analogue of Lusztig's Jordan decomposition for groups of Lie type. In this paper, we show that this equivalence can be lifted to include automorphisms of the finite group of Lie type. Moreover, we prove the existence of a local version of this equivalence which satisfies similar properties. 
\end{abstract}

\maketitle

\section*{Introduction}

\subsection*{Representation theory of groups of Lie type}

Establishing a conjecture by Brou\'e, Bonnaf\'e--Rouquier \cite{BoRo} and later Bonnaf\'e--Dat--Rouquier \cite{Dat} proved a Jordan decomposition for blocks of groups of Lie type. Let $\G$ be a connected reductive group with Frobenius endomorphism $F: \G \to \G$ defining an $\mathbb{F}_q$-structure on $\G$. Fix a prime $\ell$ coprime to $q$ and let $(\mathcal{O},K,k)$ be an $\ell$-modular system as in \ref{Grothendieck group} below. Suppose that $(\G^\ast,F^\ast)$ is a group in duality with $(\G,F)$ and fix a semisimple element $ s\in (\G^\ast)^{F^\ast}$ of $\ell'$-order. We denote by $e_s^{\G^F} \in \mathrm{Z}(\mathcal{O} \G^F)$ the central idempotent associated to the $(\G^\ast)^{F^\ast}$-conjugacy class of $s$ as in Brou\'e--Michel \cite{BrMi}. Suppose that $\Levi^\ast$ is the minimal $F^\ast$-stable Levi subgroup of $\G^\ast$ containing $\mathrm{C}_{(\G^\ast)^{F^\ast}}(s)\mathrm{C}_{\G^\ast}^\circ(s)$. Let $\Para$ be a parabolic subgroup of $\G$ with Levi decomposition $\Para= \Levi \ltimes \U$. Then the associated \textit{Deligne--Lusztig} variety
$$\Y_\U^\G= \{ g \U \in \G/ \U \mid g^{-1} F(g) \in \U F(\U) \}$$
has a left $\G^F$- and a right $\Levi^F$-action. Its $\ell$-adic cohomology groups $H^i_c(\Y_\U^\G,\mathcal{O})$ can therefore be considered as $\Lambda \G^F$-$\Lambda \Levi^F$-bimodules. Then the following was proved in \cite{Dat}:

\begin{theo}[Bonnaf\'e--Dat--Rouquier]\label{BDRintro}
	Let $\Levi^*$ be an $F^*$-stable Levi subgroup of $\G^*$ containing $\C^\circ_{\G^*}(s) \C_{(\G^\ast)^{F^\ast}}(s)$. Then the complex $C=G\Gamma_c(\Y_{\U}^\G, \mathcal{O})^{\operatorname{red}}e_s^{\Levi^F}$ of $\mathcal{O} \G^F e_s^{\G^F}$-$\mathcal{O} \Levi^F e_s^{\Levi^F}$ bimodules induces a splendid Rickard equivalence between $\mathcal{O}\G^F e_s^{\G^F}$ and $\mathcal{O}\Levi^F e_s^{\Levi^F}$. The bimodule $H^{\operatorname{dim}(\Y_{\U}^\G)}(C)$ induces a Morita equivalence between $\mathcal{O} \G^F e_s^{\G^F}$ and  $\mathcal{O} \Levi^F e_s^{\Levi^F}$.
\end{theo}

\subsection*{Clifford theory and group automorphisms}

The Jordan decomposition by Bonnaf\'e--Rouquier has proved to be extremely useful in the representation theory of finite groups of Lie type. For instance, the Bonnaf\'e--Rouquier Morita equivalence was a crucial ingredient in the verification of one direction of Brauer's height zero conjecture by Malle--Kessar \cite{KessarMalle}. Our main objective in this article is therefore to extend their results to include automorphisms.

Let us therefore from now on assume that $\G$ is a simple algebraic group of simply connected type and (for the sake of exposition) not of type $D_4$. Let $F: \G \to \G$ be a Frobenius endomorphism such that $\G^F/\mathrm{Z}(\G^F)$ is a finite simple group. We fix a regular embedding $ \iota: \G \hookrightarrow \Gtilde$.

 Using the classification of automorphisms of finite simple groups of Lie type we prove the existence of bijective morphisms $F_0: \tilde{\G} \to \tilde{\G}$ and $\sigma: \tilde{\G} \to \tilde{\G}$ stabilizing a Levi subgroup $\Levi$ of $\G$ in duality with $\Levi^\ast$ and such that the image of $\tilde{\G}^F \rtimes \mathcal{A}$, where $\mathcal{A}:= \langle \sigma|_{\tilde{\G}^F}, F_0|_{\tilde{\G}^F} \rangle$, generates the stabilizer of $e_s^{\G^F}$ in $\mathrm{Out}(\G^F)$. Moreover, these bijective morphisms commute with each other and the Frobenius endomorphism $F: \G \to \G$ is an integral power of $F_0$. Using this explicit description of automorphisms we can prove the following:

\begin{theo}[see Theorem \ref{Morita lift}]\label{intro2}
Assume that the order of $\sigma: \G^F \to \G^F$ is coprime to $\ell$.
Then $H^{\mathrm{dim}(\Y_\U^\G)}_c(\Y_\U^\G,\mathcal{O}) e_s^{\Levi^F}$ extends to an $\mathcal{O}[ (\G^F \times (\Levi^F)^{\mathrm{opp}}) \Delta(\mathcal{A})]$-module $M$. Moreover, the bimodule $\tilde{M}:=\mathrm{Ind}_{ (\G^F \times (\Levi^F)^{\mathrm{opp}}) \Delta(\mathcal{A}) }^{\tilde{\G}^F \mathcal{A} \times (\tilde{\Levi}^F \mathcal{A})^{\mathrm{opp}}} (M)$ induces a Morita equivalence between $\mathcal{O} \tilde{\Levi}^F \mathcal{A} e_s^{\Levi^F}$ and $\mathcal{O} \tilde{\G}^F \mathcal{A} e_s^{\G^F}$.
\end{theo}

One of the main results of \cite{Dat} is that the Morita equivalence in Theorem \ref{BDRintro} does not depend on the choice of the parabolic subgroup $\Para$. This shows that the bimodule $H^{\mathrm{dim}(\Y_\U^\G)}_c(\Y_\U^\G,\mathcal{O}) e_s^{\Levi^F}$ is $\Delta(\mathcal{A})$-invariant. However, this does not imply that the Morita bimodule extends to $\Delta(\mathcal{A})$ since $\mathcal{A}$ might not be cyclic. To remedy this problem we use a certain idea introduced by Digne \cite{Digne} in the context of restriction of scalars for Deligne--Lusztig varieties. This allows us to show that the module $H^{\mathrm{dim}(\Y^ \G_\U)}_c(\Y_\U^\G,\Lambda)e_s^{\Levi^F}$ can be endowed with a natural diagonal action of the automorphism $F_0|_{\tilde{\G}^F}$. From this we can show using the aforementioned independence result that the so-obtained bimodule is still invariant under the automorphism $\sigma$. Once we have proved this, Theorem \ref{intro2} is then a consequence of general results on Clifford theory of Morita equivalences. This result gives us the desired compatibility of the Bonnaf\'e--Rouquier equivalence with group automorphisms:

\begin{cor}
In the situation of Theorem \ref{intro2} we have the following commutative square of Grothendieck groups:
\begin{center}
\begin{tikzpicture}
  \matrix (m) [matrix of math nodes,row sep=3em,column sep=4em,minimum width=2em] {
  
    G_0(\mathcal{O} \Ltilde^F \mathcal{A} e_s^{\Levi^F} )  & G_0(\mathcal{O} \Gtilde^F \mathcal{A}  e_s^{\G^F})  \\
      G_0(\mathcal{O} \Levi^F e_s^{\Levi^F}) & G_0(\mathcal{O} \G^F e_s^{\G^F}) 
     \\};
\path[-stealth]
(m-1-2) edge node [right] {$\mathrm{Res}^{\Gtilde^F \mathcal{A}}_{\G^F}$} (m-2-2)
(m-1-1) edge node [above] {$[\tilde{M} \otimes -]$} (m-1-2)
(m-2-1) edge node [above] {$(-1)^{\mathrm{dim}(\Y_\U^\G)}  R_\Levi^{\G}$} (m-2-2)
(m-1-1) edge node [left] {$\mathrm{Res}^{\Ltilde^F \mathcal{A}}_{\Levi^F}$} (m-2-1);

\end{tikzpicture}
\end{center}
\end{cor}

\subsection*{Local equivalences}

Many open local-global conjectures, like the Alperin--McKay conjecture and the Alperin weight conjecture relate certain representations of a finite group to certain data of its local subgroups. Therefore, it is desirable to have a similar statement as Theorem \ref{intro2} for local subgroups, i.e. normalizers of $\ell$-subgroups.

Let $b$ be a block corresponding to the block $c$ under the Morita equivalence induced by $H_c^{\mathrm{dim}(\Y_\U^\G)}(\Y_\U^\G,\mathcal{O}) e_s^{\Levi^F}$. Then the blocks $b$ and $c$ have a common defect group $D$ contained in $\Levi^F$. We denote by $B_D$ the Brauer correspondent of $b$ and by $C_D$ the Brauer correspondent of $c$. In addition, we let $B_D'= \mathrm{Tr}_{ \mathrm{N}_{\tilde{\G}^F \mathcal{A}}(D,B_D)}^{\mathrm{N}_{\tilde{\G}^F \mathcal{A}}(D)} (B_D)$ and $C_D'= \mathrm{Tr}_{ \mathrm{N}_{\tilde{\Levi}^F \mathcal{A}}(D,C_D)}^{\mathrm{N}_{\tilde{\Levi}^F \mathcal{A}}(D)}( C_D)$ be the corresponding central idempotents of $\mathrm{N}_{\Gtilde^F \mathcal{A}}(D)$ and $\mathrm{N}_{\Ltilde^F \mathcal{A}}(D)$.

\begin{theo}[see Theorem \ref{loc}]\label{intro3}
Suppose that the assumptions of Theorem \ref{intro2} are satisfied.
Then the cohomology module $H^{\mathrm{dim}}_c(\Y^{\mathrm{N}_\G(D)}_{\C_{\U}(D)},\mathcal{O} )C_D$ extends to an $\mathcal{O}[(\mathrm{N}_{\G^F}(D) \times \mathrm{N}_{\Levi^F}(D)^{\mathrm{opp}} )\Delta( \mathrm{N}_{\tilde{\Levi}^F \mathcal{A}}(D,C_D))  ]$-module $M_D$. In particular, the bimodule
$$\mathrm{Ind}_{(\mathrm{N}_{\G^F}(D) \times \mathrm{N}_{\Levi^F}(D)^{\mathrm{opp}}) \Delta \mathrm{N}_{\tilde{\Levi}^F \mathcal{A}}(D,C_D)}^{\mathrm{N}_{\tilde{\G}^F \mathcal{A}}(D) \times \mathrm{N}_{\tilde{\Levi}^F \mathcal{A}}(D)^{\mathrm{opp}}}(M_D)$$
induces a Morita equivalence between $\mathcal{O} \mathrm{N}_{\tilde{\G}^F \mathcal{A}}(D) B_D'$ and $\mathcal{O} \mathrm{N}_{\tilde{\Levi}^F \mathcal{A}}(D) C_D'$.
%

\end{theo}

To prove this theorem, we first use the fact that the complex $C$ from theorem \ref{BDRintro} induces a splendid Rickard equivalence between $\mathcal{O} \G^F e_s^{\G^F}$ and $\mathcal{O} \Levi^F e_s^{\Levi^F}$. Using a theorem of Puig, we can deduce from this that the bimodule $H_c^{\mathrm{dim}}(\Y_{\C_\U(Q)}^{\mathrm{N}_\G(Q)}  ,\mathcal{O})$ induces a Morita equivalence between $\mathcal{O} \mathrm{N}_{\G^F}(D) B_D$ and $\mathcal{O} \mathrm{N}_{\Levi^F}(D) C_D$. We then generalize the proof of Theorem \ref{intro2} to the local situation to again extend this bimodule. An additional difficulty here is that we have to work in the non-connected reductive group $\mathrm{N}_{\G}(D)$.

\subsection*{Applications to local global conjectures}

In a second paper we use the strong equivariance properties obtained in Theorem \ref{intro2} and Theorem \ref{intro3} to reduce the verification of the inductive Alperin--McKay condition to quasi-isolated blocks, see \cite{Jordan2} for an exact statement. Quasi-isolated semisimple elements for reductive groups $\G$ have been classified by Bonnaf\'e \cite{Bonnafe} and are better understood by fundamental work of Cabanes--Enguehard and recent work of Enguehard and Kessar--Malle, see \cite{KessarMalle}. Our hope is therefore that the strong equivariance of the Bonnafé--Dat--Rouquier equivalence established in this paper will provide a method for verifying the inductive conditions for the Alperin--McKay and related local-global conjectures.

\subsection*{Acknowledgement}
%

%
%
The results of this article were obtained as part of my Phd thesis at the Bergische Universität Wuppertal. Therefore, I would like to express my gratitude to my supervisor Britta Späth for suggesting this topic and for her constant support. I would like to thank Marc Cabanes for reading through my thesis and for many interesting discussions. I am deeply indebted to Gunter Malle for his suggestions and his thorough reading.
I thank Radha Kessar and Markus Linckelmann for answering my technical questions on Rickard equivalences at the MSRI.
\ \\
This material is partly based upon work supported by the NSF under Grant DMS-1440140 while the author was in residence at the MSRI, Berkeley CA. The research was conducted in the framework of the research training group GRK 2240: Algebro-geometric
Methods in Algebra, Arithmetic and Topology, which is funded by the DFG.

\section{Representation theory}

In this chapter we introduce the necessary background material from the representation theory of finite groups. We give a brief overview on various categorical equivalences of module categories associated to finite groups. Subsequently, we then discuss the Clifford theory of these equivalences.

\subsection{Modular representation theory}\label{Grothendieck group}

Let $\ell$ be a prime and $K$ be a finite field extension of $\mathbb{Q}_\ell$. We say that $K$ is \textit{large enough for a finite group} $G$ if $K$ contains all roots of unity whose order divides the exponent of the group $G$. In the following, $K$ denotes a field which we assume to be large enough for the finite groups under consideration. We denote by $\mathcal{O}$ the ring of integers of $K$ over $\mathbb{Z}_\ell$ and by $k=\mathcal{O}/J(\mathcal{O})$ its residue field. We will use $\Lambda$ to interchangeably denote $\mathcal{O}$ or $k$.

Let $A$ be a $\Lambda$-algebra, finitely generated and projective as a $\Lambda$-module. We denote by $A^{\mathrm{opp}}$ its opposite algebra. Moreover, we mean by $A \text{-} \mathrm{mod}$ the category of left $A$-modules, that are finitely generated as $\Lambda$-modules. We denote by $G_0(A)$ the \textit{Grothendieck group} of the category $A \text{-} \mathrm{mod}$, see \cite[Section 5.1]{Benson}.
%

\subsection{Module categories}\label{Module category}

Let $\mathcal{A}$ be an abelian category. We denote by $\mathrm{Comp}^{b}(\mathcal{A})$ the category of bounded complexes of $\mathcal{A}$ and by $\mathrm{Ho}^b(\mathcal{A})$ its homotopy category. In addition, $\mathrm{D}^b(\mathcal{A})$ denotes the bounded derived category of $\mathcal{A}$. When $ \mathcal{A}=A \text{-} \mathrm{mod}$ we abbreviate $\mathrm{Comp}^{b}(\mathcal{A})$, $\mathrm{Ho}^b( \mathcal{A})$ and $\mathrm{D}^b(\mathcal{A})$ by $\mathrm{Comp}^b(A)$, $\mathrm{Ho}^b(A)$ and $\mathrm{D}^b(A)$ respectively. 

For $C \in \mathrm{Comp}^b(A)$ there exists (see for instance \cite[2.A.]{Dat}) a complex $C^{\mathrm{red}}$ with $C \cong C^{\mathrm{red}}$ in $\mathrm{Ho}^b(A)$ such that $C^{\mathrm{red}}$ has no non-zero direct summand which is homotopy equivalent to $0$. Moreover, $C \cong C^{\mathrm{red}} \oplus C_0$ with $H^\bullet(C_0) \cong 0$.

Let $A\text{-} \mathrm{proj}$ denote the full subcategory of $A \text{-} \mathrm{mod}$ consisting of all projective $A$-modules. We then denote by $A\text{-} \mathrm{perf}$ the full subcategory of $\mathrm{D}^b(A)$ consisting of complexes quasi-isomorphic to complexes of $\mathrm{Comp}^b(A\text{-} \mathrm{proj})$.

 Let $H$ and $G$ be finite groups and $C$ be a complex of $\Lambda G $-$\Lambda H$-bimodules. Then we write $C^\vee$ for the complex $\mathrm{Hom}_{\Lambda G}(C,\Lambda G)$ viewed as complex of $\Lambda H$-$\Lambda G$-bimodules. If $\Lambda$ denotes the trivial $\Lambda G$-$\Lambda H$-bimodule then we have by \cite[3.A.]{Broue2} an isomorphism $C^\vee \cong  \mathrm{Hom}_{\Lambda }(C,\Lambda ).$
Moreover, if $X$ is another complex of $\Lambda G$-$\Lambda H$ modules and $C$ is projective as $\Lambda H$-module then by \cite[3.A.]{Broue2} there is a canonical isomorphism $$C^\vee \otimes_{\Lambda G} X \cong \mathrm{Hom}_{\Lambda G}(C,X).$$

Let $\sigma:G\to G$ be an automorphism of a finite group $G$ and $H$ a subgroup of $G$. If $M$ is a left (resp. right) $\Lambda H$-module then we denote by ${}^\sigma M$ (resp. $M^\sigma$) the left (resp. right) $\Lambda \sigma(H)$-module which coincides with $M$ as a $\Lambda$-module but with action of $\sigma(H)$ given by $\sigma(h)m:= \sigma^{-1}(h) m$ (resp. by $m \sigma(h):=m \sigma^{-1}(h)$).


\subsection{The Brauer functor}\label{thebrauerfunctor}

Let $G$ be a finite group and $Q$ an $\ell$-subgroup of $G$. For a $\Lambda G$-module $M$ we let $M^Q$ denote the subset of $Q$-fixed points of $M$. We consider the \index{Brauer functor}\textit{Brauer functor} $$\Br^G_{Q}: \Lambda G \text{-} \operatorname{mod} \to k \operatorname{N}_G(Q)/Q \text{-} \operatorname{mod}$$
which for a $\Lambda G$-module $M$ is given by 
$$\Br^G_Q(M)= k \otimes_{\Lambda}(M^Q / \sum_{P < Q} \mathrm{Tr}_P^Q (M^P)),$$
where $\mathrm{Tr}_P^Q: M^P \to M^Q, \, m \mapsto \sum_{g \in Q/P} gm$ is the relative trace map on $M$.

Let $f: M_1 \to M_2$ be a morphism of $\Lambda G$-modules. Then $f$ restricts to a morphism $f: M_1^Q \to M_2^Q$ of $\Lambda \mathrm{N}_G(Q)$-modules. One readibly checks that $f$ maps $ \sum_{P < Q} \mathrm{Tr}_P^Q (M_1^P)$ to  $\sum_{P < Q} \mathrm{Tr}_P^Q (M_2^P)$ and we hence obtain by taking quotients a morphism $\Br_Q(f): \Br_Q(M_1) \to \Br_Q(M_2)$.

 If $H$ is a subgroup of $G$ containing $Q$ then by definition we have $$\Br_Q^H \circ \Res_H^G=\Res_{\mathrm{N}_H(Q)}^{\mathrm{N}_G(Q)} \circ \Br_Q^G.$$ Therefore, we will sometimes omit the upper index and write $\Br^G_Q=\Br_Q$ if the group under consideration is clear from the context.
 Since $\mathrm{Br}_Q$ is an additive functor it respects homotopy equivalences and therefore extends to a functor
$$\Br^G_Q:\mathrm{Ho}^b(\Lambda G) \to \mathrm{Ho}^b(k \mathrm{N}_G(Q)/Q).$$
Recall that a $\Lambda G$-module $M$ is called an \textit{$\ell$-permutation module} if
it is a direct summand of a permutation module, i.e., a module of the form $\Lambda[\Omega]$, where $\Omega$ is a $G$-set, see \cite[4.1.3]{Rouquier2}.
We let $\Lambda G \text{-} \operatorname{perm}$ be the full subcategory of $\Lambda G \text{-} \operatorname{mod}$ consisting of all $\ell$-permutation modules of $\Lambda G$.
If we consider $\Lambda G$ as $G$-module via $G$-conjugation, then $\Br_Q(\Lambda G) \cong k\C_G(Q)$ and the so-obtained surjection
$$ \br_Q^G: (\Lambda G)^Q \to k \C_G(Q), \, \sum_{g\in G} \lambda_g g \mapsto \sum_{g \in \C_G(Q)} \lambda_g g,$$
induces an algebra homomorphism, the so called \textit{Brauer morphism}, see \cite[Section 4.2]{Rouquier2}.


\subsection{Brauer pairs and the Brauer category}\label{BrauerBrauer}


Since the blocks of $\mathcal{O} G$ and $k G$ correspond to each other via lifting of idempotents, see \cite[Theorem 3.1]{Jacques}, we will identify blocks of $\mathcal{O} G$ and $k G$ if they correspond to each other via reduction modulo $J(\mathcal{O})$.

%
%
%
%

%


If $H$ is a subgroup of $G$ and $f \in \mathrm{Z}(\Lambda H)$ then we write $$\mathrm{N}_{G}(H,f):= \{ x\in \mathrm{N}_G(H) \mid {}^x f = f \}$$ for the set of elements normalizing $H$ and $f$. Moreover, we write $\mathrm{Tr}^G_H(f)=\sum_{ x \in G/H} {}^x f \in \mathrm{Z}(\Lambda G)$ for the trace of the element $f$.

If $b$ is a block of $G$ then we denote by $\mathcal{F}(G,b)$
the
\textit{Brauer category} of $b$, see \cite[§ 47]{Jacques}. 
If $(D,b_D)$ is a maximal $b$-Brauer pair then we denote by $\mathcal{F}(G,D)_{\leq (D,b_D)}$ the full subcategory of $\mathcal{F}(G,b)$ with objects consisting of all $b$-Brauer pairs contained in $(D,b_D)$. Recall that the natural inclusion functor $\mathcal{F}(G,b)_{\leq (D,b_D)} \hookrightarrow \mathcal{F}(G,b)$ induces an equivalence of categories, see e.g. \cite[Lemma 47.1  and afterwards]{Jacques}.


%

%
If $(Q,b_Q)$ is a $b$-Brauer pair then the idempotent $b_Q$ is also a block of $\mathrm{N}_G(Q,b_Q)$ by \cite[Exercise 40.2(b)]{Jacques}. Consequently, $B_Q:=\mathrm{Tr}^{\mathrm{N}_{G}(Q)}_{\mathrm{N}_G(Q,b_Q)}(b_Q)$ is a block of $\mathrm{N}_G(Q)$. Since all maximal Brauer pairs are $G$-conjugate it follows that $$\mathrm{br}_D(b)=\mathrm{Tr}^{\mathrm{N}_{G}(D)}_{\mathrm{N}_G(D,b_D)}(b_D).$$
For a finite group $G$ and $D$ an $\ell$-subgroup of $G$ we denote by $\mathrm{Bl}(G \mid D)$ the set of blocks of $G$ with defect group $D$. Then by Brauer's first main theorem we obtain a bijection
$$\br_D: \mathrm{Bl}(G \mid D) \to \mathrm{Bl}(\mathrm{N}_G(D) \mid D).$$

\subsection{Morita equivalences and splendid Rickard equivalences}


Let $G$ and $H$ be finite groups and let $e \in \Z(\Lambda G)$ and $f \in \Z(\Lambda H)$ be central idempotents. In addition, denote $A=\Lambda G e$ and $B=\Lambda H f$.

\begin{definition}\label{rickard}
Let $C$ be a bounded complex of $A$-$B$-bimodules, finitely generated and projective as $A$-modules resp. $B$-modules. We say that $C$ induces a \index{Rickard equivalence}\textit{Rickard equivalence between $A$ and $B$} if the following holds:
\begin{enumerate}[label={\alph*)}]
\item The canonical map $A \to \mathrm{End}^\bullet_{B^{\mathrm{opp}}}(C)^{\mathrm{opp}}$ is an isomorphism in $\mathrm{Ho}^b(A \otimes_\Lambda A^{\mathrm{opp}})$ and
\item the canonical map $B \to \mathrm{End}^\bullet_A(C)$ is an isomorphism in $\mathrm{Ho}^b(B \otimes_\Lambda B^{\mathrm{opp}})$.
\end{enumerate}
\end{definition}

We say that a complex $C \in \mathrm{Comp}^b(A \otimes B^{\mathrm{opp}})$ induces a \textit{derived equivalence} between $A$ and $B$ if the functor $$C \otimes^{\mathbb{L}}_{B} - \, : \mathrm{D}^b(B) \to \mathrm{D}^b(A)$$
induces an equivalence of triangulated categories.
%
%
A theorem of Rickard, see \cite[Section 2.1]{Rickard}, asserts that $A$ and $B$ are Rickard equivalent if and only if they are derived equivalent. More precisely, the proof of said theorem shows that not every complex $C \in \mathrm{Comp}^b(A \otimes B^{\mathrm{opp}})$ inducing a derived equivalence between $A$ and $B$ gives necessarily rise to a Rickard equivalence between $A$ and $B$.

Assume now that $H$ is a subgroup of $G$. For any subgroup $X$ of $H$ we let $\Delta X:=\{(x,x^{-1}) \mid x \in X \}$, a subgroup of $G\times H^{\mathrm{opp}}$.

\begin{definition}\label{defsplendid}
A bounded complex $C$ of $A$-$B$-bimodules is called \textit{splendid} if $C^{\mathrm{red}}$ is a complex of $\ell$-permutation modules such that every indecomposable direct summand of a component of $C$ has a vertex contained in $\Delta H$. If $C$ is splendid and induces a Rickard equivalence between $A$ and $B$ we say that $C$ induces a \textit{splendid Rickard equivalence} between $A$ and $B$.
\end{definition}

Note that our definition of a splendid Rickard equivalence is not symmetric since we assume that $H$ is a subgroup of $G$.


\subsection{First properties of splendid complexes}

Let $L$ be a subgroup of a finite group $G$ and $Q$ an $\ell$-subgroup of $L$. Then we can consider the Brauer functor
%
%
$$\Br_{\Delta Q}: \Lambda[ G\times L^{\mathrm{opp}}] \text{-} \operatorname{perm} \to k \operatorname{N}_{G\times L^{\mathrm{opp}}}(\Delta Q)/\Delta Q \text{-} \operatorname{perm}.$$
Notice that
$$\mathrm{N}_{G\times L^{\mathrm{opp}}}(\Delta Q)= (\mathrm{C}_G(Q) \times \mathrm{C}_L(Q)^{\mathrm{opp}}) \Delta (\mathrm{N}_L(Q)).$$
Let $c \in \mathrm{Z}(\Lambda L)$ and $b \in \mathrm{Z}(\Lambda G)$ be two central idempotents and suppose that $C$ a bounded complex of $\Lambda G b$-$\Lambda L c$-bimodules. Since $$\operatorname{C}_{G\times L^{\mathrm{opp}}}(\Delta Q)=\C_G(Q)\times \C_L(Q)^{\mathrm{opp}}\subseteq \operatorname{N}_{G\times L^{\mathrm{opp}}}(\Delta Q)$$ we can consider the image $\Br_{\Delta Q}(C)$ as a complex of $k\C_G(Q) \br_Q(b)$-$k\C_L(Q) \br_Q(c)$ bimodules.

The following theorem crucially uses an important theorem of Puig showing that the Brauer categories of splendid Rickard equivalent blocks are isomorphic.


\begin{theorem}\label{isomorphicdefect}
Let $L$ be a subgroup of a finite group $G$. Let $b \in \Z(\Lambda G )$ and $c \in \Z(\Lambda L )$ be primitive idempotents. Suppose that there exists a bounded complex $C$ of $\Lambda G b$-$\Lambda L c$-bimodules inducing a splendid Rickard equivalence between $\Lambda G b$ and $\Lambda L c$. If $D$ is a defect group of the block $c$ then $D$ is a defect group of $b$.
\end{theorem}

\begin{proof}
Denote $A= \Lambda G b$ and $B= \Lambda L c$. Since $C$ induces a splendid Rickard equivalence between $A$ and $B$ it follows by definition that $B \cong \mathrm{End}_A^{\bullet}(C)$ in $\mathrm{Ho}^b(B \otimes_\Lambda B^{\mathrm{opp}})$. We obtain $$\mathrm{Br}_{\Delta D}(\mathrm{End}_A^{\bullet}(C)) \cong \mathrm{End}^\bullet_{k \mathrm{C}_G(D)}(\Br_{\Delta D}(C)).$$
Since $\Br_{\Delta D}(B) \cong k \mathrm{C}_L(D) \mathrm{br}_D(c)$ we obtain 
$$ \mathrm{End}^\bullet_{k \mathrm{C}_G(D)}(\Br_{\Delta D}(C)) \cong k \mathrm{C}_L(D) \mathrm{br}_D(c).$$
Taking cohomology yields $\mathrm{End}_{\mathrm{Ho}^b(k \mathrm{C}_G(D))}(\Br_{\Delta D}(C)) \cong k \mathrm{C}_L(D) \mathrm{br}_D(c)$. Since $D$ is a defect group of $c$ it follows that $\mathrm{br}_D(c) \neq 0$. Therefore, the complex $\Br_{\Delta D}(C)$ is not homotopy equivalent to $0$ in $\mathrm{Ho}^b(k \mathrm{C}_G(D))$. As $\Br_{\Delta D}(C)$ is a complex of $k\C_G(D) \br_D(b)$-$k\C_L(D) \br_D(c)$ bimodules it follows that $\br_D(b) \neq 0$. This shows that $D$ is contained in a defect group of $b$. Since $C$ induces a splendid Rickard equivalence it follows that $C$ induces a basic Rickard equivalence between the blocks $\Lambda G b$ and $\Lambda L c$, see beginning of \cite[Section 19.2]{Puig}. Consequently, \cite[Theorem 19.7]{Puig} shows that the defect groups of $b$ and $c$ are isomorphic. Thus, $D$ is also a defect group of $b$.
\end{proof}

\begin{proposition}\label{splendidloc}
Take the notation as in Theorem \ref{isomorphicdefect} and fix a maximal $c$-Brauer pair $(D,c_D)$. Then there exists a $b$-Brauer pair $(D,b_D)$ such that the following holds: If $(Q,c_Q) \leq (D,c_D)$ is a $c$-Brauer subpair then the $b$-Brauer subpair $(Q,b_Q) \leq (D,b_D)$ is the unique $b$-Brauer pair such that the complex $b_Q \Br_{\Delta Q}(C) c_Q$ induces a Rickard equivalence between $k \C_{G}(Q) b_Q$ and $k \C_{L}(Q) c_Q$. For any other $b$-Brauer pair $(Q,b'_Q)$ we have $b'_Q \Br_{\Delta Q}(C) c_Q \cong 0$ in $\mathrm{Ho}^b( k [\mathrm{C}_G(Q)  \times  \mathrm{C}_L(Q)^{\mathrm{opp}}])$.
\end{proposition}

\begin{proof}
The subgroup $D \subseteq L \subseteq G$ is a common defect group of the blocks $b$ and $c$ by Theorem \ref{isomorphicdefect}. Moreover, the complex $C$ is splendid, so the vertices of all indecomposable direct summands of components of $C$ are by definition contained in $\Delta L$. On the other hand, if $P$ is an $\ell$-subgroup of $L$ then $\Br_{\Delta P}(C) \cong \br_P(b) \Br_{\Delta P}(C) \cong 0$, unless $P$ is contained in a defect group of the block $b$. It follows that all indecomposable direct summands of components of $C$ are relatively $\Delta D$-projective. Hence, the complex $C$ induces a splendid Rickard equivalence between $k G b$ and $k L c$ in the sense of \cite{Harris}. The statement is therefore precisely \cite[Theorem 1.6]{Harris}.
\end{proof}

Let $b$ be a block of a finite group $G$ and $(D,b_D)$ a maximal $b$-Brauer pair. Recall from \ref{BrauerBrauer} that we denote by $\mathcal{F}(G,b)$ the Brauer category of $b$ and by $\mathcal{F}(G,D)_{\leq (D,b_D)}$ its full subcategory consisting of all $b$-Brauer pairs contained in $(D,b_D)$.

\begin{theorem}\label{Puig}
Suppose that we are in the situation of Proposition \ref{splendidloc}.
Then the map $\mathcal{F}(L,c)_{\leq (D,c_D)} \to \mathcal{F}(G,b)_{\leq (D,b_D)}$ given by $(Q,c_Q) \mapsto (Q,b_Q)$ induces an isomorphism of categories. In particular, for any two $c$-Brauer subpairs $(Q,c_Q)$, $(R,c_R)$ contained in $(D,c_D)$ and $b$-Brauer subpairs $(Q,b_Q)$, $(R,b_R)$ contained in $(D,b_D)$ we have
$$\Hom_{\mathcal{F}(L,c)}((Q,c_Q),(R,c_R))=\Hom_{\mathcal{F}(G,b)}((Q,b_Q),(R,b_R)).$$
\end{theorem}

\begin{proof}
 The paragraph below \cite[Theorem 1.7]{Harris} shows that we have an inclusion
$$\Hom_{\mathcal{F}(L,c)}((Q,c_Q),(R,c_R)) \subseteq \Hom_{\mathcal{F}(G,b)}((Q,b_Q),(R,b_R)).$$
By \cite[Theorem 19.7]{Puig} the Brauer categories $\mathcal{F}(L,c)$ and $\mathcal{F}(G,b)$ are equivalent. Consequently, the inclusion above is an equality.
\end{proof}

By definition of the Brauer category we have $\operatorname{Aut}_{ \mathcal{F}(G,b)}(Q,b_Q)= \operatorname{N}_G(Q,b_Q) / \C_G(Q)$.
	Therefore, Theorem \ref{Puig} implies the following corollary.

\begin{corollary}\label{factor}
Suppose that we are in the situation of Proposition \ref{splendidloc}. Then for any subgroup $Q$ of $D$ the inclusion map $\operatorname{N}_{L}(Q)/\C_{L}(Q)\hookrightarrow \operatorname{N}_{G}(Q)/\C_{G}(Q)$ induces an isomorphism between $\operatorname{N}_{L}(Q,c_Q) / \C_{L}(Q) $ and $\operatorname{N}_{G}(Q,b_Q) / \C_{G}(Q) $.
\end{corollary}

\subsection{Lifting Rickard equivalences}\label{Marcus}

The aim of this section is to introduce a lifting result for Morita equivalences due to Marcus.
We first need to introduce some notation. Let $\tilde{L}$ be a subgroup of a finite group $\tilde{G}$. Moreover, let $G$ be a normal subgroup of $\tilde{G}$ and set $L := G \cap \tilde{L}$.

Let $e\in \Z(\mathcal{O} G)$ and $f \in \Z(\mathcal{O} L)$ be $\tilde{G}$-invariant resp. $\tilde{L}$-invariant central idempotents, such that $e \in \Z(\mathcal{O} \tilde{G})$ and $f\in \Z(\mathcal{O} \tilde{L})$. Consider the diagonal subgroup
$$\mathcal{D}:=\{ (\tilde{g},\tilde{l})\in \tilde{G} \times \tilde{L}^{\operatorname{opp}} \mid \tilde{g} G= \tilde{l}^{-1} G \}=  (G \times L^\mathrm{opp})  \Delta(\tilde{L})$$
of $\tilde{G} \times \tilde{L}^{\operatorname{opp}}$. The following was first proved in \cite[Theorem 3.4]{Marcus}. An alternative proof can be found in \cite[Lemma 2.8]{Rouquier3}.

\begin{theorem}[Marcus]\label{lifting}
Suppose that $ \tilde{G}=  \tilde{L} G$. Let $C$ be a bounded complex of $\Lambda G e$-$\Lambda L f$-bimodules inducing a Rickard equivalence between $\Lambda G e$ and $\Lambda L f$. Suppose that $C$ extends to a complex of $\mathcal{D}$-modules $C'$ and define $\tilde{C}:= \mathrm{Ind}_{\mathcal{D}}^{\tilde{G} \times \tilde{L}^{\mathrm{opp}}}(C')$.
\begin{enumerate}[label={\alph*)}]
	\item The complex $\tilde{C}$ induces a derived equivalence between $\Lambda \tilde{L} f$ and $\Lambda \tilde{G} e$.
	\item If $C$ is concentrated in one degree or if $\ell \nmid[ \tilde{L} :L]$ then $\tilde{C}$  induces a Rickard equivalence between $\Lambda \tilde{L} f$ and $\Lambda \tilde{G} e$.
\end{enumerate}
\end{theorem}

\begin{proof}
The statement of part (b) has been proved in the case where $e$ and $f$ are primitive central idempotents in \cite[Lemma 2.8]{Rouquier3}. However, the assumption in the proof of  \cite[Lemma 2.8]{Rouquier3} that $e$ and $f$ are primitive is not necessary. Furthermore, as said in \cite[Remark 5.4]{Rouquier} if we drop the assumption that $[\tilde{L}:L]$ is coprime to $\ell$ in Theorem \ref{lifting} it is still true that $\tilde{C}$ induces a derived equivalence between $\Lambda \tilde{L} f$ and $\Lambda \tilde{G} e$.
\end{proof}

In the following remark we observe some Clifford-theoretic consequences of Theorem \ref{lifting}.

\begin{remark}\label{characters}
\
\begin{enumerate}[label={\alph*)}]
\item 
Suppose that we are in the situation of Theorem \ref{lifting}.
Let $\varphi: \mathrm{Ho}^b(\Lambda L f) \to \mathrm{Ho}^b( \Lambda G e)$ and $\tilde{\varphi}: \mathrm{Ho}^b(\Lambda \tilde{L} f) \to \mathrm{Ho}^b(\Lambda \tilde{G} e)$ be the functors induced by tensoring with $C$ resp. $\tilde{C}$. Let $N$ be a complex of $\Lambda \tilde{L} f$-modules. Then by Mackey's formula $                 
\mathrm{Res}_{G \times \tilde{L}^{\mathrm{opp}}}^{\tilde{G} \times \tilde{L}^{\mathrm{opp}}} (\tilde{C}) \cong \mathrm{Ind}_{G \times L^{\mathrm{opp}}}^{G \times \tilde{L}^{\mathrm{opp}}} (C)$. In particular, we have
$$\mathrm{Res}_{G}^{\tilde{G}} (\tilde{C} \otimes^{}_{\Lambda \tilde{L}} N ) \cong \mathrm{Ind}_{G \times L^{\mathrm{opp}}}^{G \times \tilde{L}^{\mathrm{opp}}} (C) \otimes^{}_{ \Lambda \tilde{L}} N \cong (C \otimes_{\Lambda L}  \Lambda \tilde{L} ) \otimes^{}_{\Lambda \tilde{L}} N \cong C \otimes^{}_{\Lambda L} \mathrm{Res}_L^{\tilde{L}} (N).$$
In other words, $\mathrm{Res}^{\tilde{G}}_G \circ \tilde{\varphi} \cong \varphi \circ \Res^{\tilde{L}}_L $. A similar calculation (or using the fact that $\mathrm{Ind}$ and $\mathrm{Res}$ are adjoint functors) shows that $\mathrm{Ind}^{\tilde{G}}_G \circ \varphi \cong \tilde{\varphi} \circ \Ind^{\tilde{L}}_L$.
\item Let $M$ be an $\mathcal{O} G e$-$\mathcal{O} L f$ bimodule inducing a Morita equivalence between $\mathcal{O} G e$ and $\mathcal{O} L f$. Suppose that $M$ extends to an $\mathcal{O} \mathcal{D}$-module $M'$ and denote $\tilde{M}:=\mathrm{Ind}_\Delta^{\tilde{G} \times \tilde{L}^{\mathrm{opp}}}(M')$. For $R \in \{ K,k\}$ the bimodule $M \otimes_{\mathcal{O}} R$ (respectively $\tilde{M} \otimes_{\mathcal{O}} R$) induces a bijection $\varphi: \Irr( R L f) \to \Irr( R G e)$ (respectively $\tilde{\varphi}: \Irr(R \tilde{L} f) \to \Irr(R \tilde{G} e)$) between irreducible modules. Now suppose that $N$ is a simple $R \tilde{L} f$-module. By Clifford's theorem, see \cite[Theorem 3.3.1]{Nagao}
we see that the simple $R L$-module $S$ extends to an $R \tilde{L}$-module if and only if $\varphi(S)$ extends to an $R \tilde{G}$-module.
\end{enumerate}
\end{remark}

%

\subsection{Descent of Rickard equivalences}

We keep the assumptions of the previous section. Theorem \ref{lifting} shows that under certain conditions Rickard equivalences can be lifted from normal subgroups. It is therefore natural to ask whether one can also go the other way. For Rickard equivalences we obtain the following converse to Theorem \ref{lifting} which is tailored to our later applications.

\begin{lemma}\label{Marcusconverse}
Suppose that $  \tilde{G} = \tilde{L} G$. Let $C$ be a bounded complex of biprojective $\Lambda G e$-$\Lambda L f$-bimodules
with cohomology concentrated in degree $d$ such that $H^d(C)$ induces a Morita equivalence between $\Lambda G e$ and $\Lambda L f$.
Assume that $C$ extends to a complex of $\Lambda \mathcal{D}$-modules $C'$ such that $\tilde{C}:= \mathrm{Ind}_{\mathcal{D}}^{\tilde{G} \times \tilde{L}^{\mathrm{opp}}} (C')$ induces a Rickard equivalence between $\Lambda \tilde{L} f$ and $\Lambda \tilde{G} e$. Then also the complex $C$ induces a Rickard equivalence between $\Lambda G e$ and $\Lambda L f$.
\end{lemma}

\begin{proof}
By the Mackey formula we have 
$$\mathrm{Res}_{G \times \tilde{L}^{\mathrm{opp}}}^{\tilde{G} \times \tilde{L}^{\mathrm{opp}}}(\tilde{C}) \cong C \otimes_{\Lambda L} \Lambda \tilde{L} \text { and } \mathrm{Res}_{\tilde{G} \times L^{\mathrm{opp}}}^{\tilde{G} \times \tilde{L}^{\mathrm{opp}}}(\tilde{C}) \cong \Lambda \tilde{G} \otimes_{\Lambda G} C.$$
Since $\tilde{C}$ induces a Rickard equivalence between $\Lambda \tilde{L} f$ and $\Lambda \tilde{G} e$ we therefore conclude that 
$$\mathrm{Res}^{\tilde{G} \times \tilde{G}^{\mathrm{opp}}}_{G \times \tilde{G}^{\mathrm{opp}}}( \Lambda \tilde{G} e) \cong C \otimes_{\Lambda L} \Lambda \tilde{L} \otimes_{\Lambda \tilde{L}} \tilde{C}^\vee \cong C \otimes_{\Lambda L} C^\vee \otimes_{\Lambda G} \Lambda \tilde{G}.$$
Since $H^d(C)$ induces a Morita equivalence between $\Lambda G e$ and $\Lambda L f$ it follows by the remarks before \cite[Lemma 10.2.4]{Rouquier3} that we have an isomorphism $$C \otimes_{\Lambda L} C^\vee \cong \Lambda G e \oplus R$$
in $\mathrm{Comp}^b( \Lambda [G \times G^{\mathrm{opp}}])$, where $R$ is a complex of $\Lambda G e $-$\Lambda Ge $-bimodules such that $H^\bullet(R) \cong 0$ (but not necessarily homotopy equivalent to $0$). From this we deduce that
$$\mathrm{Res}^{\tilde{G} \times \tilde{G}^{\mathrm{opp}}}_{G \times \tilde{G}^{\mathrm{opp}}}( \Lambda \tilde{G} e ) \cong  C \otimes_{\Lambda L} C^\vee \otimes_{\Lambda G} \Lambda \tilde{G} \cong \mathrm{Res}^{\tilde{G} \times \tilde{G}^{\mathrm{opp}}}_{G \times \tilde{G}^{\mathrm{opp}}}( \Lambda \tilde{G} e ) \oplus ( R \otimes_{\Lambda G} \Lambda \tilde{G} )$$
in $\mathrm{Ho}^b(\Lambda[G \times \tilde{G}^{\mathrm{opp}}])$. We conclude that $$\mathrm{Ind}_{G \times G^{\mathrm{opp}}}^{G \times \tilde{G}^{\mathrm{opp}}}(R) \cong R \otimes_{\Lambda G} \Lambda \tilde{G} \cong 0$$ in $\mathrm{Ho}^b(\Lambda[G \times \tilde{G}^{\mathrm{opp}}])$. Since $R$ is a direct summand of $\mathrm{Res}_{G \times G^{\mathrm{opp}}}^{G \times \tilde{G}^{\mathrm{opp}}}( \mathrm{Ind}_{G \times G^{\mathrm{opp}}}^{G \times \tilde{G}^{\mathrm{opp}}}(R))$ as a complex we thus have $ R \cong 0$ in $\mathrm{Ho}^b(\Lambda[G \times G^{\mathrm{opp}}])$. This shows that $C \otimes_{\Lambda L} C^\vee \cong \Lambda G e$ in $\mathrm{Ho}^b(\Lambda[G \times G^{\mathrm{opp}}])$ and similarly one proves $C^\vee \otimes_{\Lambda G} C \cong \Lambda Lf$ in $\mathrm{Ho}^b(\Lambda[L \times L^{\mathrm{opp}}])$. Consequently, the complex $C$ induces a Rickard equivalence between $\Lambda G e$ and $\Lambda L f$.
\end{proof}

It would be interesting to know whether the hypothesis that $C$ has cohomology concentrated in degree $d$ such that $H^d(C)$ induces a Morita equivalence between $\Lambda G e$ and $\Lambda L f$ could be weakened or even completely removed. For Morita equivalences the following lemma shows that the situation is much easier:

\begin{lemma}\label{MarcusconverseMorita}
Suppose that $ \tilde{L} G= \tilde{G}$. Let $M$ be a biprojective $\Lambda G e$-$\Lambda L f$-bimodule and suppose that $M$ extends to a $\Lambda\mathcal{D}$-module $M'$ such that $\tilde{M}:= \mathrm{Ind}_{\mathcal{D}}^{\tilde{G} \times \tilde{L}^{\mathrm{opp}}} (M')$ induces a Morita equivalence between $\Lambda \tilde{L} f$ and $\Lambda \tilde{G} e$. Then $M$ induces a Morita equivalence between $\Lambda G e$ and $\Lambda L f$.
\end{lemma}

\begin{proof}
Since $\tilde{M}$ induces a Morita equivalence between $\Lambda \tilde{L} f$ and $\Lambda \tilde{G} e$ it follows that the natural map $\Lambda \tilde{G} e \to \End_{(\Lambda \tilde{L})^{\mathrm{opp}}}( \tilde{M})^{\mathrm{opp}}$ is an isomorphism. This shows that the natural map 
$$\Lambda \tilde{G} e \to  \End_{(\Lambda L)^{\mathrm{opp}}}( \mathrm{Res}_{\tilde{G} \times L^{\mathrm{opp}}}^{\tilde{G} \times \tilde{L}^{\mathrm{opp}}}(\tilde{M}))^{\mathrm{opp}} \cong \End_{(\Lambda L)^{\mathrm{opp}}}(\Lambda \tilde{G} \otimes_{\Lambda G} M)^{\mathrm{opp}}$$
is injective. From this it follows that the natural map $\Lambda G e \to \End_{(\Lambda L)^{\mathrm{opp}}}(M)^{\mathrm{opp}}$ is injective as well. Since $\Lambda  G e$ is projective as right $\Lambda G$-module it follows that the map $\Lambda G e \to \End_{(\Lambda L)^{\mathrm{opp}}}(M)^{\mathrm{opp}}$ is a split injection of right $\Lambda G$-modules. Consequently, there exists a right $\Lambda G$-module $R$ such that 
$$\End_{(\Lambda L)^{\mathrm{opp}}}(M)^{\mathrm{opp}} \cong M \otimes_{\Lambda L} M^\vee \cong \Lambda G e \oplus R$$
as right $\Lambda G$-modules. We now want to show that $R\cong 0$. According to the proof of Lemma \ref{Marcusconverse} we have $$\mathrm{Res}^{\tilde{G} \times \tilde{G}^{\mathrm{opp}}}_{G \times \tilde{G}^{\mathrm{opp}}}( \Lambda \tilde{G} e) \cong M \otimes_{\Lambda L} M^\vee \otimes_{\Lambda G} \Lambda \tilde{G}.$$
It follows that $ \Lambda \tilde{G} e \cong  \Lambda \tilde{G} e  \oplus ( R \otimes_{\Lambda G} \Lambda \tilde{G} e)$ as right $\Lambda \tilde{G}$-modules. We conclude that $ R \otimes_{\Lambda G} \Lambda \tilde{G} \cong 0 $ which implies that $R\cong 0$. Hence, the natural map $\Lambda G e \to \End_{(\Lambda L)^{\mathrm{opp}}}(M)^{\mathrm{opp}}$ is an isomorphism. Similarly, one shows that $\Lambda L f \to \End_{\Lambda G}(M)$ is an isomorphism.
\end{proof}

\subsection{Rickard equivalences for the normalizer}\label{Marcus2}

We continue our discussion on Marcus' theorem.
Let $\tilde{L}$ be a subgroup of a finite group $\tilde{G}$. Moreover, let $G$ be a normal subgroup of $\tilde{G}$ and set $L:= \tilde{L} \cap G$. Let $e \in \mathrm{Z}(\Lambda G)$ and $f \in \mathrm{Z}(\Lambda L)$ be central idempotents and denote by $L':=\mathrm{N}_{\tilde{L}}(f)$ and $G':=\mathrm{N}_{\tilde{G}}(e)$ their respective stabilizers. In this section we suppose that $G'=G L'$.
We denote 
$$\mathcal{D}':= (G \times L^{\opp})  \Delta(L') \text{ and } \mathcal{D}:= ( G \times L^{\opp}) \Delta(\tilde{L}).$$
In what follows, we assume that $f ({}^l f)=0$ for any $l \in \tilde{L} \setminus L'$ and $e ({}^g e)=0$ for any $g \in \tilde{G} \setminus G'$. This ensures that $F:=\mathrm{Tr}_{\mathrm{N}_{\tilde{L}}(f)}^{\tilde{L}}(f)$ is a central idempotent of $\Lambda \tilde{L}$ and $E:=\mathrm{Tr}_{\mathrm{N}_{\tilde{G}}(e)}^{\tilde{G}}(e)$ is a central idempotent of $\Lambda \tilde{G}$.

%

%

\begin{proof}
By Theorem \ref{lifting} the $\Lambda$-algebras $\Lambda L' f$ and $\Lambda G' e$ are Rickard equivalent via the complex $\Ind_{\mathcal{D}'}^{G' \times (L')^{\mathrm{opp}}}(C')$. By Clifford theory, $\Lambda L' f$ is Morita equivalent to $\Lambda \tilde{L} F$. The same argument shows that $\Lambda G' e$ and $\Lambda \tilde{G} E$ are Morita equivalent. Thus, the algebras $\Lambda \tilde{L} F$ and $\Lambda \tilde{G} E$ are Rickard equivalent and the Rickard equivalence is given by the complex
$$\Lambda \tilde{G} e \otimes_{\Lambda G' } \Ind_{\mathcal{D}'}^{G' \times (L')^{\mathrm{opp}}}(C') \otimes_{\Lambda L'} f \Lambda \tilde{L} 
\cong \Ind_{\mathcal{D}'}^{\tilde{G} \times (\tilde{L})^{\mathrm{opp}}}(C').  \qedhere $$
\end{proof}

%
%

\begin{lemma}\label{better version}
Let $C$ be a bounded complex of $\Lambda G$-$\Lambda L$-bimodules and assume that $eCf$ induces a Rickard equivalence between $\Lambda G e$ and $\Lambda L f$. In addition, suppose that ${}^l e C f \cong 0$ in $\mathrm{Ho}^b(\Lambda [G \times L^{\mathrm{opp}}])$ for all $l \in \tilde{L} \setminus L'$.  Suppose that $C$ is either concentrated in one degree or that $\ell \nmid[L':L]$. If $C$ extends to a complex of $\Lambda \mathcal{D}$-modules $C'$ then $\Lambda \tilde{L} F$ and $\Lambda \tilde{G} E$ are Rickard equivalent via the complex $$E \, \mathrm{Ind}_{\mathcal{D}}^{\tilde{G} \times \tilde{L}^{\mathrm{opp}}}(C') \, F.$$
\end{lemma}

\begin{proof}
The complex $C^0:=e \mathrm{Res}^{\mathcal{D}}_{\mathcal{D}'}(C') f$ is clearly a $\Lambda \mathcal{D}'$-complex extending $e C f$. By Theorem \ref{lifting} $\Lambda L' f$ and $\Lambda G' e$ are Rickard equivalent via the complex $\Ind_{\mathcal{D}'}^{G' \times (L')^{\mathrm{opp}}}(C^0)$. By definition, the stabilizer of $f$ in $\tilde L$ is $L'$ and so the induction functor yields a Morita equivalence between $\Lambda L' f$ and $\Lambda \tilde{L} F$. The same argument shows that $\Lambda G' e$ and $\Lambda \tilde{G} E$ are Morita equivalent. Thus, the algebras $\Lambda \tilde{L} F$ and $\Lambda \tilde{G} E$ are Rickard equivalent and the Rickard equivalence is given by the complex $\Ind_{\mathcal{D}'}^{\tilde{G} \times \tilde{L}^{\mathrm{opp}}}(C^0)$ induces a Rickard equivalence between $\Lambda \tilde{L}F$ and $\Lambda \tilde{G}E$.

Recall that $\mathcal{D}'$ is by definition the stabilizer in $\tilde{G} \times \tilde{L}^{\mathrm{opp}}$ of the idempotent $e \otimes f$. Since $\tilde{L}/L' \cong \mathcal{D} / \mathcal{D}'$, we have
$$\Tr_{\mathcal{D}'}^{\mathcal{D}}(e \otimes f)=\sum_{l\in \tilde{L} /L'} {}^{(l,l^{-1})} (e \otimes f).$$
By assumption we have ${}^l e C f \cong 0$ for all $l \in \tilde{L} \setminus L'$. From this it follows that 
$$\Tr_{\mathcal{D}'}^{\mathcal{D}}(e \otimes f) C  \cong \mathrm{Tr}_{\mathrm{N}_{\tilde{G}}(e)}^{\tilde{G}}(e) C \mathrm{Tr}_{\mathrm{N}_{\tilde{L}}(f)}^{\tilde{L}}(f) = E C F.$$
We have $\Ind_{\mathcal{D}'}^{\mathcal{D}}(e \mathrm{Res}^{\mathcal{D}}_{\mathcal{D}'} C' f)\cong E C' F$ and as $E \otimes F$ is a central idempotent of $\Lambda [\tilde{G} \times {\tilde{L}}^{\mathrm{opp}}]$ it follows that $\mathrm{Ind}_{\mathcal{D}}^{\tilde{G} \times \tilde{L}^{\mathrm{opp}}}(E C' F) \cong E \mathrm{Ind}_{\mathcal{D}}^{\tilde{G} \times \tilde{L}^{\mathrm{opp}}}(C') F$.
\end{proof}

We can use this lifting result to prove the following proposition.

\begin{proposition}\label{normalizerderived}
Let $C$ be a bounded complex of $\ell$-permutation modules inducing a splendid Rickard equivalence between the blocks $\Lambda G b$ and $\Lambda L c$. Let $(Q,c_Q)$ be a $c$-Brauer pair corresponding to the $b$-Brauer pair $(Q,b_Q)$ under the splendid Rickard equivalence given by the complex $C$ as in Proposition \ref{splendidloc}. Then the complex 
$$\mathrm{Ind}_{\mathrm{N}_{G \times L^{\mathrm{opp}}}(\Delta Q)}^{\mathrm{N}_G(Q) \times \mathrm{N}_L( Q)^{\mathrm{opp}}} (\Br_{\Delta Q}(C)) \Tr^{\mathrm{N}_L(Q)}_{\mathrm{N}_{L}(Q,c_Q)}(c_Q)$$
induces a derived equivalence between the blocks $k \mathrm{N}_{G}(Q) \Tr^{\mathrm{N}_G(Q)}_{\mathrm{N}_{G}(Q,b_Q)}(b_Q)$ and $k \mathrm{N}_L(Q) \Tr^{\mathrm{N}_L(Q)}_{\mathrm{N}_{L}(Q,c_Q)}(c_Q)$.
\end{proposition}

\begin{proof}
Recall that $\Br_{\Delta Q}(C)$ is a complex of $k \mathrm{N}_{G \times L^{\mathrm{opp}}}(Q)$-modules such that $b_Q \Br_{\Delta Q}(C) c_Q \cong \Br_{\Delta Q}(C) c_Q $ induces a Rickard equivalence between $k \C_{G}(Q) b_Q$ and $k \C_{L}(Q) c_Q$, see Proposition \ref{splendidloc}. Moreover, the groups $\mathrm{N}_G(Q,b_Q)/\C_G(Q)$ and $\mathrm{N}_L(L,c_Q)/\C_L(Q)$ are isomorphic by Corollary \ref{factor}. 
Thus, using the proof of Lemma \ref{better version} together with Theorem \ref{lifting}, we conclude that the complex $$\mathrm{Ind}_{\mathrm{N}_{G \times L^{\mathrm{opp}}}(\Delta  Q)}^{\mathrm{N}_G(Q) \times \mathrm{N}_L( Q)^{\mathrm{opp}}}(\Br_{\Delta Q}(C)) \Tr^{\mathrm{N}_L(Q)}_{\mathrm{N}_{L}(Q,c_Q)}(c_Q)$$
induces a derived equivalence between the blocks $k \mathrm{N}_{G}(Q) \Tr^{\mathrm{N}_G(Q)}_{\mathrm{N}_{G}(Q,b_Q)}(b_Q)$ and $k \mathrm{N}_L(Q) \Tr^{\mathrm{N}_L(Q)}_{\mathrm{N}_{L}(Q,c_Q)}(c_Q)$.
\end{proof}

%
If the defect group $D$ of $b$ is abelian then the quotient group $\mathrm{N}_G(D,b_D)/\C_G(D)$ is of $\ell'$-order, see \cite[Theorem 40.14]{Jacques}. In this case, the proof of Proposition \ref{normalizerderived} shows that the complex $\mathrm{Ind}_{\mathrm{N}_{G \times L^{\mathrm{opp}}}(\Delta D)}^{\mathrm{N}_G(D) \times \mathrm{N}_L( D)^{\mathrm{opp}}} \Br_{\Delta D}(C)$ induces in fact a Rickard equivalence between $k \mathrm{N}_{G}(D) \br_D(b)$ and $k \mathrm{N}_L(D) \br_D(c)$.

\subsection{The Brauer functor and Clifford theory}

In this section we recall some results of \cite[Section 3]{Marcus} and generalize them slightly. These results will be needed in Section \ref{RickardClifford}.

Whenever $G$ is a finite group and $Q$, $R$ are subgroups of $G$, then we let 
$$T_G(Q,R):= \{ g \in G \mid Q^g \subseteq R \}.$$
%
%


The following lemma is a variant of \cite[Lemma 3.7]{Marcus}. A complete proof can be found in \cite[Lemma 1.39]{PhD}

\begin{lemma}\label{BrauerMarcus}
Let $H$ be a subgroup of $G$ and $Q \subseteq P$ two $\ell$-subgroups of $H$. Suppose that $\mathrm{C}_G(Q) T_H(Q,P)= T_G(Q,P)$. Then for every relatively $P$-projective module $M \in k H \text{-} \mathrm{perm}$ there is a natural isomorphism 
$$ \mathrm{Ind}_{\mathrm{N}_H(Q)}^{\mathrm{N}_G(Q)}( \mathrm{Br}_Q^H(M)) \cong \Br_Q^G( \mathrm{Ind}_H^G(M))$$
of $k \mathrm{N}_G(Q)$-modules.
\end{lemma}

The following remark is a variant of \cite[Corollary 3.9]{Marcus}.

\begin{remark}\label{Brauercomm}
As in Section \ref{Marcus} we let $\tilde{L}$ be a subgroup of a finite group $\tilde{G}$ and $G$ be a normal subgroup of $\tilde{G}$. We set $L := G \cap \tilde{L}$ and we assume additionally that $ \tilde{G}= \tilde{L} G$.
Let $Q$ be an $\ell$-subgroup of $\tilde{L}$. In the following diagram, $\mathrm{Ind}$ and $\mathrm{Res}$ mean induction and restriction with respect to the subgroups of $\tilde{G} \times \tilde{L}^{\mathrm{opp}}$ involved.

\ \\
\ \\
\ \\

\hspace{6.5cm}
\begin{tikzpicture}[thick, scale = 0.8, transform canvas={scale=0.8}]

  \matrix (m) [matrix of math nodes,row sep=3em,column sep=2em,minimum width=3em] {
  k [ \tilde{G} \times \tilde{L}^{\mathrm{opp}} ]\text{-} \mathrm{perm} & k[ \mathrm{N}_{ \tilde{G} \times \tilde{L}^{\mathrm{opp}}} (\Delta Q)]\text{-} \mathrm{perm} & k[\mathrm{N}_{\tilde{G}}(Q) \times \mathrm{N}_{\tilde{L}}(Q)^{\mathrm{opp}} ]\text{-} \mathrm{perm} \\
  k[ G \times L^{\mathrm{opp}} \Delta(\tilde{L}) ]\text{-} \mathrm{perm} & k[\mathrm{N}_{G \times L^{\mathrm{opp}} \Delta(\tilde{L})}(\Delta Q)]\text{-} \mathrm{perm} & k[\mathrm{N}_{G}(Q) \times \mathrm{N}_{L}(Q)^{\mathrm{opp}} \Delta( \mathrm{N}_{\tilde{L}}(Q)) ] \text{-} \mathrm{perm} \\
  k[ G \times L^{\mathrm{opp}} ]\text{-} \mathrm{perm}  & k[\mathrm{N}_{G \times L^{\mathrm{opp}}}(\Delta Q)]\text{-} \mathrm{perm} & k[\mathrm{N}_{G}(Q) \times \mathrm{N}_L(Q)^{\mathrm{opp}}]\text{-} \mathrm{perm} \\};
\path[-stealth]
(m-1-1) edge node [above] {$\mathrm{Br}_{\Delta(Q)}$} (m-1-2)
(m-1-2) edge node [above] {$\mathrm{Ind}$} (m-1-3)

(m-2-1) edge node [above] {$\mathrm{Br}_{\Delta(Q)}$} (m-2-2)
(m-2-2) edge node [above] {$\mathrm{Ind}$} (m-2-3)

(m-2-1) edge node [left] {$\mathrm{Ind}$}  (m-1-1)
(m-2-2) edge node [left] {$\mathrm{Ind}$}  (m-1-2)
(m-2-3) edge node [left] {$\mathrm{Ind}$}  (m-1-3)

(m-3-1) edge node [above] {$\mathrm{Br}_{\Delta(Q)}$} (m-3-2)
(m-3-2) edge node [above] {$\mathrm{Ind}$} (m-3-3)

(m-2-1) edge node [left] {$\mathrm{Res}$} (m-3-1)
(m-2-2) edge node [left] {$\mathrm{Res}$} (m-3-2)
(m-2-3) edge node [left] {$\mathrm{Res}$} (m-3-3);

\end{tikzpicture}

\ \\
\ \\

\ \\

We claim that the upper left square commutes for all relatively $\Delta \tilde{L}$-projective $\ell$-permutation $k[ G \times L^{\mathrm{opp}} \Delta(\tilde{L})]$-modules. In view of Lemma \ref{BrauerMarcus} it is sufficient to show that
$$\mathrm{C}_{\tilde{G} \times \tilde{L}^{\mathrm{opp}}}(\Delta Q) \, T_{G \times L^{\mathrm{opp}} \Delta \tilde{L}}(\Delta Q,\Delta R) = T_{\tilde{G} \times \tilde{L}^{\mathrm{opp}}}(\Delta Q, \Delta R)$$
for all $\ell$-subgroups $R$ of $\tilde{L}$ containing $Q$. This is proved as in  \cite[Corollary 3.9]{Marcus}. The upper right and the bottom left square are clearly commutative. Moreover, the commutativity of the bottom right square is a consequence of Mackey's formula.

\end{remark}

\subsection{The Harris--Knörr correspondence}

In this section we recall the notion of block induction as given in \cite[Theorem 4.14]{NavarroBook}. This will allow us to give a nice formulation of the important Harris--Knörr correspondence.
%


\begin{definition}
Suppose that $H$ is a subgroup of $G$ and $b$ is a block idempotent of $G$. Furthermore, assume that there exists an $\ell$-subgroup $P$ of $G$ such that $P \mathrm{C}_G(P) \subseteq H \subseteq \mathrm{N}_G(P)$. Then we say that the block idempotent $c\in \mathrm{Z}(\mathrm{O} H)$ \textit{induces to} $b$ if $\mathrm{br}_P(b) c \neq 0$. In this case we write $b=c^G$.
\end{definition}

Note that the definition of block induction in \cite[page 87]{NavarroBook} is more general. However, we will not need this general definition and have therefore decided to use the characterisation of block induction in \cite[Theorem 4.14]{NavarroBook} as a definition. Recall that for a subgroup $Q$ of the defect group $D$ of $b$ we denote $B_Q:= \mathrm{Tr}^{\mathrm{N}_G(Q)}_{\mathrm{N}_G(Q,b_Q)}(b_Q)$, which is a block of $\mathrm{N}_G(Q)$.

\begin{theorem}[Harris--Kn\"orr]\label{HarrisKnorr}
Let $G$ be a normal subgroup of a finite group $\tilde{G}$. Let $b$ be a block of $G$ with defect group $D$ and denote by $B_D$ its Brauer correspondent in $k \mathrm{N}_G(D)$. Then the map 
$$\mathrm{Bl}(\mathrm{N}_{\tilde{G}}(D) \mid B_D) \to \mathrm{Bl}(\tilde{G} \mid b), \, c \mapsto c^{\tilde{G}}$$
is a bijection.
\end{theorem}

\begin{proof}
See \cite[Theorem 9.28]{NavarroBook}.
\end{proof}

%

If $Q$ is a characteristic subgroup of the defect group $D$ of $b$ we have $\mathrm{N}_G(D)\subseteq \mathrm{N}_G(Q)$. Brauer correspondence therefore yields a bijection 
$$\mathrm{br}_D:\mathrm{Bl}(\mathrm{N}_G(Q) \mid D) \to \mathrm{Bl}(\mathrm{N}_G(D) \mid D).$$
After having established this notation we can now state the following lemma:

\begin{lemma}\label{Brauercorrespondence}
Let $Q$ be a characteristic subgroup of $D$. Then $B_Q \in \mathrm{Bl}(\mathrm{N}_G(Q))$ is the Brauer correspondent of $B_D \in  \mathrm{Bl}(\mathrm{N}_G(D))$.
\end{lemma}

\begin{proof}
By \cite[Theorem 40.4(b)]{Jacques} we have $\mathrm{br}_D(b_Q) =b_D$. Since $D \subseteq \mathrm{N}_G(Q)$ we can write $B_Q \in \mathrm{Z}(k\mathrm{C}_G(Q))$ as a sum $B_Q= \sum_{i=1}^s c_i$ of block idempotents of $k\mathrm{C}_G(Q)$. Note that each $c_i$ is a sum of idempotents which constitute a $D$-orbit on $\{{}^t b_Q \mid t \in \mathrm{N}_G(Q) \}$.

Assume first that $c_i$ comes from a $D$-orbit of length greater than $1$. Let $ t\in \mathrm{N}_G(Q)$ with ${}^t b_Q c_i \neq 0$. Then the block $c_i$ covers ${}^t b_Q$ and it follows that any defect group of $c_i$ is contained in $\mathrm{N}_G(Q,{}^t b_Q )$. Since ${}^t b_Q $ is not $D$-stable it follows that $D$ is not contained in $\mathrm{N}_G(Q,{}^t b_Q)$. Thus $D$ is not contained in a defect group of $c_i$. This implies that $\mathrm{br}_D(c_i)=0$.

On the other hand, if $c_i={}^t b_Q$ for some $t \in \mathrm{N}_G(Q)$ it follows that ${}^t b_Q$ is $D$-stable. Assume that $\mathrm{br}_D({}^t b_Q) \neq 0$. Then we have ${}^t (Q,b_Q) \unlhd (D,b_D')$ for some maximal $b$-subpair $(D,b_D')$. Since also ${}^t (Q,b_Q) \unlhd {}^t (D,b_D)$ it follows by \cite[Proposition 40.15(b)]{Jacques} that there exists some $x \in \mathrm{N}_G(Q,b_Q)$ such that $tx \in \mathrm{N}_G(D)$. From this we conclude that $$\br_D({}^t b_Q)= \br_D({}^{tx} b_Q)={}^{tx} \br_D(b_Q)={}^{tx} b_D.$$
These calculations show that $\mathrm{br}_D(B_Q) B_D=B_D$. On the other hand $B_Q$ is an idempotent occurring in $\mathrm{br}_Q(b)$ and we have $\mathrm{br}_D(\mathrm{br}_Q(b))=B_D$. Writing $\mathrm{br}_Q(b)=B_Q + C$ we obtain $B_D= \br_D(B_Q) + \br_D(C)$, a sum of orthogonal idempotents. Now observe that $B_D$ is a primitive central idempotent of $\mathrm{N}_{G}(D)$ and $\mathrm{br}_D(B_Q) B_D=B_D$. Therefore, $B_D= \mathrm{br}_D(B_Q)$.
\end{proof}

We obtain a version of the Harris--Knörr theorem for characteristic subgroups of defect groups. 

\begin{corollary}\label{HKCoro}
With the notation of Theorem \ref{HarrisKnorr} assume that $Q$ is a characteristic subgroup of $D$. Let $(Q,b_Q)$ be a $b$-Brauer pair with $(Q,b_Q) \leq (D,b_D)$. Then block induction yields a bijection
$$ \mathrm{Bl}(\mathrm{N}_{\tilde{G}}(Q) \mid B_Q) \to \mathrm{Bl}(\tilde{G} \mid b), c \mapsto c^{\tilde{G}}.$$
\end{corollary}

\begin{proof}

Brauer correspondence gives a bijection $ \mathrm{br}_D:\mathrm{Bl}(G \mid D) \to \mathrm{Bl}(\mathrm{N}_G(D) \mid D)$ with $\mathrm{br}_D(b)=B_D$. Moreover, by Lemma \ref{Brauercorrespondence} the map $\mathrm{br}_D:\mathrm{Bl}(\mathrm{N}_G(Q) \mid D) \to \mathrm{Bl}(\mathrm{N}_G(D) \mid D)$ is a bijection with $\mathrm{br}_D(B_Q)=B_D$.


By Theorem \ref{HarrisKnorr} we have bijections $ \mathrm{Bl}(\mathrm{N}_{\tilde{G}}(D) \mid B_D)  \to \mathrm{Bl}(\tilde{G} \mid b)$ and $ \mathrm{Bl}(\mathrm{N}_{\tilde{G}}(D) \mid B_D)  \to \mathrm{Bl}(\mathrm{N}_{\tilde{G}}(Q) \mid B_Q)$ both given by block induction. This yields a bijection
$$\mathrm{Bl}(\mathrm{N}_{\tilde{G}}(Q) \mid B_Q) \to \mathrm{Bl}(\tilde{G} \mid b ).$$
Moreover, if $c \in \mathrm{Bl}(\mathrm{N}_{\tilde{G}}(D) \mid B_D)$ then $c^{\mathrm{N}_{\tilde{G}}(Q)}$ and $c^{\tilde{G}}$ are both defined. By \cite[Problem 4.2]{NavarroBook} it follows that $c^{\tilde{G}}=(c^{\mathrm{N}_{\tilde{G}}(Q)})^{\tilde{G}}$. Hence, the bijection $\mathrm{Bl}(\mathrm{N}_{\tilde{G}}(Q) \mid B_Q) \to \mathrm{Bl}(\tilde{G} \mid b)$ is given by block induction.
\end{proof}

%
%
%

\subsection{Splendid Rickard equivalences and Clifford theory}\label{RickardClifford}

In Proposition \ref{normalizerderived} we have shown that a splendid Rickard equivalence induces a derived equivalence on the level of normalizers. Therefore, a natural question to ask is whether the so-obtained equivalences behave nicely with respect to the Clifford theory of Rickard equivalences and with the Brauer category of the involved blocks. These questions will be addressed in this section. We first make the following useful observation.

\begin{lemma}\label{stabstab}
Let $G$ be a normal subgroup of a finite group $\tilde{G}$. Let $b$ be a $\tilde{G}$-stable block of $G$ with defect group $D$ and $Q$ a characteristic subgroup of $D$. Then $B_Q$ is an $\mathrm{N}_{\tilde{G}}(Q)$-stable block of $\mathrm{N}_G(Q)$ and we have $\mathrm{N}_{\tilde{G}}(Q) / \mathrm{N}_G(Q) \cong \tilde{G} / G$. 
\end{lemma}

\begin{proof}
	Recall that all defect groups of $b$ are $G$-conjugate. Since $b$ is a $\tilde{G}$-stable block of $G$ we thus obtain $\tilde{G}= G \mathrm{N}_{\tilde{G}}(D)$. Moreover, $Q$ is a characteristic subgroup of $D$ and so $ \mathrm{N}_{\tilde{G}}(D) \subseteq \mathrm{N}_{\tilde{G}}(Q)$. From this we conclude that $\tilde{G} / G \cong \mathrm{N}_{\tilde{G}}(Q) / \mathrm{N}_{G}(Q)$. It remains to show that $B_Q$ is $\mathrm{N}_{\tilde{G}}(Q)$-stable. If $g \in \mathrm{N}_{\tilde{G}}(Q)$ then ${}^g (D,b_D)$ is a second maximal $b$-Brauer pair, so there exists some $x \in G$ with ${}^{gx} (D,b_D)=(D,b_D)$. In particular, $gx \in \mathrm{N}_{\tilde{G}}(D) \subseteq \mathrm{N}_{\tilde{G}}(Q)$ and thus $x \in \mathrm{N}_G(Q)$. Moreover, $(Q,b_Q) \leq (D,b_D)$ and $(Q,{}^{gx} b_Q)={}^{gx} (Q,b_Q) \leq (D,b_D)$ are two $b$-Brauer pairs with first entry $Q$. Therefore, $gx \in \mathrm{N}_{\tilde{G}}(Q,b_Q)$ and so ${}^g B_Q={}^{gx} B_Q=B_Q$. 
\end{proof}

In the following, $\tilde{L}$ denotes a subgroup of a finite group $\tilde{G}$ and $G$ a normal subgroup of $\tilde{G}$. We set $L := G \cap \tilde{L}$ and assume that $\tilde{G} =\tilde{L} G $. As before we set $\mathcal{D}:= ( G \times L^{\opp}) \Delta(\tilde{L})$. Furthermore, let $c \in \mathrm{Z}(k L)$ be a $\tilde{L}$-stable block of $L$ and $b \in \mathrm{Z}(k G)$.

\begin{lemma}\label{commutativelocal}
Let $C$ be a bounded complex of $kG b$-$k L c$-bimodules inducing a splendid Rickard equivalence between the blocks $k G b$ and $k L c$. Assume that $C$ extends to a complex $C'$ of $k \mathcal{D}$-modules and denote $\tilde{C}:= \mathrm{Ind}_{\mathcal{D}}^{\tilde{G} \times \tilde{L}^{\mathrm{opp}}} (C')$.
%
Let $D$ be a defect group of $k Lc $ and $Q$ a characteristic subgroup of $D$. Let $(Q,c_Q)$ be a $c$-Brauer pair corresponding to the $b$-Brauer pair $(Q,b_Q)$ as in Proposition \ref{splendidloc}. Set
$$\tilde{\mathcal{C}}:=\mathrm{Ind}_{\mathrm{N}_{\tilde{G} \times \tilde{L}^{\mathrm{opp}}}( \Delta Q)}^{\mathrm{N}_{\tilde{G}}(Q) \times \mathrm{N}_{\tilde{L}}(  Q)^{\mathrm{opp}}}( \Br_{\Delta Q}(\tilde{C}) ) C_Q \text{ and } \mathcal{C}:=\mathrm{Ind}_{\mathrm{N}_{G \times L^{\mathrm{opp}}}( \Delta Q)}^{\mathrm{N}_{G}(Q) \times \mathrm{N}_{L}( Q)^{\mathrm{opp}}} (\Br_{\Delta Q}(C)) C_Q.$$
Then the following diagram is commutative:
\begin{center}
\begin{tikzpicture}
  \matrix (m) [matrix of math nodes,row sep=4em,column sep=6em,minimum width=2em] {
    \mathrm{D}^b(k \mathrm{N}_{\tilde{L}}(Q) C_Q) &  \mathrm{D}^b(k \mathrm{N}_{\tilde{G}}(Q) B_Q ) \\
    \mathrm{D}^b(k \mathrm{N}_{L}(Q) C_Q)  &  \mathrm{D}^b(k \mathrm{N}_{G}(Q) B_Q)  \\};
\path[-stealth]
(m-1-2) edge node [left] {$\mathrm{Res}_{\mathrm{N}_G(Q)}^{\mathrm{N}_{\tilde{G}}(Q)}$} (m-2-2)
(m-1-1) edge node [above] {$\tilde{\mathcal{C}} \otimes^{\mathbb{L}}_{k \mathrm{N}_{\tilde{L}}(Q)} -$} (m-1-2)
(m-2-1) edge node [above] {$\mathcal{C} \otimes^{\mathbb{L}}_{k \mathrm{N}_{L}(Q)} -$} (m-2-2)
(m-1-1) edge node [left] {$\mathrm{Res}_{\mathrm{N}_L(Q)}^{\mathrm{N}_{\tilde{L}}(Q)}$} (m-2-1);

\end{tikzpicture}
\end{center}
where the horizontal maps induce equivalences of the derived categories.
\end{lemma}

\begin{proof}

By the commutativity of the first two rows of the commutative diagram in Remark \ref{Brauercomm} we have a natural isomorphism
$$\tilde{\mathcal{C}} \cong \mathrm{Ind}^{\mathrm{N}_{\tilde{G}}(Q) \times \mathrm{N}_{\tilde{L}}(Q)^{\mathrm{opp}}}_{\mathrm{N}_{G}(Q) \times \mathrm{N}_{L}(Q)^{\mathrm{opp}} \Delta \mathrm{N}_{\tilde{L}}(Q)} (\mathcal{C'}),$$
where $\mathcal{C'}:= \mathrm{Ind}^{\mathrm{N}_{G}(Q) \times \mathrm{N}_{L}(Q)^{\mathrm{opp}} \Delta \mathrm{N}_{\tilde{L}}(Q)}_{\mathrm{N}_{ G \times L^{\mathrm{opp}} \Delta \tilde{L}}(\Delta Q)}(\Br_{\Delta Q}(C')) C_Q$.
Now by the commutativity of the second and the third row of the commutative diagram in Remark \ref{Brauercomm} we deduce that $$\mathrm{Res}^{\mathrm{N}_{G}(Q) \times \mathrm{N}_{L}(Q)^{\mathrm{opp}} \Delta \mathrm{N}_{\tilde{L}}(Q)}_{ \mathrm{N}_G(Q) \times \mathrm{N}_L(Q)^{\mathrm{opp}}}( \mathcal{C'}) \cong \mathcal{C}.$$

By Proposition \ref{normalizerderived} the complex $\mathcal{C}$
induces a derived equivalence between the blocks $k \mathrm{N}_{G}(Q) B_Q$ and $k \mathrm{N}_L(Q) C_Q$. By Lemma \ref{stabstab}, the block $B_Q$ is $\mathrm{N}_{\tilde{G}}(Q)$-stable and $C_Q$ is $\mathrm{N}_{\tilde{L}}(Q)$-stable. Moreover, we have
$$\mathrm{N}_{\tilde{G}}(Q) / \mathrm{N}_{G}(Q) \cong \mathrm{N}_{\tilde{L}}(Q) / \mathrm{N}_{L}(Q).$$
It follows from Theorem \ref{lifting} that the complex $\tilde{\mathcal{C}} \cong \mathrm{Ind}^{\mathrm{N}_{\tilde{G}}(Q) \times \mathrm{N}_{\tilde{L}}(Q)^{\mathrm{opp}}}_{\mathrm{N}_{G}(Q) \times \mathrm{N}_{L}(Q)^{\mathrm{opp}} \Delta \mathrm{N}_{\tilde{L}}(Q)} (\mathcal{C'})$ induces a derived equivalence between $k \mathrm{N}_{\tilde{L}}(Q) C_Q$ and $k \mathrm{N}_{\tilde{G}}(Q) B_Q$. The commutativity of the diagram is now a consequence of Remark \ref{characters}(a).
\end{proof}

In Corollary \ref{HKCoro} we have established a Harris--Knörr correspondence for characteristic subgroups of the defect group of a block. It is therefore natural to ask whether the construction in Lemma \ref{commutativelocal} is compatible with this correspondence.

\begin{remark}\label{HKcompatible}
Assume that we are in the situation of Lemma \ref{commutativelocal}. Let $c=c_1+ \dots + c_r$ be a decomposition of $c$ into block idempotents of $k \tilde{L}$. We let $b=b_1+\dots +b_r$ be the decomposition of $b$ into block idempotents of $k \tilde{G}$ such that $b_i \tilde{C} c_i \neq 0$ in $\mathrm{Ho}^b( k [\tilde{G} \times \tilde{L}^{\mathrm{opp}} ])$. Denote by $B_{Q,i}:= \br_Q(b_i)B_Q$ the Harris--Knörr correspondent of $b_i$, see Corollary \ref{HKCoro}. We deduce that 
$B_Q=B_{Q,1} + \dots + B_{Q,r}$ is a decomposition into block idempotents of $k \mathrm{N}_{\tilde{G}}(Q)$. Similarly, we have a decomposition
$C_Q=C_{Q,1} + \dots + C_{Q,r}$
into block idempotents of $k \mathrm{N}_{\tilde{L}}(Q)$, where $C_{Q,i}:= \br_Q(c_i)C_Q$. We obtain
$$\tilde{\mathcal{C}}=\mathrm{Ind}_{\mathrm{N}_{{\tilde{G}} \times {\tilde{L}}^{\mathrm{opp}}}(\Delta Q)}^{\mathrm{N}_{\tilde{G}}(Q) \times \mathrm{N}_{\tilde{L}}( Q)^{\mathrm{opp}}} (\Br_{\Delta Q}(\tilde{C}) C_Q) \cong \displaystyle \bigoplus_{i=1}^r \tilde{\mathcal{C}}  C_{Q,i}.$$
From this we conclude that the complex $\tilde{\mathcal{C}} C_{Q,i}$ induces a derived equivalence between the blocks $k \mathrm{N}_{\tilde{G}}(Q) B_{Q,i}$ and $k \mathrm{N}_{\tilde{L}}(Q) C_{Q,i}$. Thus, the local equivalences for the normalizer are compatible with the Harris--Knörr correspondence. 

\end{remark}

%


\section{Deligne--Lusztig theory and disconnected reductive groups}

In this section we recall the necessary background in the representation theory of finite groups of Lie type. We will in particular discuss extensions of this theory to disconnected reductive groups. Then we will recall the Morita equivalence constructed by Bonnafé, Dat and Rouquier which can be seen as a starting point of this work. 
 
%
%
%
%

\subsection{Disconnected reductive algebraic groups}


Fix a prime number $p$ and an algebraic closure $\overline{\mathbb{F}_p}$ of $\mathbb{F}_p$. Let $\G$ denote a (not necessarily connected) reductive algebraic group defined over $\overline{\mathbb{F}_p}$. We denote by $\G^\circ$ the connected component of $\G$ containing the identity.

In the following, we recall some standard facts, which can for instance be found in \cite[Section 2.D.]{Dat} and \cite[Section 3.A.]{Dat}. A closed subgroup $\Para$ of $\G$ is called \textit{parabolic subgroup} if the variety $\G/\Para$ is complete. One can show that a closed subgroup $\Para$ of $\G$ is a parabolic subgroup of $\G$ if and only if $\Para^\circ$ is a parabolic subgroup of $\G^\circ$. Moreover, we have $\Para \cap \G^\circ= \Para^\circ$ and the unipotent radicals of $\Para$ and $\Para^\circ$ coincide.

Suppose that $\Para$ is a parabolic subgroup of $\G$. Let $\Levi_\circ$ be a Levi subgroup of $\G^\circ$  so that $\Para^\circ= \Levi_\circ \ltimes \U$ is a Levi decomposition of the parabolic subgroup $\Para^\circ$ in $\G^\circ$. Then we call $\Levi=\mathrm{N}_\Para(\Levi_\circ)$ \textit{a Levi subgroup} of $\Para$ in $\G$. In addition, we have a decomposition $\Para= \Levi \ltimes \U$ and $\Levi_\circ$ is the connected component of $\Levi$, i.e. $\Levi^\circ=\Levi_\circ$.

\begin{example}\label{dm}
Let $\G$ be a reductive algebraic group. Let $\Para_\circ=\Levi_\circ \ltimes \U$ be a parabolic subgroup with Levi decomposition in $\G^\circ$. Then $\Para=\mathrm{N}_{\G}(\Para_\circ)$ is a parabolic subgroup of $\G$ with Levi subgroup $\Levi=\mathrm{N}_\G(\Levi_\circ,\Para_\circ)=\mathrm{N}_\Para(\Levi_\circ)$ such that $ \Para^\circ=\Para_\circ$. 
\end{example}

As we show in the next example, disconnected reductive groups arise naturally in the study of automorphisms of reductive groups. 

\begin{example}\label{graph}
Let $\G_\circ$ be a connected reductive group and $\tau: \G_\circ \to \G_\circ$ an algebraic automorphism of $\G_\circ$ of finite order. Then the semidirect product $\G:=\G_\circ  \rtimes \langle \tau \rangle$ is again a reductive algebraic group but no longer connected. This situation was for instance considered in \cite{Malle}. Let $\Para_\circ= \Levi_\circ \ltimes \U$ be a Levi decomposition of a parabolic subgroup $\Para_\circ$ of $\G_\circ$. If $\Para_\circ$ is $\tau$-stable, then $\Levi_\circ$ is $\tau$-stable as well. In particular, $\Para:=\Para_\circ \rtimes \langle \tau \rangle$ is a parabolic subgroup of $\G$ with Levi subgroup $\Levi:= \Levi_\circ \rtimes \langle \tau \rangle $, see Example \ref{dm}. 
\end{example}

Disconnected reductive groups also appear naturally as local subgroups of (connected) reductive groups.

\begin{example}\label{normcent}
Let $\G$ be a possibly disconnected reductive group, $\Para$ a parabolic subgroup of $\G$ with Levi decomposition $\Para=\Levi \ltimes \U$. In addition, we assume that $Q$ is a finite solvable $p'$-subgroup of $\Levi$. By \cite[Remark 3.5]{Dat} it follows that the normalizer $\mathrm{N}_\G(Q)$ is a reductive group. Moreover, $\mathrm{N}_\Para(Q)$ is a parabolic subgroup of $\mathrm{N}_\G(Q)$ with Levi decomposition $\mathrm{N}_\Para(Q)=\mathrm{N}_{\Levi}(Q) \ltimes \C_\U(Q)$. Similarly, $\C_\G(Q)$ is a reductive group with parabolic subgroup $\C_\Para(Q)$ and Levi decomposition $\mathrm{C}_\Para(Q)=\mathrm{C}_{\Levi}(Q) \ltimes \C_\U(Q)$, see \cite[Proposition 3.4]{Dat}. Note that $\mathrm{N}_\G(Q)/\C_\G(Q)$ is finite.
Therefore, $\mathrm{N}^\circ_\G(Q)=\C^\circ_\G(Q)$ and we have a Levi decomposition $\C^\circ_\Para(Q)=\C^\circ_\Levi(Q) \ltimes  \C_\U(Q)$ in the connected reductive group $\C^\circ_\G(Q)$.
\end{example}

\subsection{$\ell$-adic cohomology of Deligne--Lusztig varieties}\label{DLvarieties}

From now on $\ell$ denotes a prime number with $p \neq \ell$ and $q$ is an integral power of $p$. By variety we always mean a quasi-projective variety defined over $\overline{\mathbb{F}_p}$. Let $\X$ be a variety acted on by a finite group $G$. We denote by $R \Gamma_c(\X, \mathcal{O}) \in \mathrm{D}^b(\mathcal{O} G)$ the $\ell$-adic cohomology with compact support of the variety $\X$ with coefficients in $\mathcal{O}$, see \cite[A.3.7]{MarcBook} and \cite[A.3.14]{MarcBook}. For $A \in \{ K,\mathcal{O} ,k \}$ we define
$$R \Gamma_c(\X, A):=R \Gamma_c(\X,\mathcal{O}) \otimes^{\mathbb{L}}_\mathcal{O} A \in \mathrm{D}^b(A G).$$
Moreover, we denote by $H^d_c(\X,A) \in A G \text{-} \mathrm{mod}$ the $d$th cohomology module of the complex $R \Gamma_c(\X,A)$.

Let $\G$ be a reductive group with Frobenius endomorphism $F: \G \to \G$ defining an $\mathbb{F}_q$-structure on $\G$. Let $\Para$ be a parabolic subgroup of $\G$, $\Para=\Levi \U$ be a Levi decomposition and assume that $\Levi$ is $F$-stable. Consider the $\G^F$-$\Levi^F$-variety \index{Deligne--Lusztig variety}
$$\Y^\G_\U:=\{g\U \in \G/\U \mid g^{-1}F(g) \in \U F(\U) \} \subseteq \G/ \U.$$
If the ambient group $\G$ is clear from the context we will just write $\Y_\U$ instead of $\Y_\U^\G$. The cohomology of this variety provides us with a triangulated functor
$$\mathcal{R}_{\Levi\subseteq \Para}^\G: \mathrm{D}^b(\Lambda \Levi^F) \to \mathrm{D}^b(\Lambda \G^F), \, M \mapsto R \Gamma_c(\Y^\G_\U,\Lambda) \otimes^{\mathbb{L}}_{\Lambda \Levi^F} M.$$
This functor induces a map
$$R_{\Levi\subseteq \Para}^\G:=[\mathcal{R}_\Levi^\G]: G_0(\Lambda \Levi^F) \to G_0(\Lambda \G^F), \, [M] \mapsto \sum_{i}(-1)^i [H^i_c(\Y^\G_\U,\Lambda) \otimes_{\Lambda \Levi^F} M],$$
on Grothendieck groups (see Section \ref{Grothendieck group}) the so-called \textit{Lusztig induction}.

\subsection{Properties of Deligne--Lusztig varieties}\label{PropertiesofDL}

In this section we will study the following set-up: Let $\hat{\G}$ be a reductive group with Frobenius $F: \hat{\G} \to \hat{\G}$. Moreover, assume that $\G$ is a closed $F$-stable normal subgroup of $\hat{\G}$. Suppose that $\Para= \Levi \U$ and $\hat{\Para}=\hat{\Levi} \U$ are two Levi decomposition of parabolic subgroups $\Para$ of $\G$ and $\hat{\Para}$ of $\hat{\G}$ such that $\hat{\Para} \cap \G= \Para$ and $\hat{\Levi} \cap \G= \Levi$. Assume that the Levi subgroup $\hat{\Levi}$ is $F$-stable. Let us denote 
$$\mathcal{D}=\{(x,y)\in \hat{\G}^F \times (\hat{\Levi}^F)^{\operatorname{opp}} \mid x \G^F=y^{-1} \G^F \}= (\G^F \times ({\Levi}^F)^{\opp}) \Delta(\hat{\Levi}^F) .$$

\begin{lemma}\label{extend}
With the notation as above, the variety $\Y_\U^{\G}$ is a $\mathcal{D}$-stable subvariety of $\hat{\G}/\U$.
\end{lemma}

\begin{proof}
Let $(x,y)\in \mathcal{D}$ and $g \U \in \G / \U$. Since $(x^{-1},y^{-1}) \in \mathcal{D}$ we have $xy\in \G^F$ and $xgx^{-1} \in \G$. Since $\hat{\Levi}$ normalizes $\U$ we conclude that
$$x g \U y=xgy \U=x g x^{-1} xy \U \in \G / \U. $$
Hence, the group action of $\mathcal{D}$ stabilizes the subvariety $\G/ \U$ of $\hat{\G} / \U$.

Now suppose that $ g \U \in \Y_{\U}^{\G}$. Let us define $c=xgy$. It follows that 
$$c^{-1} F(c)=y^{-1} g^{-1} F(g) y \in (\U F( \U))^{y}=\U F(\U)$$
since $y\in \hat{\Levi}$ normalizes $\U$. Consequently, the Deligne--Lusztig variety $\Y_\U^{\G}$ is a $\mathcal{D}$-stable subvariety of $\hat{\G}/\U$.
\end{proof}

We also consider the generalized Deligne--Lusztig varieties as introduced in \cite[Section 6A]{Dat}. Let $\Para_1$ and $\Para_2$ be two parabolic subgroups of $\G$ with common $F$-stable Levi complement $\Levi$ and unipotent radicals $\U_1$ and $\U_2$ respectively. We define
$$\Y^{\G}_{\U_1,\U_2}= \{ (g_1 \U_1,g_2 \U_2) \in \G/ \U_1 \times \G/\U_2 \mid g_1^{-1} g_2 \in \U_1 \U_2; g_2^{-1} F(g_1) \in \U_2 F(\U_1) \}$$ 
which is a variety acted on diagonally by $\hat{\G}^F \times (\hat{\Levi}^F)^{\operatorname{opp}}$. Similarily to Lemma \ref{extend} one proves that  $\Y^{\G}_{\U_1,\U_2}$ is a $\mathcal{D}$-stable subvariety of $\hat{\G}/ \mathbf{U}_1 \times \hat{\G} / \mathbf{U}_2$.

%
%
\begin{notation}
Let $H$ be a finite group. If $\X$ is a right $H$-variety and $\Y$ a left $H$-variety we denote by $\X \times_H \Y$ the quotient of $\X \times \Y$ by the diagonal right action of the group $\Delta(H)=\{(h,h^{-1}) \mid h\in H\}$ given by
$$\X\times \Y \to \X \times \Y, (x,y) \mapsto (xh,h^{-1}y).$$
\end{notation}
%


Now assume that $\X$ is a $G$-$H$-variety and $\Y$ an $H$-$L$-variety. Then $\X\times_H \Y$ becomes a $G$-$L$-variety. To compute the cohomology of this new variety one uses the following theorem:

\begin{theorem}[Künneth formula]\label{Kuenneth}\index{Künneth formula}
If the stabilizers of points of $\X \times \Y$ under the diagonal action of $H$ are of invertible order in $\Lambda$, then we have
$$R\Gamma_c(\X,\Lambda) \otimes_{\Lambda H}^{\mathbb{L}} R\Gamma_c(\Y,\Lambda) \cong R \Gamma_c(\X \times_H \Y,\Lambda)$$
in $\mathrm{D}^b(\Lambda[G \times L^{\mathrm{opp}}])$.
\end{theorem}

\begin{proof}
See \cite[Section 3.3]{BoRo}.
\end{proof}

The following geometric lemma describes two closely related decompositions of the Deligne--Lusztig variety $\Y_\U^{\hat{\G}}$. The following result is certainly well known, but it does not appear in this exact form in the literature, see also \cite[Theorem 7.3]{MarcBook} or \cite[Proposition 1.1]{Godement}. A complete proof can be found in \cite[Lemma 2.8]{PhD}.

\begin{lemma}\label{geometricgeneralversion2}
We have the following two decompositions of $\Y^{\hat{\G}}_\U$ as a $({\hat{\G}}^F\times ({\hat{\Levi}}^F)^{\operatorname{opp}})$-variety.
\begin{enumerate}[label={\alph*)}]
\item $\Y^{\hat{\G}}_\U=\coprod_{g\in \hat{\G}^F / \G^F} g \Y_{\U}^{\G}={\hat{\G}}^F \times_{\G^F} \Y_\U^{\G}$,
\item $\Y^{\hat{\G}}_\U \cong (\hat{\G}^F \times (\hat{\Levi}^F)^{\operatorname{opp}}) \times_{\mathcal{D}} \Y_\U^{\G}$.

\end{enumerate}
\end{lemma}

\begin{corollary}\label{Kunnethcoro}
Under the assumption of Lemma \ref{geometricgeneralversion2} we have
$$R\Gamma_c(\Y_\U^{\hat{\G}},\Lambda) \cong \Lambda[\hat{\G}^F\times (\hat{\Levi}^F)^{\operatorname{opp}}] \otimes_{\Lambda \mathcal{D}}^{\mathbb{L}} R\Gamma_c(\Y_\U^{\G},\Lambda)$$
in $\mathrm{D}^b(\Lambda[\hat{\G}^F \times (\hat{\Levi}^F)^{\operatorname{opp}}])$.
\end{corollary}

\begin{proof}
By Lemma \ref{geometricgeneralversion2} we have $\Y^{\hat{\G}}_\U \cong (\hat{\G}^F\times (\hat{\Levi}^F)^{\operatorname{opp}}) \times_{\mathcal{D}} \Y_\U^{\G}$ as $\hat{\G}^F\times (\hat{\Levi}^F)^{\operatorname{opp}}$-varieties. The group $\mathcal{D}$ acts freely by right multiplication on $\hat{\G}^F\times (\hat{\Levi}^F)^{\operatorname{opp}}$. Hence, it follows that $\mathcal{D}$ acts freely on $(\hat{\G}^F\times (\hat{\Levi}^F)^{\operatorname{opp}}) \times \Y_\U^{\G}$. Thus, Theorem \ref{Kuenneth} is applicable and we obtain 
$$R\Gamma_c(\Y_\U^{\hat{\G}},\Lambda) \cong \Lambda[\hat{\G}^F\times (\hat{\Levi}^F)^{\operatorname{opp}}] \otimes_{\Lambda \mathcal{D}}^{\mathbb{L}} R\Gamma_c(\Y_\U^{\G},\Lambda)$$
in $\mathrm{D}^b(\Lambda[\hat{\G}^F \times (\hat{\Levi}^F)^{\operatorname{opp}}])$.
\end{proof}

\subsection{Godement resolutions}\label{Godement}

Let $\X$ be a variety defined over an algebraic closure of $\mathbb{F}_p$ endowed with an action of a finite group $G$. By work of Rickard and Rouquier there exists an object $G\Gamma_c(\X,\Lambda)$ in $\mathrm{Ho}^b(\Lambda G\text{-} \mathrm{perm})$ which is a representative of $R\Gamma_c(\X,\Lambda) \in \mathrm{D}^b(\Lambda G)$, see \cite{RickardEtale} and \cite[Section 2]{Rouquier}. 
%
The advantage of the Rickard--Rouquier complex $G\Gamma_c(\X,\Lambda)$ is that it is a complex of $\ell$-permutation modules which is compatible with the Brauer functor. More precisely, if $Q$ is an $\ell$-subgroup of $G$ then we have a canonical isomorphism
$$\mathrm{Br}_Q(G \Gamma_c(\X,\Lambda)) \cong G \Gamma_c(\X^Q,k)$$
in $\mathrm{Ho}^b(k \mathrm{N}_G(Q))$, see \cite[Theorem 2.29]{Rouquier}. Building on this fundamental result, Bonnafé--Dat--Rouquier show the following:

\begin{lemma}\label{BrauerGodement}
Let $\G$ be a (non-necessarily connected) reductive group with parabolic subgroup $\Para$ and Levi decomposition $\Para=\Levi \ltimes \U$ such that $F(\Levi)=\Levi$. For an $\ell$-subgroup $Q$ of $\Levi^F$ we have
$$ \Br_{\Delta Q}(G\Gamma_c(\Y_{\U}^\G,\Lambda)) \cong G\Gamma_c(\Y_{\C_\U(Q)}^{\C_\G(Q)},k) $$ in $\mathrm{Ho}^b(k [\mathrm{N}_{\G^F \times \Levi^{\mathrm{opp}}}(\Delta Q )]),$ 
\end{lemma}

\begin{proof}
See \cite[Proposition 3.4(e)]{Dat} and \cite[Remark 3.5]{Dat}.
\end{proof}



The previous lemma can be used to show that the indecomposable summands of the components of the complex $ G\Gamma_c(\Y_\U^\G,\Lambda)^{\mathrm{red}}$ of $\Lambda[\G^F \times \mathrm{N}_{\G^F}(\Para,\Levi)^{\mathrm{opp}}]$-modules have a vertex contained in $\Delta \mathrm{N}_{\G^F}(\Para,\Levi)^{\mathrm{opp}}$, see \cite[Corollary 3.8]{Dat}. Using the proof of \cite[Lemma 4.3]{Rickard} one easily observes that the components of $ G\Gamma_c(\Y_\U^\G,\Lambda)^{\mathrm{red}}$ considered as $\Lambda[\G^F \times (\Levi^F)^{\mathrm{opp}}]$-modules are relatively $\Delta \Levi^F$-projective.

\subsection{Levi subgroups and duality}

We recall the classification of $F$-stable Levi subgroups of a connected reductive group $\G$. Fix an $F$-stable maximal torus $\T_0$ of $\G$ contained in an $F$-stable Borel subgroup $\mathbf{B}_0$ of $\G$. Let $\Phi$ be the root system of $\G$ relative to the torus $\T_0$ and $\Delta \subseteq \Phi$ the base of $\Phi$ associated to $\T_0 \subseteq \B_0$. By \cite[Proposition 4.3]{DM} the $\G^F$-conjugacy classes of $F$-stable Levi subgroups of $\G$ are classified by $F$-conjugacy classes of cosets $W_I w$, where $I \subseteq \Delta$ and $w \in W$ satisfies ${}^{w F} W_I = W_I$. More precisely, if $\Levi$ is an $F$-stable Levi subgroup of $\G$ of type $W_I w$ then there exists $g \in \G^F$ such that ${}^g \Levi = \Levi_I$ for some $I \subseteq \Delta$ and ${}^{g^{-1}} \T_0$ is a maximal torus of $\Levi$ of type $w= g^{-1} F(g) \T_0$. Here, $\Levi_I$ denotes the standard Levi subgroup of $\G$ associated to a subset $I$ of the base $\Delta$, see \cite[Section 12.2]{MT}.

An important property of duality is that it extends to Levi subgroups.
%

\begin{lemma}\label{bijLevidual}
Suppose that $(\G^\ast,\T_0^\ast,F^\ast)$ is in duality with $(\G,\T_0,F)$. Then the map which sends a Levi subgroup $\Levi$ of $\G$ of type $W_I w$ to a Levi subgroup $\Levi^\ast$ of $\G^\ast$ of type $W_I^\ast F^\ast(w^\ast)$ induces a bijection between the $\G^F$-conjugacy classes of $F$-stable Levi subgroups of $\G$ and the $(\G^\ast)^{F^\ast}$-conjugacy classes of $F^\ast$-stable Levi subgroups of $\G^\ast$. 
\end{lemma}

\begin{proof}
See \cite[Section 8.2]{MarcBook}.
\end{proof}
\subsection{Isogenies}


Let $\G$ be a connected reductive group. Recall that an \textit{isogeny of algebraic groups} $\varphi:\G \to \G$ is a surjective homomorphism of algebraic groups with finite kernel. Let $\varphi: \G \to \G$ be an isogeny stabilizing a maximal torus $\T_0$ of $\G$. We write $X(\T_0)$ for the character group of $\T_0$ and $Y(\T_0)$ for the cocharacter group of $\T_0$, see \cite[Definition 3.4]{MT}. The morphism $\varphi$ induces a group homomorphism $\varphi:X(\T_0) \to X(\T_0)$ and its dual morphism $\varphi^\vee:Y(\T_0) \to Y(\T_0), \, y \mapsto \varphi \circ y$, which preserve the set of roots $\Phi(\T_0)$ resp. coroots $\Phi^\vee(\T_0)$.
We will now define what it means for isogenies to be in duality with each other.

\begin{definition}\label{isogeniesdualdef}
Suppose that $(\G,\T_0,F)$ is in duality with $(\G^\ast,\T_0^\ast,F^\ast)$. We say that isogenies $\sigma: \G \to \G$ and $\sigma^\ast: \G^\ast \to \G^\ast$ are \textit{in duality with each other} if there exist $ g\in \G$ and $h \in \G^\ast$ such that $\sigma_0:=g \sigma$ stabilizes $\T_0$ (resp. $\sigma_0^\ast:=h \sigma^\ast$ stabilizes $\T_0^\ast$) and $\delta \circ \sigma_0= (\sigma_0^\ast)^\vee \circ \delta$ on $Y(\T_0^\ast)$.
\end{definition}

Note that this means that dual isogenies are only defined up to inner automorphisms of $\G$ respectively $\G^\ast$. The following remark is crucial for working with automorphisms of finite groups of Lie type, see also \cite[Section 2]{Navarro} and the proof of \cite[Proposition 2.2]{Spaeth}. 

\begin{remark}\label{dualisogeny}
Recall that we have fixed a pair $(\T_0,\B_0)$ consisting of an $F$-stable maximal torus $\T_0$ of $\G$ contained in an $F$-stable Borel subgroup $\mathbf{B}_0$ of $\G$. Since $(\sigma(\T_0),\sigma(\mathbf{B}_0))$ is again such a pair it follows that ${}^g (\T_0,\B_0)=(\sigma(\T_0),\sigma(\mathbf{B}_0))$ for some $g \in \G^F$. Hence, we may assume that $\sigma$ stabilizes the pair $(\T_0,\B_0)$. Thus, the isogeny theorem (see \cite[Theorem 9.6.2]{Springer}) together with \cite[Lemma 5.5]{Taylor} shows that there exists a bijective morphism $\sigma^\ast: \G^\ast \to \G^\ast$ in duality with $\sigma$. Moreover, since $\sigma^\ast$ is unique up to inner automorphism we can choose $\sigma^\ast$ such that $\sigma^\ast F^\ast=F^\ast \sigma^\ast$. (We first have $t \sigma^\ast  F^\ast=F^\ast \sigma^\ast$ for some $t \in \T_0^\ast$. Then by Lang's theorem there exists $t_0 \in \T_0^\ast$ such that $t_0 \sigma^\ast$ commutes with $F^\ast$.) The isogeny $\sigma^\ast$ with these properties is then unique up to conjugation with elements of $\mathcal{L}^{-1}_{F^\ast}(\mathrm{Z}(\G))$, where $\mathcal{L}_{F^\ast}: \G^\ast \to \G^\ast,\, x \mapsto x^{-1} F^\ast(x),$ is the Lang map of $F^\ast$ on $\G^\ast$.
\end{remark}

\begin{corollary}\label{dualitysigma}
Let $\sigma: \G \to \G$ be a bijective morphism with $\sigma \circ F= F \circ \sigma$ and $\sigma^\ast: \G^\ast \to \G^\ast$ be a dual isogeny with $\sigma^\ast \circ F^\ast= F^\ast \circ \sigma^\ast$. Under the bijection in Lemma \ref{bijLevidual}, the set of $\sigma$-stable $\G^F$-conjugacy classes of $F$-stable Levi subgroups of $\G$ corresponds to the set of $\sigma^\ast$-stable $(\G^\ast)^{F^\ast}$-conjugacy classes of $F^\ast$-stable Levi subgroups of $\G^\ast$. 
\end{corollary}

\begin{proof}
As in Remark \ref{dualisogeny} we may assume without loss of generality that $\sigma$ stabilizes the pair $(\T_0,\B_0)$. We may also assume that $\sigma^\ast:\G^\ast \to \G^\ast$ satisfies $\delta \circ \sigma= (\sigma^\ast)^\vee \circ \delta$ on $Y(\T_0^\ast)$, see Definition \ref{isogeniesdualdef}. In particular, this yields $w^\ast=\sigma^\ast(\sigma(w)^\ast)$ for all $w \in W$ (same proof as in \cite[Proposition 4.3.2]{Carter}).

Observe that the $\G^F$-conjugacy class of an $F$-stable Levi subgroup of type $W_I w$ is $\sigma$-stable if and only if $\sigma(W_{I} w)$ is $F$-conjugate to $W_I w$. This is equivalent to $\sigma^\ast(W_{I}^\ast F^\ast(w^\ast))$ being $F^\ast$-conjugate to $W_I^\ast F^\ast(w^\ast)$. The latter is now equivalent to the $(\G^\ast)^{F^\ast}$-conjugacy class of $F^\ast$-stable Levi subgroups of $\G^\ast$ associated to $W_I^\ast F^\ast(w^\ast)$ being $\sigma^\ast$-stable. This gives the claim. 
\end{proof}

\subsection{Lusztig series for disconnected reductive groups}

We give an elementary description of Lusztig series for disconnected reductive groups introduced in \cite{Dat}.

Let $\G$ be a non-necessarily connected reductive group. Note that the maximal tori of $\G$ are the maximal tori of $\G^\circ$. As in the case of connected reductive groups, we denote by $\nabla(\G,F)$ the set of pairs $(\T,\theta)$ where $\T$ is an $F$-stable maximal torus of $\G$ and $ \theta \in \Irr(\T^F)$ is an irreducible character of $\T^F$. We denote by $\nabla_{\ell'}(\G,F)$ the subset of $\nabla(\G,F)$ consisting of the pairs $(\T,\theta)$ such that the order of $\theta$ is coprime to $\ell$. Note that $\Irr(\T^F)_{\ell'}$ can be identified with the set of characters $\T^F \to \Lambda^\times$ of $\ell'$-order. We denote by $e_\theta \in \mathrm{Z}(\Lambda \T^F)$ the unique central primitive idempotent of $\Lambda \T^F$ with $\theta(e_\theta) \neq 0$. 

\begin{definition}
We say that two pairs $(\T_1,\theta_1) \in \nabla(\G,F)$ and $(\T_2,\theta_2) \in \nabla(\G,F)$ are \textit{rationally conjugate} if there exists some $t \in \mathrm{N}_{\G^F}(\T_1)$ such that $(\T_1,{}^t \theta_1)$ and $(\T_2,\theta_2)$ are rationally conjugate in $\G^\circ$. We write $\nabla(\G,F) / \equiv $ for the set of equivalence classes under rational conjugation. 
\end{definition}


With this in mind, we can now state the following definition from \cite[4.D.]{Dat}.

\begin{definition}\label{Def2}
Let $\mathcal{X} \subseteq \nabla_{\ell'}(\G,F)$ be a rational series of $\G$. We denote by $\mathcal{C}_{\mathcal{X}}$ the thick subcategory of $\Lambda \G^F\text{-} \mathrm{perf}$ generated by the complexes $R\Gamma_c(\Y_\B) e_{\theta}$, with $(\T,\theta) \in \mathcal{X}$ and $\B$ a Borel subgroup of $\G^\circ$ with maximal torus $\T$. We denote by $e_{\mathcal{X}} \in \Lambda \G^F$ the central idempotent such that $\mathcal{C}_{\mathcal{X}}=\Lambda \G^F e_{\mathcal{X}}\text{-} \mathrm{perf}$.
\end{definition}

Note that the existence of the idempotents $e_{\mathcal{X}} \in \mathrm{Z}(\Lambda \G^F)$ is ensured by \cite[Theorem 4.12]{Dat}.

\begin{remark}
Suppose that $\G$ is a connected reductive group. We let $s \in (\G^\ast)^{F^\ast}$ be a semisimple element of $\ell'$-order such that its $(\G^\ast)^{F^\ast}$-conjugacy class is associated to the rational series $\mathcal{X} \in \nabla_{\ell'}( \G,F)$. Then we have $e_s^{\G^F}=e_{\mathcal{X}}$, see \cite[Remark 9.3]{BoRo}.
\end{remark}


\begin{lemma}\label{discratl}
Let $\mathcal{X} \subseteq \nabla_{\ell'}(\G,F)$ be a rational series of $\G$ and choose a rational series $\mathcal{X}^\circ$ of $\G^\circ$ such that $\mathcal{X}^\circ \subseteq \mathcal{X}$. Then we have $e_\mathcal{X}=\Tr^{\G^F}_{\mathrm{N}_{\G^F}(e_{\mathcal{X}^\circ})}(e_{\mathcal{X}^\circ})$. 
\end{lemma}
\begin{proof}
Write  $e'_{\mathcal{X}}:=\sum_{g \in \G^F/\mathrm{N}_{\G^F}(e_{\mathcal{X}^\circ})} {}^g e_{\mathcal{X}^\circ}$. 
Let $(\T,\theta) \in \mathcal{X}$ and $\B$ be a Borel subgroup of $\G^\circ$ with maximal torus $\T$. By \cite[(3.1)]{Dat} we have $$\mathrm{Ind}_{(\G^\circ)^F}^{\G^F}( R \Gamma_c(\Y^{\G^\circ}_\B) e_\theta) \cong R \Gamma_c(\Y^{\G}_{\B}) e_{\theta} \cong \mathrm{Ind}_{(\G^\circ)^F}^{\G^F} (R \Gamma_c(\Y^{\G^\circ}_{{}^t \B}) e_{{}^t \theta})$$ for any $t\in \mathrm{N}_{\G^F}(\T)$. Therefore, the generators of $\mathcal{C}_{\mathcal{X}}$ lie inside $\Lambda \G^F e'_{\mathcal{X}}\text{-} \mathrm{perf}$. Thus, $\mathcal{C}_{\mathcal{X}}$ is a subcategory of $\Lambda \G^F  e'_{\mathcal{X}}\text{-} \mathrm{perf}$ and we have $e_\mathcal{X} e'_\mathcal{X} \neq 0$. By \cite[Theorem 4.12]{Dat} it follows that we have two decompositions
$$1=\sum_{\mathcal{Z}\in \nabla_{\ell'}(\G,F)/\equiv} e_\mathcal{Z}= \sum_{\mathcal{Z}\in \nabla_{\ell'}(\G,F)/\equiv} e'_{\mathcal{Z}}.$$
into orthogonal central idempotents. From this we deduce that $e_{\mathcal{X}}=e'_{\mathcal{X}}$.
\end{proof}

We recall the definition of (super)-regular series, see \cite[Section 11.4]{BoRo} and in particular \cite[Lemma 11.6]{BoRo}. 

\begin{definition}\label{superregular}
Let $\G$ be a connected reductive group and $\Levi$ be a $F$-stable Levi subgroup of $\G$. We say that the rational series of $(\Levi,F)$ associated to the conjugacy class of the semisimple element $s\in (\Levi^\ast)^{F^\ast}$ is $(\G,\Levi)$\textit{-regular} (respectively \textit{superregular}) if $\C^\circ_{\G^\ast}(s)\subseteq \Levi^\ast$ (respectively $\C_{\G^\ast}(s) \subseteq \Levi^*$).
\end{definition}

This notion can now be naturally extended to rational series of disconnected reductive groups. Let $\G$ be a reductive group. If $\mathcal{X}$ is a rational series of $(\Levi,F)$ we say that $\mathcal{X}$ is \textit{$(\G^\circ,\Levi^\circ)$-regular} (respectively \textit{superregular}) if any (and hence every) rational series of $(\Levi^\circ,F)$ contained in $\mathcal{X}$ is $(\G^\circ,\Levi^\circ)$-regular (respectively superregular). We then say that the central idempotent $e_{\mathcal{X}}$ associated to $\mathcal{X}$ is $(\G^\circ,\Levi^ \circ)$-\textit{(super)-regular}.

\subsection{Lusztig series and Brauer morphism}

Let $\G$ be a reductive group and $Q$ a finite $\ell$-subgroup of $\G^F$. Then we consider the map 
$$i_Q^{\G}:\nabla_{\ell'}(\C_{\G}(Q),F)/\equiv \to \nabla_{\ell'}(\G,F)/\equiv$$
as defined in \cite[Theorem 4.14]{Dat}. By \cite[Theorem 4.14]{Dat} for any rational series $\mathcal{Y} \subseteq \nabla_{\ell'}(\G,F)$ we have
$$\operatorname{br}^{\G^F}_Q(e_{\mathcal{Y}})=\displaystyle\sum_{\mathcal{Z} \in(i_{Q}^{\G})^{-1}(\mathcal{Y})} e_{\mathcal{Z}}.$$

\begin{lemma}\label{brauerseries}
Let $\Levi$ be an $F$-stable Levi subgroup of $\G$ and let $\mathcal{X} \subseteq \nabla_{\ell'}(\Levi,F)$ be a $(\G^\circ,\Levi^\circ)$-(super)-regular rational series. Then for any $\ell$-subgroup $Q$ of $\Levi^F$ we have that
$$\operatorname{br}^{\Levi^F}_Q(e_{\mathcal{X}})=\displaystyle\sum_{\mathcal{Z} \in(i_{Q}^{\Levi})^{-1}(\mathcal{X})} e_{\mathcal{Z}}$$
is a decomposition into central orthogonal $(\C^\circ_\G(Q),\C^\circ_\Levi(Q))$-(super-)regular idempotents.
\end{lemma}

\begin{proof}
See \cite[Proposition 4.11]{Dat}.
\end{proof}

We gather some useful facts. 

\begin{lemma}\label{basicseries}
Let $\Levi$ be an $F$-stable Levi subgroup of $\G$ and $\Para$ a parabolic subgroup of $\G$ with Levi decomposition $\Para=\Levi \ltimes \U$. In addition, let $\mathcal{X} \subseteq \nabla_{\ell'}(\Levi,F)$ be a regular rational series of $(\G^\circ,\Levi^\circ)$.
\begin{enumerate}[label=(\alph*)]
\item There exists a unique rational series $\mathcal{Y} \subseteq \nabla_{\ell'}(\G,F)$ containing $\mathcal{X}$.
\item Deligne--Lusztig induction restricts to a functor $\mathcal{R}_{\Levi \subseteq \Para}^\G: \mathrm{D}^b(\Lambda \Levi^F e_{\mathcal{X}}) \to \mathrm{D}^b(\Lambda \G^F e_{\mathcal{Y}})$.
\item Given $\mathcal{X}' \in  (i_{Q}^{\Levi})^{-1}(\mathcal{X})$ let $\mathcal{Y}'$ be the unique rational series of $\nabla_{\ell'}(\mathrm{C}_\G(Q),F)$ containing $\mathcal{X}'$. Then we have $i_Q^{\G}(\mathcal{Y})=\mathcal{Y}'$.
\end{enumerate}
\end{lemma}

\begin{proof}
	Let $\mathcal{X}^\circ$ be a rational series of $(\Levi^\circ,F)$ contained in $\mathcal{X}$. Then $\mathcal{X}^\circ$ is associated to the conjugacy class of a semisimple element $s \in ((\Levi^\circ)^\ast)^{F^\ast}$ of $\ell'$-order. The rational series $\mathcal{Y}^\circ$ of $(\G^\circ,F)$ associated to $s \in ((\G^\circ)^\ast)^{F^\ast}$ is the unique rational series containing $\mathcal{Y}^\circ$. Thus, $\mathcal{Y}$ is the unique rational series of $(\G,F)$ containing $\mathcal{Y}^\circ$. This shows part (a).
	
	Deligne--Lusztig induction restricts to a functor $\mathcal{R}_{\Levi^\circ \subseteq \Para^\circ}^{\G^\circ}: \mathrm{D}^b(\Lambda (\Levi^\circ)^F e_s^{\Levi^F}) \to \mathrm{D}^b(\Lambda (\G^\circ)^F e_s^{\G^F})$ by \cite[Theorem 11.4]{BoRo}. For part (b) it suffices to show that $\mathcal{R}_{\Levi \subseteq \Para}^\G(M) \in \mathrm{D}^b(\Lambda \G^F e_{\mathcal{Y}})$ for $M=\Lambda \Levi^F e_{\mathcal{X}}$. By \cite[(3.1)]{Dat} we have 
	$$\mathrm{Ind}_{(\G^\circ)^F}^{\G^F} \circ \mathcal{R}_{\Levi^\circ \subseteq \Para^\circ}^{\G^\circ}= \mathcal{R}_{\Levi \subseteq \Para}^\G \circ \mathrm{Ind}_{(\Levi^\circ)^F}^{\Levi^F}.$$
By Lemma \ref{discratl} we have $M= \mathrm{Ind}_{(\Levi^\circ)^F}^{\Levi^F}( \Lambda \Levi^F e_s^{\Levi^F})$ and we can conclude that $\mathcal{R}_{\Levi \subseteq \Para}^\G(M) \in \mathrm{D}^b(\Lambda \G^F e_{\mathcal{Y}})$. This implies that $\mathcal{R}_{\Levi \subseteq \Para}^\G$ restricts to a functor $\mathcal{R}_{\Levi \subseteq \Para}^\G: \mathrm{D}^b(\Lambda \Levi^F e_{\mathcal{X}}) \to \mathrm{D}^b(\Lambda \G^F e_{\mathcal{Y}})$.
			
	We now prove part (c). By Lemma \ref{BrauerGodement} we have 
	$$\mathrm{Br}_{\Delta Q}(G \Gamma_c(\Y_\U^\G) e_{\mathcal{X}})  \cong G \Gamma_c( \Y_{\C_\U(Q)}^{\C_\G(Q)}) , \Lambda) \mathrm{br}_Q(e_{\mathcal{X}})=\br_Q(e_\mathcal{Y})  G \Gamma_c( \Y_{\C_\U(Q)}^{\C_\G(Q)}) , \Lambda) \mathrm{br}_Q(e_{\mathcal{X}}).$$
On the other hand, by part (b) $\mathcal{R}_{\C_{\Levi}(Q) \subseteq \C_\Para(Q)}^{\C_\G(Q)}$ restricts to a functor $$\mathcal{R}_{\C_{\Levi}(Q) \subseteq \C_\Para(Q)}^{\C_\G(Q)}: \mathrm{D}^b(\Lambda \C_{\Levi^F}(Q) e_{\mathcal{X}'}) \to \mathrm{D}^b(\Lambda \C_{\G^F}(Q) e_{\mathcal{Y}'}).$$ This implies that $\mathrm{br}_Q(e_\mathcal{Y}) e_{\mathcal{Y}'} \neq 0$ which shows that $e_{\mathcal{Y}'}$ appears in the decomposition into central idempotents of $\mathrm{br}_Q(e_\mathcal{Y})$ from Lemma \ref{brauerseries}. Therefore, we necessarily have $i_Q^{\G}(\mathcal{Y})=\mathcal{Y}'$.
\end{proof}

For our purposes it is necessary to slightly modify the definition of (super-)regular series.

\begin{definition}\label{almost superregular}
Let $\G$ be a connected reductive group and $\Levi$ be a $F$-stable Levi subgroup of $\G$. We say that the rational series of $(\Levi,F)$ associated to the conjugacy class of the semisimple element $s\in (\Levi^\ast)^{F^\ast}$ is \textit{almost} $(\G,\Levi)$\textit{-superregular} if $\mathrm{C}^\circ_{\G^\ast}(s) \C_{\G^\ast}(s)^{F^\ast} \subseteq \Levi^\ast$.
\end{definition}

\begin{remark}\label{almost}
	The result of Lemma \ref{basicseries} remains true if we replace everywhere the word (super)-regular by the word almost super regular. To show this one takes the proof of Lemma \ref{brauerseries} and takes $F$-fixed points at the appropriate places.
\end{remark}
\subsection{Equivariance of Deligne--Lusztig induction}

In this section we establish some elementary results on the action of group automorphisms on Deligne--Lusztig varieties. Most of the results in this section are known, see \cite[Section 2]{Navarro}.

Let $\G$ be a reductive group and $\sigma:\G \to \G$ be a bijective morphism of algebraic groups which commutes with the action of the Frobenius endomorphism $F$, i.e. we have $\sigma \circ  F= F \circ \sigma$. Let $\Para$ be a parabolic subgroup of $\G$ with Levi decomposition $\Para=\Levi \ltimes \U$ such that $F(\Levi)=\Levi$. Note that $\sigma(\Para)$ is a parabolic subgroup of $\G$ with $F$-stable Levi $\sigma(\Levi)$ and unipotent radical $\sigma(\U)$.

%
%

%
%
%

\begin{lemma}\label{sigma}
	With the notation as above, $\sigma$ induces an isomorphism
	$$\sigma^*:G\Gamma_c(\Y^{\G}_{\sigma(\U)},\Lambda)\to {}^\sigma G \Gamma_c(\Y^{\G}_\U,\Lambda)^{\sigma}$$
	in $\mathrm{Ho}^b(\Lambda[\G^F \times (\Levi^F)^\mathrm{opp}])$.
\end{lemma}

\begin{proof}
	The variety $\Y_{\U}^{\G^\circ}$ is smooth, see for instance \cite[Theorem 7.2]{MarcBook}. By Lemma \ref{geometricgeneralversion2}, we have $\Y^\G_\U=\coprod_{g\in \G^F / (\G^\circ)^F} g \Y_{\U}^{\G^\circ}$ which implies that the variety $\Y_{\U}^{\G}$ is smooth as well. Thus, as in the proof of \cite[Proposition 2.1]{Navarro} it follows that the morphism $\sigma: {}^\sigma \Y_\U^{\sigma^{-1}} \to \Y_{\sigma(\U)}$ given by $g \U \mapsto \sigma(g) \sigma(\U)$ is $\G^F \times (\Levi^F)^{\operatorname{opp}}$-equivariant and induces an isomorphism of \'etale sites.
	Now \cite[Theorem 2.12]{Rouquier} shows that the map $G\Gamma_c(\Y^{\G}_{\sigma(\U)},\Lambda)\to {}^\sigma G \Gamma_c(\Y^{\G}_\U,\Lambda)^{\sigma}$ is an isomorphism.
\end{proof}

Assume now that $\G$ is connected and let $s\in (\G^\ast)^{F^\ast}$ be a semisimple element of $\ell'$-order. Using Lemma \ref{sigma} one can show that ${}^\sigma \mathcal{E}(\G^F,t)= \mathcal{E}(\G^F,(\sigma^*)^{-1}(t))$ for every semisimple element $t \in  (\G^\ast)^{F^\ast}$, see \cite[Corollary 2.4]{Navarro} and also \cite[Proposition 7.2]{Taylor}. In particular, for a semisimple element of $\ell'$-order, $\sigma(e_s^{\G^F})=e_{(\sigma^*)^{-1}(s)}^{\G^F}$.

\subsection{The Bonnafé--Dat--Rouquier Morita equivalence}\label{BDR}

Let $\G$ be a connected reductive group defined over an algebraic closure of $\mathbb{F}_p$, where $p$ is a prime number. Let $F: \G \to \G$ be a Frobenius endomorphism of $\G$ defining an $\mathbb{F}_q$-structure on $\G$. Let $(\G^\ast,F^\ast)$ be in duality with $(\G,F)$. Fix a semisimple element $s \in (\G^\ast)^{F^\ast}$ of $\ell'$-order. Let $\Levi^\ast$ be an $F^\ast$-stable Levi subgroup of $\G^\ast$ which satisfies $\C^\circ_{\G^*}(s) \subseteq \Levi^*$ and 
$$\Levi^\ast \C_{\G^\ast}(s)^{F^\ast}=  \C_{\G^\ast}(s)^{F^\ast} \Levi^\ast.$$
This assumption is for instance satisfied if $\Levi^\ast=\C_{\G^\ast}(\mathrm{Z}^\circ(\C^\circ_{\G^\ast}(s)))$ is the minimal Levi subgroup of $\G^\ast$ containing $\mathrm{C}_{\G^\ast}^\circ(s)$ or if $\Levi^\ast$ is any Levi subgroup of $\G^\ast$ containing $\C_{\G^\ast}(s)^{F^\ast}$. Then we define $$\N^\ast:= \C_{\G^*}(s)^{F^*} \Levi^*$$ which is a subgroup of $\G^\ast$ by the property above. Note that $\N^*$ is an $F^\ast$-stable subgroup of $\operatorname{N}_{\G^*}(\Levi^*)$. Let $\Levi$ be an $F$-stable Levi subgroup of $\G$ in duality with the Levi subgroup $\Levi^\ast$ of $\G^\ast$. We let $\N$ be the subgroup of $\mathrm{N}_\G(\Levi)$ corresponding to $\N^\ast$ under the isomorphism of the relative Weyl groups
 $$\operatorname{N}_\G(\Levi)/\Levi \cong \operatorname{N}_{\G^*}(\Levi^*)/\Levi^\ast$$
induced by duality. The closed subgroup $\mathbf{N}$ of $\G$ is $F$-stable and it holds that $\mathbf{N}^F=\mathrm{N}_{\G^F}(\Levi,e_s^{\Levi^F})$ by \cite[(7.1)]{Dat}. We let $\Para$ be a paroblic subgroup with Levi decomposition $\Para= \Levi \ltimes \U$. In addition, we let $d:=\operatorname{dim}(\Y_{\U})$. By \cite[Theorem 11.7]{BoRo} we have
$$H^i_c( \Y_\U^\G, \Lambda) e_s^{\Levi^F} =0 \text{ for } i \neq d.$$
Hence, we are interested only in the $d$th cohomology group of the variety $\Y_\U^\G$. For convenience, we will therefore use the following notation.

\begin{notation}
	Let $\mathbf{X}$ be a variety of dimension $n$. Then we write $R\Gamma^{\mathrm{dim}}_c(\mathbf{X},\Lambda):=R \Gamma_c(\mathbf{X},\Lambda) [n]$ and $H_c^{\mathrm{dim}}(\mathbf{X},\Lambda):=H_c^{n}(\mathbf{X},\Lambda)$.
\end{notation}

Let $\iota: \G \hookrightarrow \tilde{\G}$ be a regular embedding. Set $\Ltilde=\Levi \mathrm{Z}(\tilde{\G})$ and $\tilde{\mathbf{N}}=\mathbf{N} \Ltilde$.

\begin{assumption}\label{assumptionBDR}
Suppose that the $k[(\G^F \times (\Levi^F)^{\mathrm{opp}}) \Delta \Ltilde^F]$-module $$H_c^{\mathrm{dim}}(\Y_\U,k) e_s^{\Levi^F}$$
extends to a $k[(\G^F \times (\Levi^F)^{\mathrm{opp}})  \Delta \tilde{\mathbf{N}}^F]$-module.
\end{assumption}

This assumption is for instance satisfied if $\mathbf{N}^F / \Levi^F$ is cyclic, see \cite[Lemma 10.2.13]{Rouquier3}.

 We have the following theorem, see \cite[Theorem 7.7]{Dat}:

\begin{theorem}[Bonnaf\'e--Dat--Rouquier]\label{Boro}
Suppose that Assumption \ref{assumptionBDR} holds. Then there exists an $\mathcal{O} \G^F $-$\mathcal{O} \N^F$-bimodule extending $H_c^d(\Y_\U^\G,\mathcal{O}) e_s^{\Levi^F}$ and for any such bimodule $M$ there exists a complex $C$ of $\mathcal{O} \G^F$-$\mathcal{O} \N^F $-bimodules extending $G\Gamma_c(\Y^\G_{\U}, \Lambda)^{\operatorname{red}}e_s^{\Levi^F}$ such that $H^{d}(C)\cong M$. The complex $C$ induces a splendid Rickard equivalence between $\mathcal{O} \G^F e_s^{\G^F}$ and $\mathcal{O} \N^F e_s^{\Levi^F}$ and the bimodule $M$ induces a Morita equivalence between $\mathcal{O} \G^F e_s^{\G^F}$ and $\mathcal{O} \N^F e_s^{\Levi^F}$.
\end{theorem} 

\begin{proof}
In the proof of \cite[Theorem 7.5]{Dat} apply Assumption \ref{assumptionBDR} instead of \cite[Proposition 7.3]{Dat}. The rest of the proof of the theorem is as in \cite[Section 7]{Dat}. 
\end{proof}


Note that the assumption previously made that $\Levi^\ast$ normalizes $\mathrm{C}_{\G^\ast}(s)^{F^\ast}$ is not necessary for the following theorem. This means we only assume that $\Levi^\ast$ is an $F^\ast$-stable Levi subgroup containing $\mathrm{C}_{\G^\ast}^\circ(s)$.

\begin{theorem}\label{independence}
Let $\Para_1$ and $\Para_2$ be two parabolic subgroups of $\G$ with common Levi complement $\Levi$ and unipotent radical $\U_1$ respectively $\U_2$. Then we have 
$$H_c^{\dim}(\Y_{\U_1},\Lambda) e_s^{\Levi^F}\cong H_c^{\dim}(\Y_{\U_2},\Lambda) e_s^{\Levi^F}$$
as $\Lambda \G^F$-$\Lambda {\Levi^F}$-bimodules.
\end{theorem}

\begin{proof}
This is proved in \cite[Theorem 7.2]{Dat}. We sketch how the isomorphism of the theorem is obtained. All mentioned statements are proved in loc. cit. We define 
$$\Y_{\U_1,\U_2}^{\operatorname{cl}}:=\{ (g_1 \U_1,g_2 \U_2) \in \Y_{\U_1,\U_2} \mid g_1 \U_1 \in \Y_{\U_1} \},$$
which is a $\G^F \times (\Levi^F)^{\operatorname{opp}}$-stable closed subvariety of $\Y_{\U_1,\U_2}$.
We have a closed immersion $i_{\U_1,\U_2}: \Y_{\U_1,\U_2}^{\operatorname{cl}} \to \Y_{\U_1,\U_2}$ and a natural projection map $\pi_{\U_1,\U_2}: \Y_{\U_1,\U_2}^{\operatorname{cl}} \to \Y_{\U_1}$.

We have an isomorphism $\pi_{\U_1,\U_2}^*: R\Gamma_c(\Y_{\U_1},\Lambda)[-2 d] \to R \Gamma_c(\Y_{\U_1,\U_2}^{\operatorname{cl}},\Lambda)$ where $d=\dim(\U_1 \cap F(\U_1))-\dim(\U_1 \cap \U_2 \cap F(\U_1))$. Moreover, we have a morphism $i_{\U_1,\U_2}^*: R \Gamma_c(\Y_{\U_1,\U_2},\Lambda) \to R\Gamma_c(\Y_{\U_1,\U_2}^{\operatorname{cl}},\Lambda)$. The resulting map 
$$\psi_{\U_1,\U_2}=  (\pi_{\U_1,\U_2}^*)^{-1} \circ i_{\U_1,\U_2}^*: R\Gamma_c^{\dim}(\Y_{\U_1,\U_2},\Lambda) \to  R\Gamma_c^{\dim}(\Y_{\U_1},\Lambda)$$
induces a quasi-isomorphism 
$$\psi_{\U_1,\U_2,s}:R\Gamma^{\dim}_c(\Y_{\U_1,\U_2},\Lambda) e_s^{\Levi^F} \to  R\Gamma^{\dim}_c(\Y_{\U_1},\Lambda) e_s^{\Levi^F}$$
of $\Lambda \G^F$-$\Lambda \Levi^F$-complexes. Similarily, the map $\psi_{\U_2,F(\U_1)}$ induces a quasi-isomorphism $$\psi_{\U_2,F(\U_1),s}:R\Gamma^{\dim}_c(\Y_{\U_2,F(\U_1)},\Lambda) e_s^{\Levi^F} \to  R\Gamma^{\dim}_c(\Y_{\U_2},\Lambda) e_s^{\Levi^F}$$
of $\Lambda \G^F$-$\Lambda \Levi^F$-complexes. However, the shift map $$\operatorname{sh}:\Y_{\U_1,\U_2} \to \Y_{\U_2,F(\U_1)}$$ given by $(g_1 \U_1,g_2 \U_2) \mapsto (g_2 \U_2,F(g_1 \U_1))$ is $\G^F$-$\Levi^F$-equivariant and induces an equivalence of \'etale sites. In particular this map induces a quasi-isomorphism
$$\operatorname{sh}^*:R\Gamma^{\dim}_c(\Y_{\U_1,\U_2},\Lambda) \to R\Gamma^{\dim}_c(\Y_{\U_2,F(\U_1)},\Lambda)$$
of $\Lambda \G^F$-$\Lambda \Levi^F$-complexes. Consequently, we have a quasi-isomorphism 
$$\Theta_{\U_2,\U_1}:=\psi_{\U_2,F(\U_1),s} \circ \operatorname{sh}^* \circ \psi_{\U_1,\U_2,s}^{-1}:R\Gamma_c(\Y_{\U_1},\Lambda) e_s^{\Levi^F} \to R\Gamma_c(\Y_{\U_2},\Lambda) e_s^{\Levi^F}$$
of $\Lambda \G^F$-$\Lambda \Levi^F$-complexes.
\end{proof}


We abbreviate $\C_{(\G^\ast)^{F^\ast}}(s) \mathrm{C}^\circ_{\G^\ast}(s)$ by $\mathrm{C}_{\G^\ast,F^\ast}(s)$ or even $\mathrm{C}(s)$ if $(\G^\ast,F^\ast)$ is clear from the context. We mention the following important special case of Theorem \ref{Boro} which follows from Theorem \ref{Boro} by observing that $\N^F=\Levi^F$ if and only if $\C(s) \subseteq \Levi^\ast$.

\begin{theorem}[Bonnaf\'e--Dat--Rouquier]\label{Boro2}
Let $\Levi^*$ be an $F^*$-stable Levi subgroup of $\G^*$ containing $\C^\circ_{\G^*}(s) \C_{(\G^\ast)^{F^\ast}}(s)$. Then the complex $C=G\Gamma_c(\Y_{\U}, \mathcal{O})^{\operatorname{red}}e_s^{\Levi^F}$ of $\mathcal{O} \G^F e_s^{\G^F}$-$\mathcal{O} \Levi^F e_s^{\Levi^F}$ bimodules induces a splendid Rickard equivalence between $\mathcal{O}\Levi^F e_s^{\Levi^F}$ and $\mathcal{O}\G^F e_s^{\G^F}$. The bimodule $H^{\operatorname{dim}(\Y_{\U})}(C)$ induces a Morita equivalence between $\mathcal{O} \G^F e_s^{\G^F}$ and  $\mathcal{O} \Levi^F e_s^{\Levi^F}$.
\end{theorem}



\section{Generalizations to disconnected reductive groups}

In this section we suppose that $\hat{\G}$ is a reductive group with Frobenius $F$ and we let $\G$ be a closed normal connected $F$-stable subgroup of $\hat{\G}$. We let $s\in (\G^\ast)^{F^\ast}$ be a semisimple element of $\ell'$-order and $\Levi^\ast$ be an $F^\ast$-stable Levi subgroup of $\G^\ast$ such that $\C^\circ_{\G^\ast}(s) \subseteq \Levi^\ast$. Let $\Levi$ be a Levi subgroup of $\G$ in duality with $\Levi^\ast$.
%
As in Section \ref{DLvarieties} we suppose that $\Para= \Levi \U$ and $\hat{\Para}=\hat{\Levi} \U$ are two Levi decomposition of parabolic subgroups $\Para$ of $\G$ and $\hat{\Para}$ of $\hat{\G}$ such that $\hat{\Para} \cap \G= \Para$ and $\hat{\Levi} \cap \G= \Levi$.
Let $\mathcal{D}:=(\G^F \times (\Levi^F)^{\mathrm{opp}})  \Delta(\hat{\Levi}^F)$
and $\mathcal{D}'$ be the stabilizer of the idempotent $e_s^{\G^F} \otimes e_{s}^{\Levi^F}$ in $\hat{\G}^F \times (\hat{\Levi}^F)^{\mathrm{opp}}$. Note that we have $\mathcal{D}'=\G^F \times (\Levi^F)^{\mathrm{opp}} \, \Delta(\mathrm{N}_{\hat{\Levi}^F}(e_s^{\Levi^F}))$

We generalize Theorem \ref{independence} to disconnected reductive groups.

\begin{lemma}\label{local}
	Let $\mathbf{Q}= \Levi \V$ and $\hat{\mathbf{Q}}=\hat{\Levi} \V$ respectively be two Levi decomposition of parabolic subgroups $\mathbf{Q}$ of $\G$ and $\hat{\mathbf{Q}}$ of $\hat{\G}$ respectively which satisfy $\hat{\mathbf{Q}} \cap \G= \mathbf{Q}$. Then we have
	$$H_c^{\dim(\Y_{\U}^{\G})}(\Y_{\U}^{\G},\Lambda) \Tr^{\hat{\Levi}^F}_{\mathrm{N}_{\hat{\Levi}^F}(e_s^{\Levi^F})}(e_s^{\Levi^F})\cong H_c^{\dim(\Y_{\V}^\G)}(\Y_{\V}^{\G},\Lambda)  \Tr^{\hat{\Levi}^F}_{\mathrm{N}_{\hat{\Levi}^F}(e_s^{\Levi^F})}(e_s^{\Levi^F})$$
	as $\Lambda \mathcal{D}$-bimodules.
\end{lemma}

\begin{proof}
	By Theorem \ref{independence} we have a quasi-isomorphism
	$$\Theta_{\U,\V}:=\psi_{\V,F(\U),s} \circ \operatorname{sh}^* \circ \psi_{\U,\V,s}^{-1}:\R\Gamma_c(\Y_{\U},\Lambda) e_s^{\Levi^F} \to \R\Gamma_c(\Y_{\V},\Lambda) e_s^{\Levi^F}.$$
	The varieties involved in the construction of this map have a $\mathcal{D}$-structure (extending the usual $\G^F \times (\Levi^F)^\mathrm{opp}$-structure), see Lemma \ref{extend}. The maps between them in the proof of Theorem \ref{independence} are easily seen to be $\mathcal{D}$-equivariant. Therefore, the map $\Theta_{\V,\U}$ is a quasi-isomorphism of $\Lambda \mathcal{D}'$-complexes. By applying the functor $\mathrm{Ind}^\mathcal{D}_{\mathcal{D}'}$ we obtain an isomorphism
	$$H_c^{\dim(\Y_{\U}^{\hat{\G}})}(\Y_{\U}^{\hat{\G}},\Lambda) \Tr^{\hat{\Levi}^F}_{\mathrm{N}_{\hat{\Levi}^F}(e_s^{\Levi^F})}(e_s^{\Levi^F})\cong H_c^{\dim(\Y_{\V}^{\hat{\G}})}(\Y_{\V}^{\hat{\G}},\Lambda)  \Tr^{\hat{\Levi}^F}_{\mathrm{N}_{\hat{\Levi}^F}(e_s^{\Levi^F})}(e_s^{\Levi^F})$$
	of $\Lambda \mathcal{D}$-modules.
\end{proof}

\begin{lemma}\label{independencedisconnected}
	Suppose that $\mathrm{C}(s) \subseteq \Levi^\ast$ and $\mathrm{N}_{\hat{\Levi}^F}(e_s^{\Levi^F}) \G^F= \mathrm{N}_{\hat{\G}^F}(e_s^{\G^F})$. Then the bimodule $H_c^{\dim}(\Y^{\hat{\G}}_{\U},\Lambda) e_\mathcal{X}$ induces a Morita equivalence between $\Lambda \hat{\G}^F \Tr^{\hat{\G}^F}_{\mathrm{N}_{\hat{\G}^F}(e_s^{\G^F})}(e_s^{\G^F})$ and $\Lambda \hat{\Levi}^F \Tr^{\hat{\Levi}^F}_{\mathrm{N}_{\hat{\Levi}^F}(e_s^{\Levi^F})}(e_s^{\Levi^F})$.
\end{lemma}
\begin{proof}
	
	By Lemma \ref{better version} it follows that the bimodule $$\Ind_{\mathcal{D}}^{\hat{\G}^F \times (\hat{\Levi}^F)^{\mathrm{opp}}} ( H^{\mathrm{dim}}_c(\Y_\U^{\G},\Lambda) ) \Tr^{\hat{\Levi}^F}_{\mathrm{N}_{\hat{\Levi}^F}(e_s^{\Levi^F})}(e_s^{\Levi^F})$$ induces a Morita equivalence between $\Lambda \hat{\G}^F \Tr^{\hat{\G}^F}_{\mathrm{N}_{\hat{\G}^F}(e_s^{\G^F})}(e_s^{\G^F})$ and $\Lambda \hat{\Levi}^F \Tr^{\hat{\Levi}^F}_{\mathrm{N}_{\hat{\Levi}^F}(e_s^{\Levi^F})}(e_s^{\Levi^F})$. On the other hand, Lemma \ref{geometricgeneralversion2} implies that 
	$$\Ind_{\mathcal{D}}^{\hat{\G}^F \times (\hat{\Levi}^F)^{\mathrm{opp}}} H^{\mathrm{dim}}_c(\Y_\U^{\G},\Lambda) \Tr^{\hat{\Levi}^F}_{\mathrm{N}_{\hat{\Levi}^F}(e_s^{\Levi^F})}(e_s^{\Levi^F}) \cong H_c^{\mathrm{dim}}(\Y^{\hat{\G}}_{\U},\Lambda) \Tr^{\hat{\Levi}^F}_{\mathrm{N}_{\hat{\Levi}^F}(e_s^{\Levi^F})}(e_s^{\Levi^F}).$$
\end{proof}

\subsection{Independence of Godement resolution}\label{independenceGodement}

Let $G$ be a finite group and $L$ be a subgroup of $G$. Let $e$ be a central idempotent of $k G$ and $f$ be a central idempotent of $k L$. In this section we consider two complexes $C_1$ and $C_2$ which both induce a splendid equivalence between $k G e$ and $k L f$ and we want to give a criterion when $C_1 \cong C_2$ in $\mathrm{Ho}^b( k [ G \times L^{\mathrm{opp}}])$. 

The following lemma should be compared to \cite[Lemma A.5]{Dat}. Note that we denote by $Q=1$ the trivial subgroup of $G$.

\begin{lemma}\label{homotopy}
	Let $C_1$ and $C_2$ be two bounded complexes of $\ell$-permutation $k G e$-$k L f$-modules inducing a splendid Rickard equivalence between $k G e$ and $k L f$. Suppose that for all $\ell$-subgroups $Q$ of $L$ there exists an integer $d_Q$ such that the cohomology of $\Br_{\Delta Q}(C_1)$ and $\Br_{\Delta Q}(C_2)$ is concentrated in the same degree $d_Q$. In addition, assume that $H^{d_1}(C_1) \cong H^{d_1}(C_2)$. Then we have $C_1 \cong C_2$ in $\mathrm{Ho}^b(k[G \times L^{\mathrm{opp}}])$.
\end{lemma}

\begin{proof}
	As in the proof of \cite[Lemma 10.2.6]{Rouquier3} one shows that the complex $C_1^\vee \otimes_{\Lambda G} C_2$ induces a splendid Rickard self-equivalence of $k L f$. Therefore, we have $\Br_R(C_1^\vee \otimes_{k G} C_2) \cong 0$ in $\mathrm{Ho}^b(k [L \times L^{\mathrm{opp}}])$ if $R$ is not conjugate to a subgroup of $\Delta L$. Moreover, by the proof of \cite[Theorem 4.1]{Rickard} we have $$\Br_{\Delta Q}(C_1^\vee \otimes_{k G} C_2) \cong \Br_{\Delta Q}(C_1^\vee) \otimes_{k \C_G(Q)} \Br_{\Delta Q}(C_2)$$
	for all $\ell$-subgroups $Q$ of $L$. The complexes $\Br_{\Delta Q}(C_1)$, $\Br_{\Delta Q}(C_2)$ are complexes of finitely generated projective $k\mathrm{C}_{G}(Q)$-modules and their cohomology is by assumption concentrated in the same degree $d_Q$. By \cite[Theorem 2.7.1]{Benson} we thus have $H^{i}(\Br_{\Delta Q}(C_1^\vee \otimes_{k G} C_2))=0$ for $i\neq 0$.
	Therefore, we can apply \cite[Lemma A.3]{Dat} and obtain that 
	$$C_1^\vee \otimes_{k G} C_2 \cong  H^0(C_1^\vee \otimes_{k G} C_2) \cong H^{d_1}(C_1)^\vee \otimes_{k G} H^{d_1}(C_2).$$
	in $\mathrm{Ho}^b(k [G \times L^{\mathrm{opp}}])$. By assumption we have $H^{d_1}(C_1) \cong H^{d_1}(C_2)$. Moreover, the bimodule $H^{d_1}(C_1)$ induces a Morita equivalence between $k L f$ and $k G e$ by \cite[Section 10.2.3]{Rouquier3}. From this we can conclude that $k L f \cong C_1^\vee \otimes_{k G} C_2$ in $\mathrm{Ho}^b(k[ L \times L^{\mathrm{opp}}])$. Therefore, we have 
	$$C_1\cong C_1 \otimes_{k L} kL f \cong C_1 \otimes_{k L} C_1^\vee \otimes_{k G} C_2 \cong C_2$$
	in $\mathrm{Ho}^b(k[ G \times L^{\mathrm{opp}}])$. 
\end{proof}


\begin{corollary}\label{independencegodement}
	Let $\G$ be a connected reductive group, $s\in (\G^\ast)^{F^\ast}$ semisimple of $\ell'$-order and $\Levi^\ast$ the minimal Levi subgroup of $\G^\ast$ with $\C(s) \subseteq \Levi^\ast$. Let $\Para=\Levi \U$ and $\mathbf{Q}=\Levi \V$ be two parabolic subgroups of $\G$ with Levi subgroup $\Levi$. Then we have $$G\Gamma_c(\Y_{\U},k) e_s^{\Levi^F}\cong G \Gamma_c(\Y_{\V},k) e_s^{\Levi^F}[\mathrm{dim}(\Y_\V^\G) - \mathrm{dim}(\Y_\U^\G)]$$
	in $\mathrm{Ho}^b(k \G^F\otimes_{k }(k \Levi^F)^{\mathrm{opp}})$
	if
	$$\mathrm{dim}(\Y_\U^\G)-\mathrm{dim}(\Y_\V^\G)=\mathrm{dim}(\Y_{\C_{\U}(Q)}^{\C_\G(Q)})-\mathrm{dim}(\Y_{\C_{\V}(Q)}^{\C_{\G}(Q)})$$
	for all $\ell$-subgroups $Q$ of $\Levi^F$.
\end{corollary}

\begin{proof}
	By Theorem \ref{Boro2} the complex $G \Gamma_c( \Y_\U,k) e_s^{\Levi^F}$ induces a splendid Rickard equivalence between $k \Levi^F e_s^{\Levi^F}$ and $k \G^F e_s^{\G^F}$. Its cohomology is concentrated in degree $\mathrm{dim}( \Y_\U)$. Moreover, the cohomology of 
	$$\mathrm{Br}_{\Delta Q}(G \Gamma_c( \Y^\G_\U,k) e_s^{\Levi^F}) \cong G \Gamma_c( \Y^{\C_{\G}(Q)}_{\C_\U(Q)},k) \br_Q(e_s^{\Levi^F})$$
	is concentrated in degree $\mathrm{dim}(\Y_{\C_{\U}(Q)}^{\C_\G(Q)})$. The same holds for the variety $\Y_\V^\G$. By Theorem \ref{independence}, $H_c^{\dim}(\Y_{\U},\Lambda) e_s^{\Levi^F}\cong H_c^{\dim}(\Y_{\V},\Lambda) e_s^{\Levi^F}$. Hence the statement of the corollary is an immediate consequence of Lemma \ref{homotopy}.
\end{proof} 

We don't know when the condition of Corollary \ref{independencegodement} holds in general. The following example is an application of Corollary  \ref{independencegodement}.

\begin{example}\label{iGR}
	Suppose that $\sigma: \G \to \G$ is a bijective endomorphism with $\sigma \circ F= F \circ \sigma$ and stabilizing $\Levi$ and $e_s^{\Levi^F}$. Suppose that a Sylow $\ell$-subgroup $D$ of $\Levi^F$ is cyclic. Up to changing $\sigma$ by inner automorphisms of $\Levi^F$ we may assume that $D$ is $\sigma$-stable. Hence, for any subgroup $Q$ of $D$ we have $\sigma(Q)=Q$. It follows that 
	$$\mathrm{dim}(\Y_{\C_{\sigma(\U)}(Q)}^{\C_\G(Q)})=\mathrm{dim}(\sigma(\Y_{\C_\U(Q)}^{\C_\G(Q)}))=\mathrm{dim}(\Y_{\C_\U(Q)}^{\C_\G(Q)}).$$
	From this and Corollary \ref{independencegodement} we conclude that $$G\Gamma_c(\Y_{\sigma(\U)},k) e_s^{\Levi^F} \cong G \Gamma_c(\Y_{\U},k) e_s^{\Levi^F} $$
	in $\mathrm{Ho}^b(k \G^F\otimes_{k}(k \Levi^F)^{\mathrm{opp}})$
	Therefore, by Lemma \ref{sigma} we have
	$${}^\sigma (G\Gamma_c(\Y_{\U},k) e_s^{\Levi^F})^{\sigma} \cong G \Gamma_c(\Y_{\U},k) e_s^{\Levi^F}$$
	in $\mathrm{Ho}^b(k \G^F\otimes_{k }(k \Levi^F)^{\mathrm{opp}})$.
\end{example}

%
%


\subsection{Comparing Rickard and Morita equivalences}

Let $\G$ be a not necessarily connected reductive group and $\Levi$ be an $F$-stable Levi subgroup of $\G$ with Levi decomposition $\Para=\Levi \ltimes \U$. Let $\mathcal{X}$ be a $(\G^\circ,\Levi^\circ)$-regular series of $(\Levi,F)$. Denote by $\mathcal{Y}$ the unique series of $(\G,F)$ containing $\mathcal{X}$. We denote $d:= \mathrm{dim}(\Y_\U^\G)$. We recall the following important result:

\begin{proposition}\label{Endnoncon}
	We have
	$$\mathrm{End}_{k \G^F}^\bullet((G\Gamma_c(\Y_\U^\G,k) e_\mathcal{X}){^\mathrm{red}})\cong \mathrm{End}_{D^b(k \G^F)}((G\Gamma_c(\Y_\U^\G,k) e_\mathcal{X}){^\mathrm{red}}) $$
	in $\mathrm{Ho}^b(k [\Levi^F \times (\Levi^F)^\mathrm{opp}])$.
\end{proposition}

\begin{proof}
	This is proved in Step 1 of the proof of \cite[Theorem 7.6]{Dat}. Note that the assumption $\G$ is connected is not needed in this step of the proof.
	%
\end{proof}

\begin{proposition}\label{equiv}
	Let $b$ be a block of $\Lambda \G^F e_{\mathcal{Y}}$ and $c$ be a block of $\Lambda \Levi^F e_{\mathcal{X}}$. Denote $C:= b G\Gamma_c(\Y_\U^\G,\Lambda)^{\mathrm{red}} c$ and $d:= \mathrm{dim}(\Y_\U^\G)$. Then the complex $C$ induces a splendid Rickard equivalence between $ \Lambda \G^F b$ and $ \Lambda \Levi^F c$ if and only if $H^d(C)$ induces a Morita equivalence between $\Lambda \G^F b$ and $ \Lambda \Levi^F c$. 
\end{proposition}

\begin{proof}
	Let us first assume that $ \Lambda=k$.
 	By Proposition \ref{Endnoncon} we have
	$$\mathrm{End}_{k \G^F}^\bullet(C)\cong \mathrm{End}_{D^b(k \G^F)}(C). $$
	Since $C$ is a complex of projective $k \G^F$-modules we have $\mathrm{End}_{D^b(k \G^F)}(C) \cong H^0(\mathrm{End}_{k \G^F}^\bullet(C))$ and as the cohomology of $C$ is concentrated in degree $d$, we deduce that $H^0(\mathrm{End}_{k \G^F}^\bullet(C)) \cong \mathrm{End}_{k \G^F}(H^d(C))$. Therefore, $\mathrm{End}_{k \G^F}^\bullet(C) \cong  \mathrm{End}_{k \G^F}(H^d(C))$ in $\mathrm{Ho}^b(k [ \Levi^F \times (\Levi^F)^\mathrm{opp} ])$. By \cite[Theorem 2.1]{Rickard} it follows that $C$ induces a Rickard equivalence if and only if $H^d(C)$ induces a Morita equivalence.
	
	Let us now assume that $\Lambda=\mathcal{O}$. If $H^d(C)$ induces a Morita equivalence between $\mathcal{O} \G^F b$ and $\mathcal{O} \Levi^F c$ then $H^d(C \otimes_\mathcal{O} k) \cong H^d(C) \otimes_{\mathcal{O}} k$ induces a Morita equivalence between $k \G^F b$ and $k \Levi^F c$. Using the result for the case $\Lambda =k$ shows that the complex $C \otimes_\mathcal{O} k$ induces a splendid Rickard equivalence between $k \G^F b$ and $k \Levi^F c$. Thus, by the proof of \cite[Theorem 5.2]{Rickard} the complex $C$ induces a splendid Rickard equivalence between $\mathcal{O} \G^F b$ and $\mathcal{O} \Levi^F c$.
	On the other hand, if the complex $C$ induces a Rickard equivalence then it follows by \cite[Section 10.2.3]{Rouquier3} that $H^d(C)$ induces a Morita equivalence.
\end{proof}


\subsection{Morita equivalences for local subgroups}

In this section we give some applications of Proposition \ref{equiv}. We keep the notation of the previous section and assume additionally that the rational series $\mathcal{X}$ of $(\Levi,F)$ is almost $(\G^\circ,\Levi^\circ)$-superregular, see Definition \ref{almost superregular}.

\begin{corollary}\label{Rickarddiscon}
	Suppose that $\mathrm{N}_{\Levi^F}(e_s^{(\Levi^\circ)^F}) (\G^\circ)^F = \mathrm{N}_{\G^F}(e_s^{(\G^\circ)^F})$. Then the complex $G \Gamma_c(\Y^\G_{\U},\Lambda) e_\mathcal{X}$ induces a splendid Rickard equivalence between $\Lambda \G^F e_{\mathcal{Y}}$ and $\Lambda \Levi^F e_{\mathcal{X}}$.
\end{corollary}
\begin{proof}
	
	By Lemma \ref{independencedisconnected} the bimodule $H_c^{\dim(\Y^{\G}_{\U})}(\Y^{\G}_{\U},\Lambda) e_\mathcal{X}$ induces a Morita equivalence between $\Lambda \G^F e_{\mathcal{Y}}$ and $\Lambda \Levi^F e_{\mathcal{X}}$. Write $e_{\mathcal{X}}=c_1 + \dots + c_r$ as a sum of block idempotents. Then there exists a decomposition $e_{\mathcal{Y}}=b_1 + \dots + b_r$ into block idempotents such that $H_c^{\dim(\Y^{\G}_{\U})}(\Y^{\G}_{\U},\Lambda) c_i$ induces a Morita equivalence between $\Lambda  \G^F b_i$ and $\Lambda \Levi^F c_i$. Set $C:=G \Gamma_c(\Y^\G_{\U},\Lambda)^{\mathrm{red}} e_\mathcal{X}$. It follows from Proposition \ref{equiv} that the complex $b_i C c_i$ induces a splendid Rickard equivalence between $\Lambda  \G^F b_i$ and $\Lambda \Levi^F c_i$. Consequently, the complex $\bigoplus_{i=1}^r b_i C c_i$ induces a splendid Rickard equivalence between $\Lambda \G^F e_{\mathcal{Y}}$ and $\Lambda \Levi^F e_{\mathcal{X}}$.
	
	For $j \neq i$ consider the complex $X:=b_i C c_j$. By the proof of Proposition \ref{equiv} we have 
	$$X^\vee \otimes_{k \G^F} X \cong \mathrm{End}_{k \G^F}^\bullet(X) \cong  \mathrm{End}_{k \G^F}(H^{\mathrm{dim}(\Y_\U^\G)}(X)) \cong 0$$ in $\mathrm{Ho}^b(k[\Levi^F \times (\Levi^F)^{\mathrm{opp}}])$. By the proof of \cite[Theorem 2.1]{Rickard}, the complex $X$ is a direct summand of $X \otimes_{k \Levi^F} X^\vee \otimes_{k \G^F} X$. This shows that $X=b_i C c_j \cong  0$ in $\mathrm{Ho}^b(k[\G^F \times (\Levi^F)^{\mathrm{opp}}])$ for $j \neq i$. Hence, the complex $G \Gamma_c(\Y^\G_{\U},\Lambda) e_\mathcal{X}$ induces a splendid Rickard equivalence between $\Lambda \G^F e_{\mathcal{Y}}$ and $\Lambda \Levi^F e_{\mathcal{X}}$.
\end{proof}

Suppose that we are in the situation of Corollary \ref{Rickarddiscon}. Let $b$ be a block of $\Lambda \G^F e_{\mathcal{Y}}$ corresponding to the block $c$ of $\Lambda \Levi^F e_{\mathcal{X}}$ under the splendid Rickard equivalence between $\Lambda \G^F e_{\mathcal{Y}}$ and $\Lambda \Levi^F e_{\mathcal{X}}$ given by $C:= G \Gamma_c(\Y_\U^\G,\Lambda)^{\mathrm{red}} e_{\mathcal{X}}$. Let $(Q,c_Q)$ be a $c$-Brauer pair and $(Q,b_Q)$ be the unique $b$-Brauer pair of $k \mathrm{C}_{\G^F}(Q)$ such that the complex $b_Q \Br_{\Delta Q}(C) c_Q \cong \Br_{\Delta Q}(C) c_Q$ induces a Rickard equivalence between $k \C_{\G^F}(Q) b_Q$ and $k \C_{\Levi^F}(Q) c_Q$, see Proposition \ref{splendidloc}.

%

%


The following proposition is yet another application of Proposition \ref{equiv}.

\begin{proposition}\label{centralizerMorita}
	Suppose that $\mathrm{N}_{\Levi^F}(e_s^{(\Levi^\circ)^F}) (\G^\circ)^F = \mathrm{N}_{\G^F}(e_s^{(\G^\circ)^F})$. Then the bimodule $ H_c^{\mathrm{dim}}(\Y_{\C_\U(Q)}^{\C_\G(Q)},\Lambda)c_Q$ induces a Morita equivalence between the blocks $\Lambda\C_{\Levi^F}(Q) c_Q$ and $\Lambda\C_{\G^F}(Q) b_Q$.
\end{proposition}

\begin{proof}
	Recall that $c$ is a block of $\Lambda \Levi^F e_{\mathcal{X}}$. Since $(Q,c_Q)$ is a $c$-subpair we have $\br^{\Levi^F}_Q(c) c_Q=c_Q$. Thus, there exists some rational series $\mathcal{X}' \in (i_Q^\Levi)^{-1}(\mathcal{X})$ such that $c_Q$ is a block of $k \C_{\Levi^F}(Q) e_{\mathcal{X}'}$, see Lemma \ref{brauerseries}. Let $\mathcal{Y}'$ be the unique rational series of $(\C_\G(Q),F)$ containing $\mathcal{X}'$, see Lemma \ref{basicseries}(a).
	
	Since the complex $\Br_{\Delta Q}(C) c_Q \cong G \Gamma_c(\Y_{\C_\U(Q)}^{\C_\G(Q)},k) c_Q$ induces a Rickard equivalence between $k\C_{\Levi^F}(Q) c_Q$ and $k\C_{\G^F}(Q) b_Q$ it follows by Lemma \ref{basicseries}(c) that $b_Q$ is a block of $k \mathrm{C}_{\G^F}(Q) e_{\mathcal{Y'}}$. By the remarks following Lemma \ref{BrauerGodement} the complex $G \Gamma_c(\Y_{\C_\U(Q)}^{\C_\G(Q)},\mathcal{O}) c_Q$ is a splendid complex of $\mathcal{O} \C_{\G^F}(Q)$-$\mathcal{O} \C_{\Levi^F}(Q)$-bimodules, which is a lift to $\mathcal{O}$ of $G \Gamma_c(\Y_{\C_\U(Q)}^{\C_\G(Q)},k) c_Q$. By the proof of \cite[Theorem 5.2]{Rickard} it follows that $G \Gamma_c(\Y_{\C_\U(Q)}^{\C_\G(Q)},\mathcal{O}) c_Q$ induces a Rickard equivalence between $\mathcal{O}\C_{\Levi^F}(Q) c_Q$ and $\mathcal{O}\C_{\G^F}(Q) b_Q$. It therefore follows by Proposition \ref{equiv} that $H_c^{\mathrm{dim}}(\Y_{\C_\U(Q)}^{\C_\G(Q)},\Lambda)c_Q$ induces a Morita equivalence between $\Lambda\C_{\Levi^F}(Q) c_Q$ and $\Lambda\C_{\G^F}(Q) b_Q$.
\end{proof}


%
In the following we consider the subgroup 
$$\mathcal{D}:=\{(x,y)\in \operatorname{N}_{\G^F}(Q)\times \operatorname{N}_{ \Levi^F}(Q)^{\operatorname{opp}} \mid x \C_{\G^F}(Q)=y^{-1} \C_{\G^F}(Q) \}$$
of $\operatorname{N}_{\G^F}(Q)\times \operatorname{N}_{\Levi^F}(Q)^{\operatorname{opp}}$.

In addition, we let $B_Q:=\Tr^{\mathrm{N}_{\G^F}(Q)}_{\mathrm{N}_{\G^F}(Q,b_Q)}(b_Q)$ and $C_Q:=\Tr^{\mathrm{N}_{\Levi^F}(Q)}_{\mathrm{N}_{\Levi^F}(Q,c_Q)}(c_Q)$. The following can be seen as a geometric version of Proposition \ref{normalizerderived}.

\begin{theorem}\label{Morita2}
	Suppose that $\mathrm{N}_{\Levi^F}(e_s^{(\Levi^\circ)^F}) (\G^\circ)^F = \mathrm{N}_{\G^F}(e_s^{(\G^\circ)^F})$. Then the bimodule $ H_c^{\mathrm{dim}}(\Y_{\C_\U(Q)}^{\mathrm{N}_\G(Q)},\Lambda) C_Q$
	induces a Morita equivalence between $ \Lambda \mathrm{N}_{\Levi^F}(Q) C_Q$ and $\Lambda \mathrm{N}_{\G^F}(Q) B_Q $.
\end{theorem}

\begin{proof}
	Corollary \ref{factor} shows that the factor groups $\operatorname{N}_{\Levi^F}(Q,c_Q)/ \C_{\Levi^F}(Q)$ and $\operatorname{N}_{\G^F}(Q,b_Q)/\C_{\G^F}(Q)$ are isomorphic via the inclusion $\mathrm{N}_{\Levi^F}(Q) \subseteq \mathrm{N}_{\G^F}(Q)$. Moreover, by Proposition \ref{splendidloc} we deduce ${}^x b_Q  H_c^{\mathrm{dim}}(\Y_{\C_\U(Q)}^{\mathrm{C}_\G(Q)},\Lambda) c_Q=0$ for all $x \in \mathrm{N}_{\G^F}(Q) \setminus \operatorname{N}_{\G^F}(Q,b_Q)$. The bimodule $ H_c^{\mathrm{dim}}(\Y_{\C_\U(Q)}^{\C_\G(Q)},\Lambda) c_Q$ induces by Proposition \ref{centralizerMorita} a Morita equivalence between the blocks $ \Lambda \mathrm{C}_{\Levi^F}(Q) c_Q$ and $\Lambda \mathrm{C}_{\G^F}(Q) b_Q$.
	
	Recall from Example \ref{normcent} that $\mathrm{N}_{\G}(Q)$ is a reductive group. Moreover, $\mathrm{N}_{\Para}(Q)$ is a parabolic subgroup of $\mathrm{N}_{\G}(Q)$ with Levi decomposition $\mathrm{N}_{\Para}(Q)=\mathrm{N}_{\Levi}(Q) \ltimes \C_\U(Q)$. Note that $\C_{\G}(Q)$ is a normal subgroup of $\mathrm{N}_{\G}(Q)$ and we have a Levi decomposition $\mathrm{C}_{\Para}(Q)=\mathrm{C}_{\Levi}(Q) \ltimes \C_\U(Q)$ in $\C_{\G}(Q)$, see Example \ref{normcent}.
	%
	%
	By Corollary \ref{Kunnethcoro} it follows that the bimodule
	$H_c^{\mathrm{dim}}(\Y_{\C_\U(Q)}^{\C_\G(Q)},\Lambda)$ has a natural $\mathcal{D}$-action and we have an isomorphism
	$$\Ind_{\mathcal{D}}^{\mathrm{N}_{\G^F}(Q) \times \mathrm{N}_{\Levi^F}(Q)^{\opp}} H_c^{\mathrm{dim}}(\Y_{\C_\U(Q)}^{\C_\G(Q)},\Lambda) \cong H_c^{\mathrm{dim}}(\Y_{\C_\U(Q)}^{\mathrm{N}_\G(Q)},\Lambda).$$
	By Lemma \ref{better version} it follows that the bimodule $ H_c^{\mathrm{dim}}(\Y_{\C_\U(Q)}^{\mathrm{N}_\G(Q)},\Lambda) C_Q$
	induces a Morita equivalence between $ \Lambda \mathrm{N}_{\Levi^F}(Q) C_Q$ and $\Lambda \mathrm{N}_{\G^F}(Q) B_Q $.
\end{proof}

\section{Automorphisms of quasi-simple groups of Lie type}

%

\subsection{Automorphisms of simple groups of Lie type}\label{automorphism}

We briefly recall the classification of automorphisms of finite simple groups of Lie type.
Let $\G$ be a simple algebraic group of simply connected type with Frobenius $F: \G \to \G$ such that $\G^F$ is perfect. Fix a maximal torus $\T_0$ and a Borel subgroup $\B_0$ of $\G$ containing $\T_0$. We let $\Phi$ be the root system relative to $\T_0$ and $\Delta $ be the base of $\Phi$ relative to $\T_0 \subseteq \B_0$. For every $\alpha \in \Phi$ we fix a one-parameter subgroup $\mathbf{x}_\alpha: (\overline{\mathbb{F}_p},+) \to \G$. We consider the following bijective morphisms of $\G$ as introduced before \cite[Theorem 30]{Steinberg1968}:
\begin{itemize}
\item
The \textit{field endomorphism} $\phi_0: \G \to \G, \, \x_{\alpha}(t) \mapsto \x_{\alpha}(t^p)$ for every $t \in \overline{\mathbb{F}_p}$ and $\alpha \in \Phi$.
\item 
For any angle preserving permutation $\gamma$ of the simple roots $\Delta$ we consider the \textit{graph endomorphism} $\gamma: \G \to \G$ given by 
\begin{equation*}
\gamma(\x_{\alpha}(t):=
\begin{cases}
\x_{\gamma(\alpha)}(t) & \text{ if } \alpha \text{ is long or all roots have the same length}\\
\x_{\gamma(\alpha)}(t^p)       & \text{ if } \alpha \text{ is short}
\end{cases}
\end{equation*}
for every $t \in \overline{\mathbb{F}_p}$ and $\alpha \in \pm \Delta$.
\end{itemize}

For any fixed prime power $q=p^f$ of $p$ and a graph automorphism $\gamma$ we consider the Frobenius endomorphism $F= \phi_0^f \gamma: \G \to \G$. Note that any Frobenius endomorphism of $\G$ is (up to inner automorphisms of $\G$) of this form by \cite[Theorem 22.5]{MT}. We say that $(\G,F)$ is \textit{untwisted} if $\gamma$ is the identity and \textit{twisted} otherwise. We let $\iota: \G \hookrightarrow \tilde{\G}$ be a regular embedding as in \cite[Section 2.B]{MS}. In particular there, exist suitable extensions of $\phi$ and $\gamma$ to $\Gtilde$ (denoted by the same letter) such that all relevant relations are preserved. The automorphisms of $\G^F$ obtained by conjugation with $\tilde{\G}^F$ are called \textit{diagonal automorphisms} of $\G^F$.

To avoid cumbersome notation we will use the same letter for bijective morphisms of $\G$ commuting with $F$ and their restriction to $\G^F$:

\begin{notation}\label{restriction}
Let $\sigma:\G \to \G$ be a bijective morphism of algebraic groups with $\sigma \circ F= F \circ \sigma$. Then we also denote by $\sigma:\G^F \to \G^F$ the automorphism of $\G^F$ obtained by restricting $\sigma$ to $\G^F$. In particular, the expression $\G^F \rtimes \langle \sigma \rangle$ always denotes the semidirect product of finite groups obtained by letting $\sigma: \G^F \to \G^F$ act on $\G^F$.
\end{notation}


%
%

%

\begin{corollary}\label{explicit}
		Let $\G$ be a simple algebraic group of simply connected type not of type $D_4$ with Frobenius $F$. Let $s \in (\G^\ast)^{F^\ast}$ be a semisimple element of $\ell'$-order. There exists a Frobenius endomorphism $F_0 : \tilde{\G} \to \tilde{\G}$ with $F_0^{r}=F$ for some positive integer $r$ and a bijective morphism $\sigma: \tilde{\G} \to \tilde{\G}$ such that the image of $\tilde{\G}^F \rtimes \langle F_0,\sigma \rangle$ in $\mathrm{Out}(\G^F)$ is the stabilizer of $e_s^{\G^F}$ in $\mathrm{Out}(\G^F)$.
\end{corollary}

\begin{proof}

Let $\mathrm{Diag}_{\G^F}$ be the image of the set of diagonal automorphisms in $\mathrm{Out}(\G^F)$. The stabilizer of $e_s^{\G^F}$ in $\mathrm{Out}(\G^F)$ contains $\mathrm{Diag}_{\G^F}$, see e.g. \cite[Lemma 7.4]{Dat}. Suppose that $(\G,F)$ is twisted. In this case, the group $\mathrm{Out}(\G^F)/ \mathrm{Diag}_{\G^F}$ is cyclic and the statement of the corollary can be deduced from this. Now suppose that $(\G,F)$ is untwisted. Let $\gamma: \G \to \G$ be a non-trivial graph endomorphism (if it exists). Then the classification of automorphisms of simple groups of Lie type (see \cite[Theorem 30 and 36]{Steinberg1968}) shows that $\mathrm{Out}(\G^F)/ \mathrm{Diag}_{\G^F} \cong \langle \gamma , \phi_0 \rangle \cong C_t \times C_m$, where $t \leq 3$. Thus, every subgroup of $\mathrm{Out}(\G^F)/ \mathrm{Diag}_{\G^F}$ is either cyclic or isomorphic to $\langle \gamma \rangle \times \langle F_0 \rangle$, where $F_0= \phi_0^i: \G \to \G$ for some $i$ with $i \mid f$.
\end{proof}

\subsection{Automorphisms and stabilizers of idempotents}\label{auto and stabilizer}

Let $\G$ be a simple algebraic group of simply connected type not of type $D_4$ with Frobenius $F$. Let $s \in (\G^\ast)^{F^\ast}$ be a semisimple element of $\ell'$-order. Recall that by Corollary \ref{explicit} there exists a Frobenius endomorphism $F_0 : \Gtilde \to \Gtilde $ with $F_0^{r}=F$ for some positive integer $r$ and a bijective morphism $\sigma: \Gtilde \to \Gtilde$ commuting with $F_0$ such that $\tilde{\G}^F \rtimes \langle F_0,\sigma \rangle$ is the stabilizer of $e_s^{\G^F}$ in $\mathrm{Out}(\G^F)$.

Let $\Levi^\ast$ be a minimal Levi subgroup of $\G^\ast$ containing $\mathrm{C}(s):=\mathrm{C}_{(\G^\ast)^{F^\ast}}(s) \mathrm{C}^\circ_{\G^\ast}(s)$. Our first aim in this section is to show that we can assume that $F_0^\ast(s)=s$ and therefore $F_0^\ast(\Levi^\ast)=\Levi^\ast$.

To achieve this, we need to recall some general observations on conjugacy classes of semisimple elements. For this suppose that $\mathbf{H}$ is an $F^\ast$-stable connected reductive subgroup of $ \G^\ast$ and let $t \in \mathbf{H}^{F^\ast}$. Then we denote $A_{\mathbf{H}}(t)= \mathrm{C}_{\mathbf{H}}(t) / \mathrm{C}_{\mathbf{H}}^\circ(t)$ and write $\mathcal{T}_{\mathbf{H},t}$ for the set of $\mathbf{H}^{F^\ast}$-conjugacy classes of elements of $\mathbf{H}^{F^\ast}$ which are $\mathbf{H}$-conjugate to $t$. Furthermore, we write $H^1(F^\ast,A_{\mathbf{H}}(t))$ for the set of $F^\ast$-conjugacy classes of $A_{\mathbf{H}}(t)$. 

\begin{lemma}\label{bla}
Under the notation as above the set $\mathcal{T}_{\mathbf{H},t}$ is in natural bijection with $H^1(F^\ast,A_{\mathbf{H}}(t))$.
\end{lemma}

\begin{proof}
This is proved in \cite[Proposition 3.21]{DM}. We recall the construction of this bijection. If $y \in \mathbf{H}^{F^\ast}$ is $\mathbf{H}$-conjugate to $t$ then there exists some $x\in \mathbf{H}$ such that ${}^x y =t$. Since both $y$ and $t$ are $F^\ast$-stable it follows that $x^{-1} F^\ast(x) \in \mathrm{C}_{\mathbf{H}}(t)$. Then one defines the map $\mathcal{T}_{\mathbf{H},t} \to H^1(F^\ast,A_{\mathbf{H}}(t))$ by sending the conjugacy class of $y$ to the $F^ \ast$-conjugacy class of $x^{-1} F^\ast(x)$.
\end{proof}

\begin{lemma}\label{new}
	The set $H^1(F^\ast,A_{\mathbf{H}}(t))$ of $F^\ast$-conjugacy classes of $A_{\mathbf{H}}(t)$ is in bijection with $A_{\mathbf{H}}(t)^{F^\ast}$.
\end{lemma}

\begin{proof}
	We observe that $A:=A_{\mathbf{H}}(t)$ is a finite abelian group, see \cite[Proposition 13.16]{MarcBook}. In particular, the set $H^1(F^\ast,A)$ can be identified with the quotient $A_{F^\ast}$ of $A$ by the subgroup generated by all elements of the form $a^{-1}F^\ast(a)$ with $a \in A$. Let $m$ be the order of $F^\ast$ on the finite group $A$. Then it follows that the norm map $N_{(F^\ast)^m/F^\ast}: A_{F^\ast} \to A^{F^\ast}, \, a \mapsto \prod_{i=0}^{m-1} (F^\ast)^i(a)$ is a bijection.
\end{proof}

\begin{corollary}\label{equivariantconjugacy}
Assume additionally that $\mathbf{H}$ and $t$ are $F_0^\ast$-stable.
Then the natural bijection $\mathcal{T}_{\mathbf{H},t} \to H^1(F^\ast,A_{\mathbf{H}}(t))$ is $F_0^\ast$-equivariant.
\end{corollary}

\begin{proof}
We show that the bijection constructed in the proof of Lemma \ref{bla} is $F_0^\ast$-equivariant. Let $y \in \mathbf{H}^{F^\ast}$ be $\mathbf{H}$-conjugate to $t$ and $x\in \mathbf{H}$ such that ${}^x y =t$. Since $t$ is $F_0^\ast$-stable it follows that ${}^{F_0^\ast(x)} F_0^\ast(y) =t$. In particular, $F^\ast_0(y)$ is $\mathbf{H}$-conjugate to $t$. It follows that $F_0^\ast(x^{-1} F^\ast(x))$ is the image of the conjugacy class of $F^\ast_0(y)$ under the map $\mathcal{T}_{\mathbf{H},t} \to H^1(F^\ast,A_{\mathbf{H}}(t))$. The claim follows.  
\end{proof}

Recall that in the beginning of \ref{auto and stabilizer} we assumed that $F_0$ stabilizes $e_s^{\G^F}$. Thus, by duality, the $(\G^\ast)^{F^\ast}$-conjugacy class of $s$ is $F_0^\ast$-stable. Using the equivariant bijection from Corollary \ref{equivariantconjugacy} we can show the following:

\begin{lemma}\label{mayassume}
Let $\Levi^\ast$ be the minimal Levi subgroup of $\G^\ast$ containing $\C(s)$. Then we may assume that $\Levi^\ast$ is $F_0$-stable and that the $(\Levi^\ast)^{F^\ast}$-conjugacy class of $s$ is $F_0^\ast$-stable.
\end{lemma}

\begin{proof}
	Since the $\G^\ast$-conjugacy class of $s$ is $F_0^\ast$-stable it follows by Lang's theorem that there exists some $t \in \G^\ast$ which is $F_0^\ast$-fixed and $\G^\ast$-conjugate to $s$, see the remarks following \cite[Proposition 13.12]{DM}
Let $\mathbf{K}^\ast$ be the unique minimal Levi subgroup of $\G^\ast$ containing $\mathrm{C}_{\G^\ast,F^\ast}(t)$. Since $F_0^\ast(t)=t$ it follows that $\mathbf{K}^\ast$ is $F_0^\ast$-stable. Moreover, we have $\mathrm{C}_{\mathbf{K}^\ast,F^\ast}(t)=\mathrm{C}_{\G^\ast,F^\ast}(t)$ and therefore $A_{\mathbf{K}^\ast}(t)^{F^\ast}=A_{\G^\ast}(t)^{F^\ast}$. Let $ \mathcal{T}_{\mathbf{K}^\ast,s}  \to \mathcal{T}_{\G^\ast,s}$ be the natural map which sends the $(\mathbf{K}^\ast)^{F^\ast}$-conjugacy class to its $(\G^\ast)^{F^\ast}$-conjugacy class. Furthermore, the map $A_{\mathbf{K}^\ast}(t) \hookrightarrow A_{\mathbf{G}^\ast}(t)$ yields an injective map $H^1(F^\ast, A_{\mathbf{K}^\ast}(t)) \hookrightarrow H^1(F^\ast, A_{\G^\ast}(t))$. These maps fit into the following commutative square:
	
	\begin{center}
\begin{tikzpicture}
  \matrix (m) [matrix of math nodes,row sep=3em,column sep=4em,minimum width=2em] {
  
   \mathcal{T}_{\mathbf{K}^\ast,t}  & H^1(F^\ast, A_{\mathbf{K}^\ast}(t))  \\
    \mathcal{T}_{\G^\ast,t}  & H^1(F^\ast, A_{\G^\ast}(t))
     \\};
\path[-stealth]
(m-1-2) edge node [left] {} (m-2-2)
(m-1-1) edge node [above] {$\cong$} (m-1-2)
(m-2-1) edge node [above] {$\cong$} (m-2-2)
(m-1-1) edge node [left] {} (m-2-1);

\end{tikzpicture}
\end{center}
Using Lemma \ref{new} we deduce that the right vertical map is a bijection. From this i follows that we have a bijection between the set of $(\mathbf{K}^\ast)^{F^ \ast}$-conjugacy classes of elements which are $\mathbf{K}^\ast$-conjugate to $t$ and the set of $(\G^\ast)^{F^\ast}$-conjugacy classes of elements which are $\G^\ast$-conjugate to $t$.

In fact, since this bijection is $F_0^\ast$-equivariant by Corollary \ref{equivariantconjugacy} we deduce that it maps $F_0^\ast$-stable classes to $F_0^\ast$-stable classes. As the $(\G^\ast)^{F^\ast}$-conjugacy class of $s$ is $F_0^ \ast$-stable there exists some $x \in (\mathbf{K}^\ast)^{F^\ast}$ which is $(\G^\ast)^{F^\ast}$-conjugate to $s$. Moreover, the $(\mathbf{K}^{\ast})^{F^\ast}$-conjugacy class of $x$ is $F_0^\ast$-stable. Thus, the assumptions of the lemma are satisfied if we replace $s$ by $x$ and $\Levi^\ast$ by $\mathbf{K}^\ast$.
\end{proof}


\begin{remark}\label{observelevidual}
Let $\T_0$ be an $F_0$-stable maximal torus of $\G$. Suppose that the triple $(\G^\ast,\T_0^\ast,F_0^\ast)$ is in duality with $(\G,\T_0,F_0)$. We obtain a bijection between the $\G^{F_0}$-conjugacy classes of $F_0$-stable Levi subgroups of $\G$ and the $(\G^\ast)^{F_0^\ast}$-conjugacy classes of $F_0^\ast$-stable Levi subgroups of $\G^\ast$, see Lemma \ref{bijLevidual}. Moreover, it follows that $(\G^\ast,\T_0^\ast,F^\ast)$ is in duality with $(\G,\T_0,F)$, where $F^\ast:=(F_0^\ast)^r$. This in turn gives a  bijection between the $\G^{F}$-conjugacy classes of $F$-stable Levi subgroups of $\G$ and the $(\G^\ast)^{F^\ast}$-conjugacy classes of $F^\ast$-stable Levi subgroups of $\G^\ast$. This bijection is compatible with the aforementioned bijection, i.e. if $\Levi$ is an $F_0$-stable Levi subgroup of $\G$ in duality with the $F_0^\ast$-stable Levi subgroup $\Levi^\ast$ of $\G^\ast$ then $\Levi$ and $\Levi^\ast$ are $F$- respectively $F^\ast$-stable and correspond to each other under the bijection induced by the duality between $(\G,\T_0,F)$ and $(\G^\ast,\T_0^\ast,F^\ast)$. However, note that two $F_0$-stable Levi subgroups can be $\G^F$-conjugate but not $\G^{F_0}$-conjugate.
\end{remark}

For the remainder of this section we may assume by Lemma \ref{mayassume} that $\Levi^\ast$ is $F_0$-stable and that the $\Levi^\ast$-conjugacy class of $s$ is $F_0$-stable. Hence, by Remark \ref{observelevidual} there exists an $F_0$-stable Levi subgroup $\Levi$ of $\G$ which is in duality with $\Levi^\ast$ under the duality between $(\G,F_0)$ and $(\G^\ast,F_0^\ast)$.

Recall that we assume that $\sigma: \G \to \G$ is a bijective morphism with $F_0 \circ \sigma=\sigma \circ F_0$ which stabilizes the idempotent $e_s^{\G^F}$. By duality, we therefore obtain a bijective morphism $\sigma^\ast: \G^\ast \to \G^\ast$ of algebraic groups with $\sigma^\ast \circ F_0^\ast=  F_0^\ast \circ \sigma^\ast$. Recall that $\sigma^\ast: \G^\ast \to \G^\ast$ is only unique up to inner automorphisms of $(\G^\ast)^{F_0^\ast}$, see Remark \ref{dualisogeny}.

\begin{proposition}\label{comparing stabilizers}
	There exists some $x \in \G^{F_0}$ such that $x \sigma$ normalizes $\Levi$ and ${}^{x\sigma} e_s^{\Levi^F}= e_s^{\Levi^F}$.
\end{proposition}

\begin{proof}
Since the $(\G^\ast)^{F^\ast}$-conjugacy class of $s$ is $\sigma^\ast$-stable it follows that the $(\G^\ast)^{F^\ast}$-conjugacy class of $\Levi^\ast$ is $\sigma^\ast$-stable. By Corollary \ref{dualitysigma} it follows that the $\G^F$-conjugacy class of $\Levi$ is $\sigma$-stable. We can therefore find $g \in \G^F$ and $h \in (\G^\ast)^{F^\ast}$ such that $\sigma_0:=g \sigma$ stabilizes $\Levi$ and $\sigma_0^\ast:=h \sigma^\ast$ stabilizes $\Levi^\ast$. Moreover, we can choose $g$ and $h$ with the additional property that $\sigma_0|_\Levi$ and $\sigma_0^\ast|_{\Levi^\ast}$ are in duality with each other. Since the $(\G^\ast)^{F^\ast}$-conjugacy class of $s$ is $\sigma_0$-stable there exists some $n^\ast \in (\G^\ast)^{F^\ast}$ such that $\sigma_0^\ast(s)={}^{n^\ast} s$. Since $\Levi^\ast$ is $\sigma_0^\ast$-stable it follows that $n^\ast \in \mathrm{N}_{(\G^\ast)^{F^\ast}}(\Levi^\ast)$. Let $n \in \mathrm{N}_{\G^F}(\Levi)$ be an element corresponding to $n^\ast$ under the canonical isomorphism
	$$\mathrm{N}_{\G^F}(\Levi) / \Levi^F \cong \mathrm{N}_{(\G^\ast)^{F^\ast}}(\Levi^\ast) / (\Levi^\ast)^{F^\ast}$$
induced by duality. Using the remarks following Lemma \ref{sigma} we obtain 
$$ \sigma_0(e_s^{\Levi^F})=e_{\sigma_0^\ast(s)}^{\Levi^F}=e_{{}^{n^\ast} s}^{\Levi^F}= {}^{n} e_s^{\Levi^F}.$$
Therefore $y:=g^{-1} n$ satisfies ${}^{y \sigma} \Levi=\Levi$ and ${}^{y \sigma} e_s^{\Levi^F}=e_s^{\Levi^F}$. Since $\Levi$ and $e_s^{\Levi^F}$ are $F_0$-stable we conclude that $F_0(y) y^{-1} \in \mathrm{N}_{\G^F}(\Levi,e_s^{\Levi^F})$. By the remarks before Theorem \ref{Boro2} we have $\mathrm{N}_{\G^F}(\Levi,e_s^{\Levi^F})=\Levi^F$. Therefore by applying Lang's theorem to $F_0 : \Levi \to \Levi$ there exists some $l \in \Levi$ such that $F_0(y) y^{-1}=F_0(l) l^{-1}$. This implies that $x:=l^{-1} g \in \G^{F_0}$ and $x \sigma$ normalizes $\Levi$ and $e_s^{\Levi^F}$.
\end{proof}

The next proposition describes the set of automorphisms stabilizing the idempotent $e_s^{\G^F}$ in a nice way:

\begin{proposition}\label{propo}
Let $\G$ be a simple algebraic group of simply connected type not of type $D_4$ with Frobenius $F$. Let $s \in (\G^\ast)^{F^\ast}$ be a semisimple element of $\ell'$-order. There exists a Frobenius endomorphism $F_0 : \Gtilde \to \Gtilde $ with $F_0^{r}=F$ for some positive integer $r$ and a bijective morphism $\sigma: \Gtilde \to \Gtilde$ such that $\mathcal{A}=\langle F_0, \sigma \rangle \subseteq \mathrm{Aut}(\tilde{\G}^F)$ satisfies:
\begin{enumerate}[label=(\alph*)]
\item $F_0 \circ \sigma = \sigma \circ F_0$ as morphisms of $\tilde{\G}$.
\item The image of $\tilde{\G}^F \rtimes \mathcal{A}$ in $\mathrm{Out}(\G^F)$ is the stabilizer of $e_s^{\G^F}$ in $\mathrm{Out}(\G^F)$.
\item There exists a Levi subgroup $\Levi$ of $\G$ in duality with $\Levi^\ast$ such that $\mathcal{A}$ stabilizes $\Levi$ and $e_s^{\Levi^F}$.
\end{enumerate}
\end{proposition}

\begin{proof}
The existence of the bijective morphisms $F_0:\Gtilde \to \Gtilde$ and $\sigma: \Gtilde \to \Gtilde$ satisfying properties (a) and (b) follows from Corollary \ref{explicit}. By  Lemma \ref{mayassume} we can assume that the Levi subgroup $\Levi$ is $F_0$-stable. By Proposition \ref{comparing stabilizers} there exists $x \in \G^{F_0}$ such that $x \sigma$ stabilizes $\Levi$ and $e_s^{\Levi^F}$. The result now follows by replacing $\sigma$ with $x \sigma$.
\end{proof}

%
Note that the natural map $\mathrm{Aut}(\G^F) \twoheadrightarrow \mathrm{Out}(\G^F)$ does not necessarily induce an isomorphism of $\mathcal{A}$ with its image in $\mathrm{Out}(\G^F)$. This is essentially the case since we need to replace the automorphism $\sigma$ by $\sigma x$ in the proof of Proposition \ref{propo}.
%

\section{Extending the Morita equivalence}

Suppose that we are in the situation of Corollary \ref{explicit}. Then by Theorem \ref{Boro2} the bimodule $H_c^{\operatorname{dim}}(\Y_\U^\G,\mathcal{O})e_s^{\Levi^F}$ induces a Morita equivalence between $\mathcal{O} \Levi^F e_s^{\Levi^F}$ and $\mathcal{O} \G^F e_s^{\G^F}$. In Proposition \ref{propo} we have constructed a group $\mathcal{A}$ such that $\tilde{\G}^F \rtimes \mathcal{A}$ generates the stabilizer of $e_s^{\G^F}$ in $\mathrm{Out}(\G^F)$. The aim of this section is to show that the Morita equivalence induced by $H^{\mathrm{dim}}_{c}( \Y_\U^\G, \mathcal{O})e_s^{\Levi^F}$ lifts (under mild assumptions on $\ell$) to a Morita equivalence between $\mathcal{O} \tilde{\Levi}^F \mathcal{A} e_s^{\Levi^F}$ and $\mathcal{O} \tilde{\G}^F \mathcal{A} e_s^{\G^F}$.
%
%
%

\subsection{Disconnected reductive groups and Morita equivalences}\label{stable pair}

Let $\G$ be a connected reductive group which is not a torus with Frobenius $F: \G \to \G$ and $\iota: \G \hookrightarrow \Gtilde$ be a regular embedding. Consider an algebraic automorphism $\tau: \Gtilde \to \Gtilde$ satisfying $\tau \circ F= F \circ \tau$ and $\tau(\G)=\G$. By the discussion at the beginning of \cite[Paragraph 2.4]{Spaeth} it follows that the automorphism $\tau$ is uniquely determined up to powers of $F$ by its restriction to $\tilde{\G}^F$. Consequently, the automorphisms $\tau$ and its restriction to $\tilde{\G}^F$ have the same order.
 As in Example \ref{graph} we consider the not necessarily connected reductive group $\Gtilde \rtimes \langle \tau \rangle$.

Let $\G^\ast$ be in duality with $\G$. Fix a semisimple element $s \in (\G^\ast)^{F^\ast}$ of $\ell'$-order and let $\Levi^\ast$ be a Levi subgroup with $\mathrm{C}^\circ_{\G^\ast}(s) \subseteq \Levi^\ast$.
Let $\Para$ be a parabolic subgroup of $\G$ with Levi decomposition $\Para= \Levi \ltimes \U$. We have a Levi decomposition $\tilde{\Para}= \Ltilde \ltimes \U$ in $\tilde{\G}$, where $\tilde{\Para}:= \Para \Z(\Gtilde)$ and $\Ltilde:= \Levi \Z(\Gtilde)$. Suppose that the parabolic subgroup $\Para$ is $\tau$-stable. Then $\hat{\Para}:=\tilde{\Para} \langle \tau \rangle$ is a parabolic subgroup of $\hat{\G}:=\Gtilde \rtimes \langle \tau \rangle $ with Levi decomposition $\hat{\Para}= \hat{\Levi}  \ltimes  \U $, where $\hat{\Levi}:=\tilde{\Levi} \langle \tau \rangle$, see Example \ref{graph}. The Frobenius endomorphism $F$ extends to a Frobenius endomorphism of $\Gtilde \rtimes \langle \tau \rangle$ by defining 
$$F: \Gtilde \rtimes \langle \tau \rangle \to \Gtilde \rtimes \langle \tau \rangle ,\ g \tau^i \mapsto F(g) \tau^i, \, (i\geq 0).$$
Since $\tau$ and its restriction to $\Gtilde^F$ have the same order we have an isomorphism 
$$(\Gtilde \rtimes \langle \tau \rangle)^F \cong \Gtilde^F \rtimes \langle \tau|_{\Gtilde^F} \rangle.$$
In the following, we will as in Notation \ref{restriction} use the same letter $\tau$ for the automorphism $\tau: \Gtilde \to \Gtilde$ and its restriction to $\Gtilde^F$.

We let $\mathcal{Y}$ and $\mathcal{X}$ be the rational series of $(\G \langle \tau \rangle,F)$ and $(\Levi \langle \tau \rangle,F)$ which contain the rational series associated to the semisimple element $s$ of $\G$ and $\Levi$ respectively.

Let $\sigma: \Gtilde \to \Gtilde$ be a bijective morphism of algebraic groups commuting with the action of $\tau$ and $F$. Then $\sigma$ extends to a bijective morphism $\sigma: \Gtilde\rtimes \langle \tau \rangle \to \Gtilde \rtimes \langle \tau \rangle ,\ g \tau^i \mapsto \sigma(g) \tau^i, \, (i \geq 0).$
With this notation we have the following:

\begin{lemma}\label{easycase}
The bimodule $H_c^{\mathrm{dim}}(\Y_\U^{\G },\Lambda) e_{\mathcal{X}}$ is endowed with a natural $(\G^F \times (\Levi^F)^{\mathrm{opp}}) \Delta(\tilde{\Levi}^F  \langle \tau \rangle)$-action. If $\Levi$ is $\sigma$-stable then we have
$$H_c^{\mathrm{dim}}(\Y_\U^{\G },\Lambda ) e_{\mathcal{X}} \cong {}^\sigma H_c^{\mathrm{dim}}(\Y_\U^{\G },\Lambda)^{\sigma} \sigma(e_{\mathcal{X}} )$$
as $\Lambda[ (\G^F \times (\Levi^F)^{\mathrm{opp}}) \Delta(\tilde{\Levi}^F \langle \tau \rangle )]$-bimodules.
\end{lemma}

\begin{proof}
This follows from Lemma \ref{local} and Lemma \ref{sigma}.
\end{proof}

Suppose that $b$ is a block of $\Lambda\G^F e_s^{\G^F}$ corresponding to a block $c$ of $\Lambda \Levi^F e_s^{\Levi^F}$ under the Morita equivalence induced by $H^{\mathrm{dim}}_c(\Y_\U^\G,\Lambda) e_s^{\Levi^F}$. Let $(Q,c_Q)$ be a $c$-Brauer pair and $(Q,b_Q)$ the corresponding $b$-Brauer pair such that $b_Q H_c^{\mathrm{dim}}(\Y_{\C_\U(Q)}^{\C_\G(Q)},k)=H_c^{\mathrm{dim}}(\Y_{\C_\U(Q)}^{\C_\G(Q)},k) c_Q$. As usually, we define $B_Q:=\Tr_{\mathrm{N}_{\G^F}(Q,b_Q)}^{\mathrm{N}_{\G^F}(Q)}(b_Q)$ and $C_Q:=\Tr_{\mathrm{N}_{\Levi^F}(Q,c_Q)}^{\mathrm{N}_{\Levi^F}(Q)}(c_Q)$. We will now provide a local version of Lemma \ref{easycase}.
The technical difficulty is to keep track of the diagonal actions.

\begin{theorem}\label{independencelocal}
Assume that $\hat{\mathbf{Q}}= \hat{\Levi} \ltimes \mathbf{V}$ is a parabolic subgroup of $\hat{\G}$ with Levi subgroup $\hat{\Levi}$. Then we have
	$$ H_c^{\mathrm{dim}}(\Y_{\C_\U(Q)}^{\mathrm{N}_{\G}(Q)},\Lambda) C_Q \cong  H_c^{\mathrm{dim}}(\Y_{\C_{\mathbf{V}}(Q)}^{\mathrm{N}_{\G}(Q)},\Lambda) C_Q$$
as $\Lambda [(\mathrm{N}_{\G^F}(Q) \times \mathrm{N}_{\Levi^F}(Q)^{\mathrm{opp}}) \Delta (\mathrm{N}_{\hat{\Levi}^F}(Q,C_Q))]$-modules.

\end{theorem}

\begin{proof}
Firstly, recall that $\mathrm{N}_{\hat{\G}}(Q)$ is a reductive group with closed connected normal subgroup $\mathrm{C}^\circ_{\hat{\G}}(Q)$, see Example \ref{normcent}. We have a Levi decomposition $\mathrm{N}_{\hat{\Para}}(Q)=\mathrm{N}_{\hat{\Levi}}(Q) \ltimes \mathrm{C}_\U(Q)$ in $\mathrm{N}_{\hat{\G}}(Q)$. Furthermore, $\C^\circ_\G(Q)$ is a closed normal subgroup of $\mathrm{N}_{\hat{\G}}(Q)$ and we have a Levi decomposition $\C^\circ_\Para(Q)=\C^\circ_\Levi(Q) \ltimes  \C_\U(Q)$ in the connected reductive group $\C^\circ_\G(Q)$, see also Example \ref{normcent}. In addition, we have 
$$\mathrm{N}_{\hat{\Levi}}(Q) \cap \mathrm{C}^\circ_{\G}(Q) = \mathrm{C}_{\hat{\Levi} \cap \G}(Q) \cap \mathrm{C}^\circ_{\G}(Q)= \mathrm{C}_{\Levi}(Q) \cap \mathrm{C}^\circ_{\G}(Q)= \mathrm{C}_{\Levi}^\circ(Q)$$
and similarly $\mathrm{N}_{\hat{\Para}}(Q) \cap \mathrm{C}^\circ_\G(Q)=\mathrm{C}^\circ_{\Para}(Q)$. This shows that we are in the situation of Section \ref{PropertiesofDL}.

Recall that since $(Q,c_Q)$ is a $c$-subpair we have $\br^{\Levi^F}_Q(c) c_Q=c_Q$ and by Lemma \ref{brauerseries} there exists $\mathcal{X}' \in (i_Q^\Levi)^{-1}(\mathcal{X})$ such that $c_Q$ is a block of $k \C_{\Levi^F}(Q) e_{\mathcal{X}'}$.

%

Let $\mathcal{Z}$ be a rational series of $\mathrm{C}_\Levi^\circ(Q)$ contained in $\mathcal{X}'$. By Lemma \ref{brauerseries} and Remark \ref{almost} we obtain that the rational series $\mathcal{Z}$ is $(\mathrm{C}_\G^\circ(Q),\mathrm{C}_{\Levi}^\circ(Q))$-regular. By the proof of Lemma \ref{local} we thus obtain an isomorphism
$$H_c^{\mathrm{dim}}(\Y_{\C_\U(Q)}^{\mathrm{C}^\circ_{\G}(Q)},\Lambda) e_{\mathcal{Z}} \cong H_c^{\mathrm{dim}}(\Y_{\C_\V(Q)}^{\mathrm{C}^\circ_{\G}(Q)},\Lambda) e_{\mathcal{Z}}$$
of $\Lambda[(\mathrm{C}^\circ_{\G}(Q)^F \times (\mathrm{C}^\circ_{\Levi}(Q)^F)^{\mathrm{opp}}) \Delta \mathrm{N}_{\hat{\Levi}^F}(Q,e_\mathcal{Z})]$-modules.
Moreover we have $e_{\mathcal{X}'}=\mathrm{Tr}_{(\C_{\Levi}^\circ(Q))^F}^{\C_{\Levi^F}(Q)}(e_\mathcal{Z})$ by Lemma \ref{discratl}, which implies that $\mathrm{N}_{\hat{\Levi}^F}(Q,e_{\mathcal{X}'})=\mathrm{C}_{\Levi^F}(Q) \mathrm{N}_{\hat{\Levi}^F}(Q,e_{\mathcal{Z}})$. We obtain an isomorphism
$$H_c^{\mathrm{dim}}(\Y_{\C_\U(Q)}^{\mathrm{C}_{\G}(Q)},\Lambda) e_{\mathcal{X}'} \cong  H_c^{\mathrm{dim}}(\Y_{\C_\V(Q)}^{\mathrm{C}_{\G}(Q)},\Lambda) e_{\mathcal{X}'}$$
of $\Lambda[(\mathrm{C}_{\G^F}(Q) \times (\mathrm{C}_{\Levi^F}(Q))^{\mathrm{opp}}) \Delta \mathrm{N}_{\hat{\Levi}^F}(Q,e_{\mathcal{X}'})]$-modules. Since $c_Q$ is a block of $k \mathrm{C}_{\Levi^F}(Q) e_{\mathcal{X}'}$ we obtain, by truncating to $c_Q$, an isomorphism
$H_c^{\mathrm{dim}}(\Y_{\C_\U(Q)}^{\mathrm{C}_{\G}(Q)},\Lambda) c_Q \cong  H_c^{\mathrm{dim}}(\Y_{\C_\V(Q)}^{\mathrm{C}_{\G}(Q)},\Lambda) c_Q$
of $\Lambda[\mathrm{C}_{\G^F}(Q) \times (\mathrm{C}_{\Levi^F}(Q))^{\mathrm{opp}} \Delta \mathrm{N}_{\hat{\Levi}^F}(Q,c_Q)]$-modules. Applying Lemma \ref{geometricgeneralversion2} yields an isomorphism
$$ H_c^{\mathrm{dim}}(\Y_{\C_\U(Q)}^{\mathrm{N}_{\G}(Q)},\Lambda) C_Q \cong  H_c^{\mathrm{dim}}(\Y_{\C_{\mathbf{V}}(Q)}^{\mathrm{N}_{\G}(Q)},\Lambda) C_Q$$
of $\Lambda[(\mathrm{N}_{\G^F}(Q) \times \mathrm{N}_{\Levi^F}(Q)^{\mathrm{opp}}) \Delta \mathrm{N}_{\hat{\Levi}^F}(Q,C_Q)]$-modules.
\end{proof}

The previous statements rely on the parabolic subgroup $\Para$ being $\tau$-stable. In the following, we will use an idea from \cite{Digne} to reduce the situation outlined at the beginning of this section to this case.

\subsection{Restriction of scalars for Deligne--Lusztig varieties}\label{ShintaniDL}


Let $\G$ be a reductive group with Frobenius endomorphism $F_0:\G \to \G$. For an integer $r$ we let $F:=F_0^r: \G \to \G$. We consider the reductive group $\underline{\G}= \G^r$ with Frobenius endomorphism $F_0 \times \dots \times F_0: \underline{\G} \to \underline{\G}$ which we also denote by $F_0$. We consider the permutation 
$$\tau: \underline{\G} \to \underline{\G}$$
given by $\tau(g_1,\dots,g_r)=(g_2,\dots,g_r,g_1)$.
Consider the projection onto the first component
$$\mathrm{pr}:\underline{\G} \to \G, \, (g_1,\dots,g_r)\mapsto g_1.$$
The restriction of $\mathrm{pr}$ to $\underline{\G}^{F_0 \tau}$ induces an isomorphism
$$\mathrm{pr}:\underline{\G}^{F_0 \tau} \to \G^F$$
of finite groups with inverse map given by $\mathrm{pr}^{-1}(g)=(g,F_0^{r-1}(g),\dots,F_0(g))$ for $ g \in \G^F$.

For any subset $\mathbf{H}$ of $\G$ we set
$$\underline{\mathbf{H}}:=\mathbf{H} \times F_0^{r-1}(\mathbf{H}) \times \dots \times F_0(\mathbf{H}).$$
Note that if $\mathbf{H}$ is $F$-stable then $\underline{\mathbf{H}}$ is $\tau F_0$-stable and the projection map $\mathrm{pr}: \underline{\mathbf{H}} \to \mathbf{H}$ induces an isomorphism $\underline{\mathbf{H}}^{\tau F_0} \cong \mathbf{H}^F$. Conversely, one easily sees that any $\tau F_0$-stable subset of $\underline{\G}$ is of the form $\underline{\mathbf{H}}$ for some $F$-stable subset $\mathbf{H}$ of $\G$. 


Let $\Levi$ be an $F$-stable Levi subgroup of $\G$ and $\Para$ a parabolic subgroup of $\G$ with Levi decomposition $\Para=\Levi \ltimes \U$. Then $\underline{\Para}$ is a parabolic subgroup of $\underline{\G}$ with Levi decomposition $\underline{\Para}=\underline{\Levi} \ltimes \underline{\U}$ such that $\tau F_0(\underline{\Levi})=\underline{\Levi}$. We can therefore consider the Deligne--Lusztig variety $\Y_{\underline{\U}}^{\underline{\G},F_0 \tau}$ which is a $\underline{\G}^{F_0 \tau} \times (\underline{\Levi}^{F_0 \tau})^{\mathrm{opp}}$-variety. Under the isomorphism $\G^F\cong \underline{\G}^{F_0 \tau}$ we will in the following regard it as a $\G^F \times (\Levi^F)^{\mathrm{opp}}$-variety. 

The following proposition is proved in \cite[Proposition 3.1]{Digne} under the additional assumptions that $\G$ is connected and that the Levi subgroup $\Levi$ is $F_0$-stable. Here, we give a complete proof of this proposition and thereby show that these assumptions are superfluous.
\begin{proposition}\label{ShintaniDLV}
Let $\Levi$ be an $F$-stable Levi subgroup of $\G$ and $\Para$ a parabolic subgroup of $\G$ with Levi decomposition $\Para=\Levi \ltimes \U$. Then the projection $\mathrm{pr}:\underline{\G} \to \G$ onto the first coordinate defines an isomorphism
$$ \Y_{\underline{\U}}^{\underline{\G},\tau F_0} \cong  \Y_\U^{\G,F} $$
of varieties which is $\G^F \times (\Levi^F)^{\mathrm{opp}}$-equivariant.
\end{proposition}
\begin{proof}
Let $\underline{g}=(g_1,\dots, g_r) \in \underline{\G}$. Then $\underline{g} \underline{\U} \in \Y_{\underline{\U}}^{\G,\tau F_0}$ if and only if $$\underline{g}^{-1} (\tau F_0)(\underline{g}) \in \underline{\U} (\tau F_0)(\underline{\U})=\U F(\U) \times F_0^{r-1}(\U) \times \dots \times F_0(\U).$$
This is equivalent to $g_1^{-1} F_0(g_2) \in \U F(\U)$ and $g_i^{-1} F_0(g_{i+1}) \in F_0^{r+1-i}(\U)$ for all $i=2,\dots,r$ (where $g_{r+1}:=g_1$). Therefore, $\underline{g} \underline{\U} \in \Y_{\underline{\U}}^{\underline{\G},\tau F_0}$ if and only if $$\underline{g}\underline{\U}=(g_1,F_0^{r-1}(g_1), \dots, F_0(g_1)) \underline{\U} \text{ and } g_1^{-1} F(g_1) \in  \U F(\U).$$
Hence, an element $\underline{g} \underline{\U} \in \Y_{\underline{\U}}^{\underline{\G},\tau F_0}$ is uniquely determined by its first component $g_1 \U \in \Y_\U^\G$ and each element of $\Y_{\underline{\U}}^{\underline{\G}}$ arises from an element $g_1 \U \in \Y_\U^\G$. This shows that $\mathrm{pr}:\underline{\G} \to \G$ induces an isomorphism $$ \Y_\U^{\G,F} \cong \Y_{\underline{\U}}^{\underline{\G},\tau F_0},$$
which is clearly $\G^F \times (\Levi^F)^{\mathrm{opp}}$-equivariant.
\end{proof}

%
%
%

We will now provide a local version of Proposition \ref{ShintaniDLV}. Let $Q$ be a finite $F$-stable solvable $p'$-subgroup of $\Levi$. Recall from Example \ref{normcent} that the normalizer $\mathrm{N}_\G(Q)$ is a reductive group and $\mathrm{N}_\Para(Q)$ is a parabolic subgroup of $\mathrm{N}_\G(Q)$ with Levi decomposition $\mathrm{N}_\Para(Q)=\mathrm{N}_{\Levi}(Q) \ltimes \C_\U(Q)$. We denote
$$\underline{Q}:= Q \times F_0^{r-1}(Q) \times \dots \times F_0(Q)$$
and observe that $\underline{Q}$ is a finite $\tau F_0$-stable solvable $p'$-subgroup of $\underline{\Levi}$. By the same argument as before, we see that $\mathrm{N}_{\underline{\G}}(\underline{Q})$ is a reductive group with parabolic subgroup $\mathrm{N}_{\underline{\Para}}(\underline{Q})$ and Levi decomposition $\mathrm{N}_{\underline{\Para}}(\underline{Q})=\mathrm{N}_{\underline{\Levi}}(\underline{Q}) \ltimes \C_{\underline{\U}}(\underline{Q})$. We can therefore consider the Deligne--Lusztig variety $\Y_{\C_{\underline{\U}}(\underline{Q})}^{\mathrm{N}_{\underline{\G}}(\underline{Q}),\tau F_0}$ which is a $\mathrm{N}_{\underline{\G}^{F_0 \tau}}(\underline{Q}) \times \mathrm{N}_{\underline{\Levi}^{F_0 \tau}}( \underline{Q})^{\mathrm{opp}}$-variety. Under the isomorphism $\mathrm{pr}:\underline{\G}^{F_0 \tau} \to \G^F$ we may consider it as a $\mathrm{N}_{\G^F}(Q) \times \mathrm{N}_{\Levi^F}(Q)^{\mathrm{opp}}$-variety. Thus, we can apply Proposition \ref{ShintaniDLV} in this situation and obtain the following corollary:

\begin{corollary}\label{ShintaniDLVCoro}
Suppose that we are in the situation of Proposition \ref{ShintaniDLV} and assume that $Q$ is a finite $F$-stable solvable $p'$-group of $\Levi$. Then the projection map $\mathrm{pr}: \mathrm{N}_{\underline{\G}}(\underline{Q}) \to \mathrm{N}_{\G}(Q)$ induces an isomorphism
$$\Y_{\C_{\underline{\U}}(\underline{Q})}^{\mathrm{N}_{\underline{\G}}(\underline{Q}),\tau F_0} \cong \Y^{\mathrm{N}_\G(Q),F}_{\C_{\U}(Q)}$$
of varieties which is $\mathrm{N}_{\G^F}(Q) \times \mathrm{N}_{\Levi^F}(Q)^{\mathrm{opp}}$-equivariant.
\end{corollary}

\subsection{Restriction of scalars and Jordan decomposition of characters}\label{SDJD}

In the following section we use ideas from \cite[Corollary 3.5]{Digne} and apply them to our set-up.

In addition to the notation of \ref{ShintaniDL} we assume that $\G$ is connected. We let $(\G^\ast,F_0^\ast)$ be a pair in duality with $(\G,F)$. We consider the $r$-fold product $\underline{\G^*}:=(\G^*)^r$ of the dual group $\G^\ast$ endowed with the Frobenius endomorphism $F_0^\ast:=F_0^* \times \dots \times F_0^*: \underline{\G^*} \to \underline{\G^*}$. Moreover, let
$$\tau^\ast: \underline{\G^\ast} \to \underline{\G^\ast}, \, (g_1,\dots,g_r) \mapsto (g_r,g_1\dots,g_{r-1}).$$
Again $\mathrm{pr}: \underline{\G^\ast} \to \G^\ast$ denotes the projection onto the first coordinate. 

\begin{corollary}\label{ShintaniLusztig}
For any semisimple $\ell'$-element $\underline{s} \in (\underline{\G}^\ast)^{\tau^\ast F_0^\ast}$ we have $e_{\mathrm{pr}(\underline{s})}^{\G^F}=e_{\underline{s}}^{\underline{\G}^{F_0 \tau}}$ considered as idempotents of $\Lambda\G^F$ under the isomorphism $\Lambda \underline{\G}^{\tau F_0} \cong \Lambda \G^F$ given by $\mathrm{pr}$.
\end{corollary}

\begin{proof}
	Note that $e_{\mathrm{pr}(\underline{s})}^{\G^F}$ is the idempotent associated to $\mathcal{E}_\ell(\G^F,\mathrm{pr}(\underline{s}))$ and $e_{\underline{s}}^{\underline{\G}^{F_0 \tau}}$ is the idempotent associated to $\mathcal{E}_{\ell}(\underline{\G}^{\tau F_0},\underline{s})$. Thus it is clearly sufficient to show that $\mathcal{E}(\G^F,\mathrm{pr}(\underline{x}))=\mathcal{E}(\underline{\G}^{F_0 \tau},\underline{x})$ for any semisimple $\underline{x} \in (\underline{\G^\ast})^{\tau^\ast F_0^\ast}$. This follows from \cite[Corollary 8.8]{Taylor2} or \cite[Proposition 5.11]{PhD}
\end{proof}

Let $\Levi^\ast$ be an $F_0^\ast$-stable Levi subgroup of $\G^\ast$ with $\C^\circ_{\G^\ast}(s) \subseteq \Levi^\ast$. Suppose that $\Levi$ is an $F_0$-stable Levi subgroup of $\G$ in duality with $\Levi^\ast$.
We denote by 
$$\underline{\sigma}=\sigma \times \dots \times \sigma: \underline{\G} \to \underline{\G}$$
the induced map on $\underline{\G}$ which commutes with the action of $\tau F_0$ and its restriction
$$\underline{\sigma}:\underline{\G}^{F_0 \tau} \to \underline{\G}^{F_0 \tau}. $$
Observe that if the isogeny ${\sigma}^\ast$ is dual to $\sigma$ then the isogeny $ \underline{\sigma}^\ast \tau^\ast$ is dual to $\tau \underline{\sigma}$.

%


We consider the unipotent radical $\underline{\U}':=\U^r$ of the parabolic subgroup $\underline{\Para}'=\Para^r$ of $\underline{\G}$. Note that we have a Levi decomposition $\underline{\Para}'= \underline{\Levi} \ltimes \underline{\U}'$ in $\underline{\G}$ and the parabolic subgroup $\underline{\Para}'$ is $\tau$-stable. The following is an application of Lemma \ref{easycase}:

\begin{lemma}\label{preparation}
Suppose that the idempotent $e_s^{\Levi^F}$ is $\langle F_0, \sigma \rangle$-stable. Then $H^{\mathrm{dim}}_c(\Y_{\underline{\U'}}^{\underline{\G}, \tau F_0}) e_{\underline{s}}^{\underline{\Levi}^{\tau F_0}}$ is endowed with a natural $\Lambda [(\underline{\G}^{\tau F_0} \times  (\underline{\Levi}^{\tau F_0})^{\mathrm{opp}}) \Delta( \underline{\Ltilde}^{\tau F_0} \langle \tau \rangle)]$-structure.
Moreover, $H^{\mathrm{dim}}_c(\Y_{\underline{\U'}}^{\underline{\G}, \tau F_0}) e_{\underline{s}}^{\underline{\Levi}^{ \tau F_0}}$ is $(\underline{\sigma},\underline{\sigma}^{-1})$-invariant as $\Lambda[ (\underline{\G}^{\tau F_0} \times  (\underline{\Levi}^{\tau F_0})^{\mathrm{opp}}) \Delta( \underline{\Ltilde}^{\tau F_0} \langle \tau \rangle )]$-module.
\end{lemma}

\begin{proof}
The pair $(\underline{\Levi},\underline{\Para}')$ is $\tau$-stable and $\underline{\Levi}$ is $\underline{\sigma}$-stable. We have $\mathrm{pr}(\underline{s})=s$. We note that $\sigma\in \mathrm{Aut}(\G^F)$ corresponds to $\underline{\sigma} \in \mathrm{Aut}(\underline{\G}^{\tau F_0})$ under the isomorphism $\mathrm{pr}: \underline{\G}^{\tau F_0} \to \G^F$. Furthermore,
the automorphism $F_0: \G^F \to \G^F$ corresponds under the identification of $\G^F$ with $\underline{\G}^{F_0 \tau}$ via the projection map $\mathrm{pr}$ to the automorphism $\tau^{-1}: \underline{\G}^{F_0 \tau} \to \underline{\G}^{F_0 \tau}$. Since $e_s^{\Levi^F}$ is $\langle F_0, \sigma \rangle$-stable it therefore follows that $e_{\underline{s}}^{\underline{\Levi}^{\tau F_0}}$ is $\langle \tau, \underline{\sigma} \rangle$-stable. Moreover, $\Levi^\ast$ is $F_0^\ast$-stable by assumption, so we obtain
$$\C^\circ_{\underline{\G}^\ast}(\underline{s})=\C^\circ_{\G^\ast}(s) \times \dots \times \C^\circ_{\G^\ast}(F_0^{r-1}(s)) \subseteq \underline{\Levi}^\ast.$$
%
%
We conclude that Lemma \ref{easycase} applies which gives the claim of the lemma.
\end{proof}

\begin{proposition}\label{field}
Suppose that $\Levi$ and $e_s^{\Levi^F}$ are $\langle F_0, \sigma \rangle$-stable. Then the bimodule $H^{\mathrm{dim}}_c(\Y_\U,\Lambda)e_s^{\Levi^F}$ can be equipped with a $\Lambda[(\G^F \times (\Levi^F)^{\operatorname{opp}}) \Delta (\Ltilde^F \langle F_0 \rangle) ]$-module structure with which it is $(\sigma,\sigma^{-1})$-stable.
\end{proposition}

\begin{proof}
By Theorem \ref{independence}, we have an isomorphism
$$H^{\mathrm{dim}}_c(\Y_{\underline{\U}}^{\underline{\G}, \tau F_0}) e_{\underline{s}}^{\underline{\Levi}^{\tau F_0}} \cong H^{\mathrm{dim}}_c(\Y_{\underline{\U'}}^{\underline{\G}, \tau F_0}) e_{\underline{s}}^{\underline{\Levi}^{\tau F_0}}$$
of $\Lambda [(\underline{\G}^{\tau F_0} \times (\underline{\Levi}^{\tau F_0})^{\mathrm{opp}}) \Delta( \underline{\Ltilde}^F)]$-modules.

It follows by Lemma \ref{preparation} that the bimodule $H^{\mathrm{dim}}_c(\Y_{\underline{\U'}}^{\underline{\G}, \tau F_0}) e_{\underline{s}}^{\underline{\Levi}^{\tau F_0}}$ has a $\Lambda[( \underline{\G}^{\tau F_0} \times  (\underline{\Levi}^{\tau F_0})^{\mathrm{opp}}) \Delta( \underline{\Ltilde}^{\tau F_0} \langle \tau \rangle)]$-structure with which it is $(\underline{\sigma},\underline{\sigma}^{-1})$-stable.
%
By Proposition \ref{ShintaniDLV} and Corollary \ref{ShintaniLusztig} the bimodule $H^{\mathrm{dim}}_c(\Y_{\underline{\U}}^{\underline{\G}, \tau F_0}) e_{\underline{s}}^{\underline{\Levi}^{\tau F_0}}$ is isomorphic to $H^{\mathrm{dim}}_c(\Y_\U^\G,\Lambda) e_s^{\Levi^F}$ as $\Lambda[(\G^F \times (\Levi^F)^{\mathrm{opp}}) \Delta (\tilde{\Levi}^F)]$-modules. As noted above, the group isomorphism $\sigma\in \mathrm{Aut}(\G^F)$ corresponds to $\underline{\sigma} \in \mathrm{Aut}(\underline{\G}^{\tau F_0})$ under the isomorphism $\mathrm{pr}: \underline{\G}^{\tau F_0} \to \G^F$. The automorphism $\tau\in \mathrm{Aut}(\underline{\G}^{\tau F_0})$ corresponds to $F_0^{-1} \in \mathrm{Aut}(\G^F)$. From this we can, by transport of structure, endow the bimodule $H^{\mathrm{dim}}_c(\Y_\U,\Lambda) e_s^{\Levi^F}$ with a $\Lambda[(\G^F \times \Levi^{F^{\operatorname{opp}}}) \Delta (\Ltilde^F \langle F_0 \rangle )]$-module structure with which it is $(\sigma,\sigma^{-1})$-stable.
\end{proof}

In the following, we denote $\mathcal{A}=\langle  \sigma, F_0 \rangle \subseteq \mathrm{Aut}(\tilde{\G}^F)$ and $ \mathcal{D}= (\G^F \times (\Levi^F)^{\operatorname{opp}}) \Delta(\tilde{\Levi}^F \mathcal{A})$. Furthermore, let $A \in \{K, \mathcal{O},k \}$.

\begin{theorem}\label{Morita lift}
Suppose that $\Levi$ and $e_s^{\Levi^F}$ are $\mathcal{A}$-stable. Assume that $\mathrm{C}_{\G^\ast}(s) \subseteq \Levi^\ast$ and the order of $\sigma: \Gtilde^F \to \Gtilde^F$ is invertible in $A$. Then $H^{\mathrm{dim}}_c(\Y_\U,A) e_s^{\Levi^F}$ extends to an $A \mathcal{D}$-module $M$. Moreover, the bimodule $\mathrm{Ind}_{ \mathcal{D} }^{\tilde{\G}^F \mathcal{A} \times (\tilde{\Levi}^F \mathcal{A})^{\mathrm{opp}}} (M)$ induces a Morita equivalence between $A \tilde{\Levi}^F \mathcal{A} e_s^{\Levi^F}$ and $A \tilde{\G}^F \mathcal{A} e_s^{\G^F}$.

\end{theorem}

\begin{proof} The existence of the extension $M$ follows from Proposition \ref{field} and \cite[Lemma 10.2.13]{Rouquier3}. The bimodule $H^{\mathrm{dim}}_c(\Y_\U^\G,A) e_s^{\Levi^F}$ induces a Morita equivalence between $A \G^F e_s^{\G^F}$ and $A \Levi^F e_s^{\Levi^F}$. Since $e_s^{\Levi^F}$ is $\mathcal{A}$-invariant we conclude that the assumptions of Theorem \ref{lifting} are satisfied. From this it follows that $\mathrm{Ind}_{ \mathcal{D} }^{\tilde{\G}^F \mathcal{A} \times (\tilde{\Levi}^F \mathcal{A})^{\mathrm{opp}}} (M)$ gives a Morita equivalence between $A \tilde{\Levi}^F \mathcal{A} e_s^{\Levi^F}$ and $A \tilde{\G}^F \mathcal{A} e_s^{\G^F}$.
\end{proof}

%
We remark the following consequence of Theorem \ref{Morita lift} which is important for characgter theoretic applications.

\begin{corollary}
In the situation of Theorem \ref{Morita lift} we have the following commutative square:
\begin{center}
\begin{tikzpicture}
  \matrix (m) [matrix of math nodes,row sep=3em,column sep=4em,minimum width=2em] {
  
    G_0(A \Ltilde^F \langle F_0, \sigma \rangle e_s^{\Levi^F} )  & G_0(A \Gtilde^F \langle F_0, \sigma \rangle  e_s^{\G^F})  \\
      G_0(A\Levi^F e_s^{\Levi^F}) & G_0(A \G^F e_s^{\G^F}) 
     \\};
\path[-stealth]
(m-1-2) edge node [left] {$\mathrm{Res}$} (m-2-2)
(m-1-1) edge node [above] {$[M \otimes -]$} (m-1-2)
(m-2-1) edge node [above] {$(-1)^{\mathrm{dim}(\Y_\U^\G)} R_\Levi^{\G}$} (m-2-2)
(m-1-1) edge node [left] {$\mathrm{Res}$} (m-2-1);

\end{tikzpicture}
\end{center}
\end{corollary}
\begin{proof}
This has been discussed in Remark \ref{characters}(a).
\end{proof}

\subsection{Jordan decomposition for local subgroups}

We keep the assumptions of Section \ref{SDJD}. The aim of this section is to obtain a local version of Theorem \ref{Morita lift}. We will essentially use the same strategy of Section \ref{SDJD} to prove this local version. However, we need to adapt some of the arguments.

Recall that the projection map $\mathrm{pr}:\underline{\G}^{\tau F_0} \to \G^F$ onto the first coordinate induces an isomorphism of groups, which extends to an isomorphism $\mathrm{pr}:\Lambda \underline{\G}^{\tau F_0} \to \Lambda \G^F$ of $\Lambda$-algebras. 
Hence, under the isomorphism $\mathrm{pr}:\underline{\G}^{\tau F_0} \to \G^F$ the notions of blocks, Brauer subpairs and defect groups translate. From now on we will use the following notation: If $H$ is a subgroup of $\G^F$ we let $\underline{H}:=\mathrm{pr}^{-1}(H)$ and if $x \in \Lambda H$ then we let $\underline{x}:= \mathrm{pr}^{-1}(x) \in \Lambda \underline{H}$.
%

Let $b \in \mathrm{Z}(\Lambda \G^F e_s^{\G^F})$ and $c \in \mathrm{Z}(\Lambda \Levi^F e_s^{\Levi^F})$ be blocks which correspond to each other under the splendid Rickard equivalence given by $G \Gamma_c(\Y_\U^\G, \Lambda)e_s^{\Levi^F}$. By Proposition \ref{ShintaniDLV} and Corollary \ref{ShintaniLusztig} the projection map $\mathrm{pr}$ yields an isomorphism between $G \Gamma_c(\Y_\U^\G, \Lambda)e_s^{\Levi^F}$ and $G \Gamma_c(\Y_{\underline{\U}}^{\underline{\G},\tau F_0},\Lambda) e_{\underline{s}}^{\underline{\Levi}^{\tau F_0}}$. Hence, the blocks $\underline{b}\in \mathrm{Z}(\Lambda \underline{\G}^{\tau F_0} e_{\underline{s}}^{\underline{\G}^{\tau F_0}})$ and $\underline{c}\in \mathrm{Z}(\Lambda \underline{\Levi}^{\tau F_0} e_{\underline{s}}^{\underline{\Levi}^{\tau F_0}})$ correspond to each other under the splendid Rickard equivalence induced by $G \Gamma_c(\Y_{\underline{\U}}^{\underline{\G},\tau F_0},\Lambda) e_{\underline{s}}^{\underline{\Levi}^{\tau F_0}}$. We fix a maximal $c$-Brauer pair $(D,c_D)$ and let $(D,b_D)$ be the $b$-Brauer pair corresponding to it under the splendid Rickard equivalence induced by $G \Gamma_c(\Y_\U^\G, \Lambda) c$ in the sense of Proposition \ref{splendidloc}. Consequently, the $\underline{c}$-subpair $(\underline{D},\underline{c}_D)$ corresponds to the $\underline{b}$-subpair $(\underline{D},\underline{b}_D)$ under the Rickard equivalence induced by $G \Gamma_c(\Y_{\underline{\U}}^{\underline{\G},\tau F_0},\Lambda) e_{\underline{s}}^{\underline{\Levi}^{\tau F_0}}$.

If $Q$ is a subgroup of $D$ we let $(Q,c_Q) \leq (D,c_D)$ and $(Q,b_Q) \leq (D,b_D)$ be the corresponding Brauer subpairs. We denote $B_Q= \mathrm{Tr}^{\mathrm{N}_{\G^F}(Q)}_{\mathrm{N}_{\G^F}(Q,b_Q)}(b_Q)$ and $C_Q=\mathrm{Tr}^{\mathrm{N}_{\Levi^F}(Q)}_{\mathrm{N}_{\Levi^F}(Q,c_Q)}(c_Q)$.




\begin{proposition}\label{fieldlocal}
The bimodule $H_c^{\mathrm{dim}}(\Y^{\mathrm{N}_\G(Q)}_{\C_{\U}(Q)},\Lambda )C_Q$ can be equipped with a $\Lambda[\mathrm{N}_{\G^F}(Q) \times \mathrm{N}_{\Levi^F}(Q)^{\operatorname{opp}} \Delta \mathrm{N}_{\Ltilde^F \langle F_0 \rangle }(Q,C_Q)]$-module structure.
\end{proposition}

\begin{proof}
By Theorem \ref{independencelocal} (set $\hat{\G}:= \underline{\G}$), we have
	$$ H_c^{\mathrm{dim}}(\Y_{\C_{\underline{\U'}}(\underline{Q})}^{\mathrm{N}_{\underline{\G}}(\underline{Q}),\tau F_0}) \underline{C}_Q \cong   H_c^{\mathrm{dim}}(\Y_{\C_{\underline{\U}}(\underline{Q})}^{\mathrm{N}_{\underline{\G}}(\underline{Q}),\tau F_0}) \underline{C}_Q$$
as $\Lambda [(\mathrm{N}_{\underline{\G}^{\tau F_0}}(\underline{Q}) \times \mathrm{N}_{\underline{\Levi}^{\tau F_0}}(\underline{Q})^{\mathrm{opp}}) \Delta (\mathrm{N}_{\tilde{\underline{{\Levi}} }^{\tau F_0}}(\underline{Q},\underline{C}_Q))]$-modules. Moreover, Corollary \ref{ShintaniDLVCoro} shows that $H_c^{\mathrm{dim}}(\Y_{\C_{\underline{\U}}(\underline{Q})}^{\mathrm{N}_{\underline{\G}}(\underline{Q}),\tau F_0}) \underline{C}_Q$ is isomorphic to $H_c^{\mathrm{dim}}(\Y^{\mathrm{N}_\G(Q)}_{\C_{\U}(Q)},\Lambda ) C_Q$ as $\Lambda [(\mathrm{N}_{\G^F}(Q) \times \mathrm{N}_{\Levi^F}(Q)^{\mathrm{opp}}) \Delta( \mathrm{N}_{\tilde{\Levi}^F}(Q,c_Q))]$-modules.

Since $\tau( \underline{\U}')= \underline{\U}'$ we obtain a Levi decomposition $\underline{\tilde{\Para}} \langle \tau \rangle= \underline{\tilde{\Levi}} \langle \tau \rangle \ltimes \underline{\U}'$ in the reductive group $\underline{\tilde{\G}} \rtimes \langle \tau \rangle$. Hence we obtain a Levi decomposition $\mathrm{N}_{\underline{\tilde{\Para}} \langle \tau \rangle}( \underline{Q})= \mathrm{N}_{\underline{\tilde{\Levi}} \langle \tau \rangle}(\underline{Q})  \ltimes \mathrm{C}_{\underline{\U}'}(\underline{Q})$ in $\mathrm{N}_{\underline{\tilde{\G}} \langle \tau \rangle}( \underline{Q})$, see Example \ref{normcent}.
From this we conclude (see Lemma \ref{extend}) that the bimodule $H_c^{\mathrm{dim}}(\Y_{\C_{\underline{\U'}}(\underline{Q})}^{\mathrm{N}_{\underline{\G}}(\underline{Q}),\tau F_0}) \underline{C}_Q$ has a natural $\Delta (\mathrm{N}_{\underline{\tilde{\Levi}}^{\tau F_0} \langle \tau \rangle}(\underline{Q},\underline{C}_Q))$-action. It follows that the Morita bimodule $H_c^{\mathrm{dim}}(\Y^{\mathrm{N}_\G(Q)}_{\C_{\U}(Q)},\Lambda ) C_Q$ can be equipped with a $\Delta (\mathrm{N}_{\Ltilde^F \langle F_0 \rangle }(Q,C_Q))$-action.
\end{proof}

From now on we will assume that $Q$ is a characteristic subgroup of the defect group $D$.
Let us denote $B_Q'= \mathrm{Tr}_{ \mathrm{N}_{\tilde{\G}^F \mathcal{A}}(Q,B_Q)}^{\mathrm{N}_{\tilde{\G}^F \mathcal{A}}(Q)}( B_Q)$ and $C_Q'= \mathrm{Tr}_{ \mathrm{N}_{\tilde{\Levi}^F \mathcal{A}}(Q,C_Q)}^{\mathrm{N}_{\tilde{\Levi}^F \mathcal{A}}(Q)} (C_Q)$. Recall that $A \in \{K,\mathcal{O},k \}$.

\begin{theorem}\label{loc}
Suppose that the assumptions of Theorem \ref{Morita lift} are satisfied. Let $Q$ be a characteristic subgroup of $D$.
Then $H^{\mathrm{dim}}_c(\Y^{\mathrm{N}_\G(Q)}_{\C_{\U}(Q)},A )C_Q$ extends to an $A[(\mathrm{N}_{\G^F}(Q) \times \mathrm{N}_{\Levi^F}(Q)^{\mathrm{opp}}) \Delta( \mathrm{N}_{\tilde{\Levi}^F \mathcal{A}}(Q,C_Q))  ]$-module $M_Q$. In particular, the bimodule
$$\mathrm{Ind}_{(\mathrm{N}_{\G^F}(Q) \times \mathrm{N}_{\Levi^F}(Q)^{\mathrm{opp}}) \Delta (\mathrm{N}_{\tilde{\Levi}^F \mathcal{A}}(Q,C_Q))}^{\mathrm{N}_{\tilde{\G}^F \mathcal{A}}(Q) \times \mathrm{N}_{\tilde{\Levi}^F \mathcal{A}}(Q)^{\mathrm{opp}}}(M_Q)$$
induces a Morita equivalence between $A \mathrm{N}_{\tilde{\G}^F \mathcal{A}}(Q) B_Q'$ and $A \mathrm{N}_{\tilde{\Levi}^F \mathcal{A}}(Q) C_Q'$.
%

\end{theorem}

\begin{proof}


The quotient group $\mathrm{N}_{\tilde{\Levi}^F \mathcal{A}}(Q,C_Q)/ \mathrm{N}_{\tilde{\Levi}^F \langle F_0 \rangle }(Q,C_Q)$ is cyclic and of order divisible by the order of $\sigma \in \mathrm{Aut}(\tilde{\G}^F)$. Hence, there exist $x \in \tilde{\Levi}^F$ and a bijective morphism $\phi_0: \tilde{\G} \to \tilde{\G}$ such that $x \phi_0|_{\tilde{\G}^F}$ generates the quotient group $\mathrm{N}_{\tilde{\Levi}^F \mathcal{A}}(Q,C_Q)/ \mathrm{N}_{\tilde{\Levi}^F \langle F_0 \rangle }(Q,C_Q)$. Let $\underline{x}:=(x,F^{r-1}_0(x),\dots,F_0(x)) \in \underline{\G}^{\tau F_0}$ such that $\mathrm{pr}(\underline{x})=x$. Denote
$$\underline{\phi_0}: \underline{\tilde{\G}} \langle \tau \rangle \to \underline{\tilde{\G}} \langle \tau \rangle, \, (g_1,\dots,g_r) \tau \mapsto (\phi_0(g_1),\dots,\phi_0(g_r)) \tau$$
and consider the bijective morphism
$$\underline{\phi}:=\underline{x} \, \underline{\phi_0}:  \underline{\tilde{\G}} \langle \tau \rangle \to \underline{\tilde{\G}} \langle \tau \rangle, z \mapsto {}^{\underline{x}} \underline{\phi_0}(z),$$
of the reductive group $\underline{\tilde{\G}} \langle \tau \rangle$. Note that $\underline{\phi}$ stabilizes $\underline{\tilde{\G}}$ and commutes with the Frobenius endomorphism $\tau F_0$ of $\underline{\tilde{\G}} \rtimes \langle \tau \rangle$. Moreover, ${}^{\underline{x}}(\underline{\tilde{\Levi}} \langle \tau \rangle)= \underline{\tilde{\Levi}} \langle \tau \rangle$ and $\phi_0(\tilde{\Levi})=\tilde{\Levi}$. Therefore, the bijective morphism $\underline{\phi}$ also stabilizes the Levi subgroup $\tilde{\underline{\Levi}} \langle \tau \rangle$ of $\underline{\tilde{\G}} \rtimes \langle \tau \rangle$. Since $\underline{\phi}|_{\underline{\Gtilde}^{\tau F_0}} \in \mathrm{Aut}(\underline{\Gtilde}^{\tau F_0})$ corresponds to the automorphism $x \phi_0 \in \mathrm{Aut}(\tilde{\G}^F)$ under the isomorphism $\mathrm{pr}: \underline{\Gtilde}^{\tau F_0} \to \tilde{\G}^F$ we deduce that $\underline{\phi}(\underline{Q},\underline{C}_Q)=(\underline{Q},\underline{C}_Q)$. Hence, Lemma \ref{sigma} applies and we obtain an isomorphism
$${}^{\underline{\phi}} (H_c^{\mathrm{dim}}(\Y_{\C_{\underline{\U'}}(\underline{Q})}^{\mathrm{N}_{\underline{\G}}(\underline{Q}),\tau F_0},\Lambda ) \underline{C}_Q )^{\underline{\phi}} \cong  H_c^{\mathrm{dim}}(\Y_{\C_{\underline{\phi}(\underline{\U'})}(\underline{Q})}^{\mathrm{N}_{\underline{\G}}(\underline{Q}),\tau F_0},\Lambda ) \underline{ C}_Q$$
of $\Lambda [(\mathrm{N}_{\underline{\G}^{\tau F_0}}(\underline{Q}) \times \mathrm{N}_{\underline{\Levi}^{\tau F_0}}(\underline{Q})^{\mathrm{opp}}) \Delta (\mathrm{N}_{\underline{\Ltilde}^{\tau F_0} \langle \tau \rangle}(\underline{Q},\underline{C}_Q))]$-modules.
We have two Levi decompositions 
$$\underline{\tilde{\Para}} \langle \tau \rangle= \underline{\tilde{\Levi}} \langle \tau \rangle \ltimes \underline{\U} \text{ and } \underline{\phi}(\underline{\tilde{\Para}} \langle \tau \rangle)= \underline{\tilde{\Levi}} \langle \tau \rangle \ltimes \underline{\phi}(\underline{\U})$$
with the same Levi subgroup  $\underline{\tilde{\Levi}} \langle \tau \rangle $ of $\tilde{\G} \langle \tau \rangle$. Therefore, Theorem \ref{independencelocal} yields
$$H_c^{\mathrm{dim}}(\Y_{\C_{\underline{\phi}(\underline{\U'})}(\underline{Q})}^{\mathrm{N}_{\underline{\G}}(\underline{Q}),\tau F_0},\Lambda ) \underline{ C}_Q
\cong  H_c^{\mathrm{dim}}(\Y_{\C_{\underline{\U'}}(\underline{Q})}^{\mathrm{N}_{\underline{\G}}(\underline{Q}),\tau F_0},\Lambda ) \underline{ C}_Q.$$
It follows from this that $H_c^{\mathrm{dim}}(\Y_{\C_{\underline{\U'}}(\underline{Q})}^{\mathrm{N}_{\underline{\G}}(\underline{Q}),\tau F_0},\Lambda ) \underline{C}_Q$ is $( \underline{\phi},\underline{\phi}^{-1})$-invariant. Hence, the bimodule $H^{\mathrm{dim}}_c(\Y^{\mathrm{N}_\G(Q)}_{\C_{\U}(Q)},A )C_Q$ is by transport of structure $(x \phi_0,x \phi_0^{-1})$-invariant as $A[(\mathrm{N}_{\G^F}(Q) \times \mathrm{N}_{\Levi^F}(Q)^{\operatorname{opp}}) \Delta (\mathrm{N}_{\Ltilde^F \langle F_0 \rangle  }(Q,C_Q))]$-module. Thus, \cite[Lemma 10.2.13]{Rouquier3} shows that there exists an $A[(\mathrm{N}_{\G^F}(Q) \times \mathrm{N}_{\Levi^F}(Q)^{\mathrm{opp}}) \Delta( \mathrm{N}_{\tilde{\Levi}^F \mathcal{A}}(Q,C_Q))  ]$-module $M_Q$ extending $H^{\mathrm{dim}}_c(\Y^{\mathrm{N}_\G(Q)}_{\C_{\U}(Q)},A )C_Q$. By Theorem \ref{Morita2} the bimodule $H^{\mathrm{dim}}_c(\Y^{\mathrm{N}_\G(Q)}_{\C_{\U}(Q)},A )C_Q$ induces a Morita equivalence between the blocks $A \mathrm{N}_{\G^F}(Q) B_Q$ and $A \mathrm{N}_{\Levi^F}(Q) C_Q$. Moreover, Lemma \ref{stabstab} implies
$$\mathrm{N}_{\tilde{\Levi}^F \mathcal{A} }(Q,C_Q) \mathrm{N}_{\G^F}(Q) = \mathrm{N}_{\tilde{\G}^F \mathcal{A} }(Q,B_Q).$$
Hence, Lemma \ref{better version} shows that the bimodule $$\mathrm{Ind}_{(\mathrm{N}_{\G^F}(Q) \times \mathrm{N}_{\Levi^F}(Q)^{\mathrm{opp}}) \Delta (\mathrm{N}_{\tilde{\Levi}^F \mathcal{A}}(Q,C_Q))}^{\mathrm{N}_{\tilde{\G}^F \mathcal{A}}(Q) \times \mathrm{N}_{\tilde{\Levi}^F \mathcal{A}}(Q)^{\mathrm{opp}}}(M_Q)$$
induces a Morita equivalence between $A \mathrm{N}_{\tilde{\G}^F \mathcal{A}}(Q) B_Q'$ and $A \mathrm{N}_{\tilde{\Levi}^F \mathcal{A}}(Q) C_Q'$.
%
%
%
\end{proof}

\begin{remark}
If one could prove a version of Theorem \ref{Morita lift} with Morita equivalence replaced by splendid Rickard equivalence then Theorem \ref{loc} would be obtained as a consequence of that theorem, see Theorem \ref{normalizerderived}. However this seems to be difficult since we would have to show that the Rickard--Rouquier complex $G \Gamma_c( \Y_\U,\Lambda) e_s^{\Levi^F}$ is independent of the choice of the unipotent radical $\U$ used in its definition. In the case where the Sylow $\ell$-subgroups of $\G^F$ are cyclic we obtained such an independence result in Example \ref{iGR}.
\end{remark}


\printindex

\bibliographystyle{alpha}

\end{document}